\documentclass{article}
\usepackage{amssymb}
\usepackage{amsmath}
\usepackage{arxiv}
\usepackage[utf8]{inputenc}
\usepackage[T1]{fontenc}
\usepackage{url}
\usepackage{booktabs}
\usepackage{amsfonts}
\usepackage{nicefrac}
\usepackage{microtype}
\usepackage{color}
\usepackage{lipsum}
\usepackage{amsthm}
\usepackage{indentfirst}
\usepackage{graphicx}
\usepackage{epstopdf}
\usepackage{fourier}
\usepackage{bm}
\usepackage{mathtools}
\usepackage{latexsym,enumerate}
\usepackage{multicol}
\usepackage[skip=2pt]{caption}
\usepackage[font=small,skip=0pt]{subcaption}
\usepackage{array}
\usepackage{ascii}

\usepackage{mathtools}
\usepackage{arydshln}

\usepackage[pdfstartview=FitH, CJKbookmarks=true,
bookmarksnumbered=true, bookmarksopen=true,
colorlinks,
linkcolor=green,
anchorcolor=blue, citecolor=blue
]{hyperref}

\def\bm{\boldsymbol}

\newcommand{\comment}[1]{}

\newcommand{\BEA}{\begin{eqnarray}}
\newcommand{\EEA}{\end{eqnarray}}

\newtheorem{thm}{Theorem}[section]
\newtheorem{prop}[thm]{Proposition}

\newtheorem{lem}[thm]{Lemma}

\newtheorem{defn}[thm]{Definition}

\newtheorem{rem}[thm]{Remark}

\newtheorem{assumption}[thm]{Assumption}
\newcommand{\PreserveBackslash}[1]{\let\temp=\\#1\let\\=\temp}
\newcolumntype{C}[1]{>{\PreserveBackslash\centering}p{#1}}
\newcolumntype{R}[1]{>{\PreserveBackslash\raggedleft}p{#1}}
\newcolumntype{L}[1]{>{\PreserveBackslash\raggedright}p{#1}}

\newcommand{\stkout}[1]{\ifmmode\text{\sout{\ensuremath{#1}}}\else\sout{#1}\fi}

\begin{document}

\title{Spectral methods for solving elliptic PDEs on unknown manifolds}
\author{ Qile Yan \\
Department of Mathematics, The Pennsylvania State University, University
Park, PA 16802, USA\\
\texttt{qzy42@psu.edu}\\
\And Shixiao Willing Jiang \\
Institute of Mathematical Sciences, ShanghaiTech University, Shanghai
201210, China\\
\texttt{jiangshx@shanghaitech.edu.cn} \\
\And John Harlim \\
Department of Mathematics, Department of Meteorology and Atmospheric
Science, \\
Institute for Computational and Data Sciences \\
The Pennsylvania State University, University Park, PA 16802, USA\\
\texttt{jharlim@psu.edu} }
\date{\today}
\maketitle

\begin{abstract}
In this paper, we propose a mesh-free numerical method for solving elliptic PDEs on unknown manifolds, identified with randomly sampled point cloud data. The PDE solver is formulated as a spectral method where the test function space is the span of the leading eigenfunctions of the Laplacian operator, which are approximated from the point cloud data. While the framework is flexible for any test functional space, we will consider the eigensolutions of a weighted Laplacian obtained from a symmetric Radial Basis Function (RBF) method induced by a weak approximation of a weighted Laplacian on an appropriate Hilbert space. {\color{black}In this paper}, we consider a test function space that encodes the geometry of the data yet does not require us to identify and use the sampling density of the point cloud. To attain a more accurate approximation of the expansion coefficients, we adopt a second-order tangent space estimation method to improve the RBF interpolation accuracy in estimating the tangential derivatives. This spectral framework allows us to efficiently solve the PDE many times subjected to different parameters, which reduces the computational cost in the related inverse problem applications. In a well-posed elliptic PDE setting with randomly sampled point cloud data, we provide a theoretical analysis to demonstrate the {\color{black}convergence} of the proposed solver as the sample size increases. We also report some numerical studies that show the convergence of the spectral solver on simple manifolds and unknown, rough surfaces. Our numerical results suggest that the proposed method is more accurate than a graph Laplacian-based solver on smooth manifolds. On rough manifolds, these two approaches are comparable. Due to the flexibility of the framework, we empirically found improved accuracies in both smoothed and unsmoothed Stanford bunny domains by blending the graph Laplacian eigensolutions and RBF interpolator.

\end{abstract}


\lhead{}

\newpage

\section{Introduction}

Many applications in the natural sciences and practical engineering involve solving Partial Differential Equations (PDEs) on manifolds. In image processing, applications include segmentation of images \cite{tian2009segmentation}, image inpainting \cite{shi2017weighted} and restoration of damaged patterns \cite{bertalmio2001navier,macdonald2010implicit}. In computer graphics, applications include flow field visualization \cite{bertalmio2001variational}, surface reconstruction \cite{zhao2001fast} and brain imaging \cite{memoli2004implicit}. In physics and biology, applications include granular flow \cite{rauter2018finite}, phase formation \cite{dogel2005two} and phase ordering \cite{schoenborn1999kinetics} on surfaces, Liquid crystals on deformable surfaces \cite{nitschke2020liquid} and phase separation of biomembranes \cite{elliott2010modeling}.

With such diverse applications, many numerical methods have been developed to approximate the solution of PDEs on manifolds. Given a triangular mesh as an approximation to the domain, the surface finite element method (FEM) \cite{dziuk2007surface,dziuk2013finite} is a robust and efficient method to discretize the PDE directly on the triangular mesh. But one issue with this class of approach is that triangular mesh can be difficult to obtain when the available point cloud is randomly sampled and when the manifolds are high-dimensional and/or embedded in a high-dimensional ambient space. {\color{black} In the absence of admissible and regular mesh, it is also unclear how to specify the corresponding test function space when the domain is unknown in the sense that it is  only identified by randomly sampled point cloud data which is the set-up considered in this paper.}

To avoid this issue, mesh-free approaches were developed. For example, {\color{black} several collocation methods, including the global Radial Basis Function (RBF) methods \cite{piret2012orthogonal,fuselier2013high} and the RBF-generated finite difference (FD) methods \cite{shankar2015radial,lehto2017radial} have been developed. In most of these approaches, they assumed that the manifolds are identified by point cloud data and the corresponding normal/tangent vectors at these points are given.} When the domain is identified only by point cloud data, these methods employ a manifold parameterization scheme, such as using local SVD approximation of the tangent space, level set methods \cite{bertalmio2001variational,greer2006improvement,xu2003eulerian},
closest point methods \cite{ruuth2008simple,petras2018rbf} and orthogonal gradient methods \cite{piret2012orthogonal}, and subsequently approximate the surface differentiation along the approximate tangent bundle. {\color{black}Another class of mesh-free approach is to identify a regression solution to the PDE by employing the Generalized Moving Least-Squares (GMLS) to approximate tangential derivatives on the point cloud data \cite{liang2013solving,gross2020meshfree}. In this class of approaches, they parameterize the metric tensor of the manifold locally by multivariate quadratic functions whereas the method in \cite{suchde2019meshfree} requires only the normal/tangent spaces at the point cloud data.} Alternatively, graph-based approaches, including Graph Laplacian, Diffusion Maps, Weighted Nonlocal Laplacian \cite{li2016convergent,li2017point,gh2019,jiang2020ghost,yan2022ghost}, do not require a parameterization of a point cloud and can handle high-dimensional manifolds and randomly sampled data, although limited to a certain class of differential operators.

{\color{black}Beyond the pointwise methods that we mentioned above, the variational approach has also been considered. For example, the approach in \cite{gross2018spectral} was proposed to solve PDEs on two-dimensional radial surfaces identified by Lebedev nodes. For this class of manifolds, they seek a PDE solution in a Hilbert space induced by the spherical harmonics and employ a Lebedev quadrature rule to approximate the inner product (or surface integral). The method proposed in the current paper is also a variational approach, where we extend the approach in \cite{gross2018spectral} to solve PDEs on any smooth manifolds identified by randomly sampled point cloud data. With this stringent characterization of the manifolds, we will identify an appropriate test function space that encodes the geometry and the distribution of the randomly sampled point cloud data that is also unknown.}

The work in this paper is motivated by the need for an efficient PDE solver to facilitate inverse problem algorithms for estimating parameters of PDE (see e.g., \cite{harlim2020kernel,harlim2022graph}), which plays a central role in modeling physical systems. Whether the maximum likelihood or Bayesian approach is used in the inversion method, the numerical algorithm usually requires an iterative procedure to compare the predicted observations corresponding to the proposed parameter value at the current iteration to the true observations. When the solution operator that maps the parameter to be determined to the PDE solution is not explicitly available, the computational cost of each iteration in this fitting procedure is dominated by the complexity of the corresponding PDE solver. To illustrate the issue more precisely, let us consider the following elliptic PDE,
\BEA
-\text{div}_g(\kappa \textup{grad}_g u)+c u=f\quad\text{on}\ M.
\label{eq1}
\EEA
Here, $M$ is a $d$-dimensional embedded compact submanifold of Euclidean space $\mathbb{R}^n$ without boundary, the differential operators are defined with respect to the Riemannian metric $g$, the parameter $c$ is a positive function, $\kappa$ is a positive diffusion coefficient, and $f$ is a known function defined on $M$. In our setup, we consider the domain to be unknown. Specifically, we have no access to either the Riemannian metric or the embedding function that parameterizes the manifold $M$. All we have is a set of independent random samples, $X=\{\mathbf{x}_1,\ldots,\mathbf{x}_N\}\subset M$ with an unknown sampling measure that is supported in $M$. If the corresponding inverse problem is to approximate $\kappa$, then we need an efficient solver to evaluate the following solution operator $\kappa\mapsto u(\kappa,\cdot)$,
where $u(\kappa,\cdot):M \to \mathbb{R}$ satisfies the PDE in \eqref{eq1}.

While there are many available PDE solvers as we listed above, a pointwise discretization of the PDE in \eqref{eq1} with any of these methods {\color{black}mentioned in the third paragraph above} on the point cloud data yields a linear system of $N\times N$ equation, $\mathbf{Lu}= \mathbf{f}$, with $\mathbf{u}=(u(\mathbf{x}_1),\ldots, u(\mathbf{x}_N))$ and $\mathbf{f}=(f(\mathbf{x}_1),\ldots, f(\mathbf{x}_N))$. The key issue with such a pointwise approximation is that while the matrix $\mathbf{L}$ depends on the function values $\boldsymbol{\kappa}=(\kappa(\mathbf{x}_1),\ldots,\kappa(\mathbf{x}_N))$, {\color{black}yet one cannot write an explicit expression for the map $\boldsymbol{\kappa} \mapsto \mathbf{L}^{-1}\mathbf{f}$, that is, how $ \mathbf{L}^{-1}\mathbf{f}$ depends on $\boldsymbol{\kappa}$ is unknown.} For example, one can see how the matrix $\mathbf{L}$ depends on $\kappa(\mathbf{x}_i)$ when the RBF interpolation or Diffusion maps approach is used to approximate the left-hand side of \eqref{eq1} in Section~\ref{sec4.1} on the point cloud, {\color{black} however,} the dependence of $\mathbf{L}^{-1}\mathbf{f}$ on $\boldsymbol{\kappa}$ is unclear. This same observation is also valid if one takes a closer inspection at the other solvers we mentioned above,
see e.g. Eq.(3.4) in \cite{liang2013solving} or Eq.~(42) in \cite{suchde2019meshfree} or Procedure~1 in \cite{li2016convergent} or Eq.~(38) in \cite{jiang2020ghost}. {\color{black}Thus, it is desirable to have an efficient way to invert the matrix $\mathbf{L}$.}

There are numerous ways to overcome this issue. An obvious method is to construct a sparse matrix $\mathbf{L}$, which was the key idea of the RBF-FD method  \cite{shankar2015radial,lehto2017radial}. With such an approximation, however, they reported less competitive accuracies compared to the global RBF approach, {\color{black}where the estimator also seemed to be sensitive to the distribution of the point cloud data and the choice of kernels as we will numerically verify in a test problem.} To overcome this issue, we consider a spectral framework that allows one to avoid the computational cost of an $N\times N$ matrix inversion in each evaluation of $\boldsymbol{\kappa} \mapsto \mathbf{L}^{-1}\mathbf{f}$. Specifically, we consider a Galerkin approximation where the test function spaces consist of a set of data-driven basis functions that encodes the geometry and sampling distribution of the point cloud data. Specifically, we consider a set of approximate eigenfunctions of a weighted Laplacian operator on the manifolds that can be constructed using any of the above operator approximation methods (including FEM, RBF, GMLS, or Diffusion Maps). In this paper, we will illustrate our idea with a symmetric RBF matrix induced by a weak approximation of the weighted Laplacian on {\color{black} an }appropriate Hilbert space \cite{harlim2022rbf}. This weak approximation guarantees non-negative real eigenvalues and real eigenfunction estimates, as opposed to the non-symmetric RBF matrix induced by the traditional pointwise approximation \cite{fuselier2013high,fornberg2015solving,shankar2015radial,lehto2017radial}. We should emphasize that since the basis is attained independent of the parameter $\kappa$, one needs to only solve the eigenvalue problem one time (or offline). Subsequently, a Galerkin projection of the PDE in \eqref{eq1} to the leading $K$-eigenspace reduces the problem of solving an $N\times N$ linear system to a $K\times K$ linear system, where the latter is computationally more tractable when $K \ll N$. Since the method is flexible, we will also adopt the framework using a data-driven basis constructed by a graph Laplacian approximation induced by the variable bandwidth diffusion maps \cite{bh:16vb} in some examples. To approximate the Galerkin expansion coefficients that depend on derivatives of the eigenfunctions, we employ the global and local RBF pointwise interpolation schemes (among many possible choices). Since RBF interpolation on the manifold setting is an extrinsic approximation, we implement a recently developed second-order local SVD scheme that allows one to accurately identify the tangent bundle on each point cloud data (see Section~3 of \cite{harlim2022rbf}). We should point out that while the approach is not restrictive to other PDEs and is just as efficient for any linear PDEs, we will only focus our discussion on the elliptic PDE in \eqref{eq1}. For this specific example, we will report the computational complexity and convergence study with error bounds in terms of the sample size, $N$, and the number of Galerkin modes truncation, $K$, under appropriate assumptions.

The paper is organized as follows. In Section~\ref{sec2}, we provide a short review of the classical RBF method for operator pointwise approximation. We also review a symmetric RBF approximation of Laplacians for solving the eigenvalue problem weakly and the second-order SVD scheme for approximating the tangent space pointwise for unknown manifolds. In Section~\ref{sec3}, we present the proposed spectral solver and the corresponding convergence analysis. In Section~\ref{section_numeric}, we demonstrate the numerical performance on simple manifolds and an unknown rough (and smoothed) surface. We conclude the paper with a summary and a list of open problems in Section~\ref{sec5}. We include some proofs in Appendix~\ref{AppA} {\color{black}and a brief overview of the RBF-FD approximation used in our numerical experiment in Appendix~\ref{AppB}.}

\section{Preliminaries}\label{sec2}
In this section, we first briefly review the radial basis function (RBF) interpolation. Then we describe the RBF projection method which is used to construct discrete operators for pointwise approximation of relevant differential operators that act on smooth functions on smooth manifolds. We also discuss a symmetric approximation to the Laplacian operator induced by the weak formulation on {\color{black}an }appropriate Hilbert space. Next, we review a method to estimate {\color{black} the }tangent space from point cloud data that lie on the manifold, which will give us an accurate estimation of the projection matrix which will be useful in the construction of the discrete operators.

\subsection{Review of RBF interpolation}\label{review-RBF}
\par We first review the radial basis function interpolation over a set of point cloud data. Let $M$ be a $d$-dimensional smooth manifold of $\mathbb{R}^n$. Given a set of (distinct) nodes $X=\{\textbf{x}_i\}_{i=1}^N\subset M$ and function values $\mathbf{f}:=(f(\textbf{x}_1),\ldots, f(\textbf{x}_N))^\top$\ at $X = \{\textbf{x}_{j}\}_{j=1}^{N}$, where $f:M\rightarrow\mathbb{R}$ is an arbitrary smooth function, a radial basis function (RBF) interpolant takes the form
\BEA
I_{\phi_s}\mathbf{f}(\textbf{x})=\sum_{j=1}^Nc_j\phi_s\big(\Vert \textbf{x}-\textbf{x}_j\Vert\big), \ \  \textbf{x}\in M
\label{RBF-form}
\EEA
where $c_j$ are determined by requiring $I_{\phi_s} \mathbf{f}|_X=\mathbf{f}$. Here, we have defined the interpolating operator $I_{\phi_s}:\mathbb{R}^{N} \to C^\alpha(\mathbb{R}^n)$, where $\alpha$ denotes the smoothness of the radial kernel $\phi_s$. In \eqref{RBF-form}, the notation $\Vert\cdot\Vert$ corresponds to the standard Euclidean norm in the ambient space $\mathbb{R}^n$. The interpolation constraints can be expressed as the following linear system
\BEA
\underbrace{\left[
\begin{array}{ccc}
\phi _{s}(\Vert \mathbf{x}_{1}-\mathbf{x}_{1}\Vert ) & \cdots  &
\phi _{s}(\Vert \mathbf{x}_{1}-\mathbf{x}_{N}\Vert ) \\
\vdots  & \ddots  & \vdots  \\
\phi _{s}(\Vert \mathbf{x}_{N}-\mathbf{x}_{1}\Vert ) & \cdots  &
\phi _{s}(\Vert \mathbf{x}_{N}-\mathbf{x}_{N}\Vert )%
\end{array}%
\right]}_{\boldsymbol{\Phi}}\underbrace{\left[\begin{array}{c}c_{1} \\ \vdots \\ c_{N}\end{array}\right]}_{\mathbf{c}}=\underbrace{\left[\begin{array}{c}f\left(\mathbf{x}_{1}\right) \\ \vdots \\ f\left(\mathbf{x}_{N}\right)\end{array}\right]}_{\mathbf{f}}
\label{RBF-inter}
\EEA
where $[\boldsymbol{\Phi}]_{i,j}=\phi_s(\Vert\textbf{ x}_i-\textbf{x}_j\Vert)$.

In literature, many types of radial functions have been proposed for the RBF interpolation. For example, the Gaussian function $\phi_s(r)=e^{-(sr)^2}$ \cite{fuselier2013high}, Multiquadric function $\phi_s(r)=\sqrt{1+(sr)^2}$, Wendland function $\phi_s(r)=(1-sr)^m_+p(sr)$ for some polynomials $p$, and the Inverse quadratic function $\phi_s(r)=1/(1+(sr)^2)$. In our numerical examples, we will implement with the Inverse quadratic function. While this kernel yields a positive definite matrix $\boldsymbol{\Phi}$, the matrix tends to have a high condition number, especially when the point cloud data is randomly sampled. To overcome this issue, we solve the linear problem in \eqref{RBF-inter} using the standard pseudo-inverse method with an appropriately specified tolerance.

\subsection{RBF pointwise approximation of surface differential operators}\label{sec2.2}
Now, we review the RBF projection method proposed in \cite{fuselier2013high} for a discrete approximation of surface differential operators on manifolds. The projection method represents the surface differential operators as tangential gradients, which are formulated as the projection of the appropriate derivatives in the ambient space. Let $M$ be a $d-$dimensional manifold embedded in the ambient Euclidean space, $\mathbb{R}^n$. For any point $\textbf{x}=(x^1,...,x^n)\in M$, we denote the tangent space of $M$ at $\textbf{x}$ as $T_\textbf{x}M$ and  a set of orthonormal vectors that span this tangent space as $\{\textbf{t}_i\}_{i=1}^d$. Then the projection matrix $\textbf{P}$ which projects vectors in $\mathbb{R}^n$ to $T_\textbf{x}M$ could be written as $\textbf{P}=\sum_{i=1}^d\textbf{t}_i\textbf{t}_i^\top$ (see e.g. \cite{harlim2022rbf} for a detailed derivation). Subsequently, the surface gradient on a smooth function $f:M\to \mathbb{R}$ evaluated at $\textbf{x}\in M$ in the Cartesian coordinates is given as,
$$
\textup{grad}_g f(\mathbf{x}):=\textbf{P}\overline{\textup{grad}}_{\mathbb{R}^n}f(\mathbf{x})=\left(\sum_{i=1}^d\textbf{t}_i\textbf{t}_i^\top\right)\overline{\textup{grad}}_{\mathbb{R}^n}f(\mathbf{x}),
$$
where $\overline{\textup{grad}}_{\mathbb{R}^n}=[\partial_{x^1}, \cdots, \partial_{x^n}]^\top$ is the usual gradient operator in Euclidean space and the subscript $g$ is to associate the differential operator to the Riemannian metric $g$ induced by $M$ from $\mathbb{R}^n$. Let $\mathbf{e}^\ell,\ell=1,...,n$ be the standard orthonormal vectors  in $x^\ell$ direction in $\mathbb{R}^n$, we can rewrite above expression in component form as
$$
\textup{grad}_g f (\mathbf{x}):=\left[\begin{array}{c}\left(\mathbf{e}^{1} \cdot \mathbf{P}\right) \overline{\textup{grad}}_{\mathbb{R}^n} f(\mathbf{x}) \\ \vdots\\ \left(\mathbf{e}^{n} \cdot \mathbf{P}\right) \overline{\textup{grad}}_{\mathbb{R}^n}f(\mathbf{x}) \end{array}\right]
=\left[\begin{array}{c}\mathbf{P}^{1} \cdot \overline{\textup{grad}}_{\mathbb{R}^n}f(\mathbf{x}) \\ \vdots \\ \mathbf{P}^{n} \cdot \overline{\textup{grad}}_{\mathbb{R}^n} f(\mathbf{x})\end{array}\right]:=\left[\begin{array}{c}\mathcal{G}_{1}f(\mathbf{x}) \\ \vdots\\ \mathcal{G}_{n} f(\mathbf{x})\end{array}\right]
$$
where $\textbf{P}^\ell$ is the $\ell-$th row of the projection matrix $\textbf{P}$. With this identity, one can write
the Laplace-Beltrami operator of $f$ at $\textbf{x}\in M$ as
$$
\Delta_{M}f (\mathbf{x}) :=-\textup{div}_{g} \textup{grad}_{g} f (\mathbf{x})=-\left(\mathbf{P} \overline{\textup{grad}}_{\mathbb{R}^n}\right) \cdot\left(\mathbf{P} \overline{\textup{grad}}_{\mathbb{R}^n}\right)f (\mathbf{x})=-\sum_{\ell=1}^n\mathcal{G}_{\ell} \mathcal{G}_{\ell}f (\mathbf{x}).
$$

Then we can approximate any differential operators (related to gradient and divergence) of functions $f$ by differentiating the RBF interpolant. That is, for $\ell=1,...,n$, and $\mathbf{x}_i \in X$,
$$
\mathcal{G}_\ell f(\textbf{x}_i)\approx (\mathcal{G}_\ell I_{\phi_s}\mathbf{f})(\textbf{x}_i)=\sum_{j=1}^Nc_j\mathcal{G}_\ell\phi_s(\Vert \textbf{x}_i-\textbf{x}_j\Vert),
$$
which can be written in matrix form as,
\BEA
\underbrace{\left[\begin{array}{c}\mathcal{G}_\ell f\left(\mathbf{x}_{1}\right) \\ \vdots \\ \mathcal{G}_\ell f\left(\mathbf{x}_{N}\right)\end{array}\right]}_{\mathcal{G}^\ell \mathbf{f}}\approx\mathbf{B}_\ell\underbrace{\left[\begin{array}{c}c_{1} \\ \vdots \\ c_{N}\end{array}\right]}_{\mathbf{c}}=\mathbf{B}_\ell \boldsymbol{\Phi}^{-1}\mathbf{f},\label{G_ell}
\EEA
where $[\mathbf{B}_\ell]_{i,j}=\mathcal{G}_\ell\phi_s(\Vert \textbf{x}_i-\textbf{x}_j\Vert)$ and we have used $\boldsymbol{\Phi}$ as defined in (\ref{RBF-inter}) in the last equality. Hence, the differential matrix for the operator $\mathcal{G}_\ell$ is given by
\BEA
\mathbf{G}_\ell:=\mathbf{B}_\ell \boldsymbol{\Phi}^{-1}
\label{diff-op}
\EEA
and the Laplace-Beltrami operator can be approximated pointwise at any $\mathbf{x}_i\in X$ as,
\BEA
\Delta_Mf (\mathbf{x}_i)=-\sum_{\ell=1}^n\mathcal{G}_\ell\mathcal{G}_\ell f (\mathbf{x}_i)\approx -\sum_{\ell=1}^n\mathcal{G}_\ell I_{\phi_s}( \mathbf{G}_\ell\mathbf{f})(\mathbf{x}_i)\approx - \big(\sum_{\ell=1}^n \mathbf{G}_\ell \mathbf{G}_\ell \mathbf{f}\big)_i:= \Delta_X^{\mathrm{RBF}} \mathbf{f} (\mathbf{x}_i). \label{NRBF_Laplacian}
\EEA

We should point out that the construction of the discrete approximation of $\Delta_M$ is a non-symmetric matrix, and hence, it does not guarantee that the eigenvalues are real-valued. In fact, it was reported in \cite{Fuselier2009Stability,piret2012orthogonal,fuselier2013high}. In particular, they numerically reported that when the number of training data is small, the eigenvalues of the non-symmetric RBF Laplacian matrix that approximates the negative-definite Laplace-Beltrami operator are not only complex-valued, but they can also be on the positive half-plane. These papers also empirically reported that this issue can be overcome with more data points. {\color{black} This phenomenon is also observed in \cite{harlim2022rbf}} where they also pointed out another important practical issue. In particular, they found that when the training data is not large enough and when $\mathbf{P}$ is not accurately estimated, the non-symmetric formulation tends to produce irrelevant eigenvalues that are not associated with any of the underlying spectra. Even with more training data, these irrelevant spectral estimates still persistently occur on the larger modes. In addition to these inconsistencies, we also numerically find that the estimate is not robust when the manifold is rough (which we will report in Section~\ref{section_bunny}). These issues motivate us to consider a symmetric approximation induced by the weak formulation, which we discuss in the next subsection.

\subsection{RBF symmetric approximation of Laplacians}

In this section, we discuss a symmetric discrete approximation of Laplacians induced by a weak formulation on appropriate Hilbert spaces. This formulation has several advantages, namely, the estimated eigenvalues are guaranteed to be non-negative real-valued, and the corresponding estimates for the eigenvectors
are real-valued and orthogonal. Besides, the spectral convergence {\color{black}(i.e., convergence of eigenvalues and eigenvectors, which is reported in Lemma~\ref{lem_eig_conv})} is guaranteed with a Monte-Carlo error rate provided that a smooth enough kernel is adopted.

As we mentioned above, this approximation is motivated by the following identity,
\BEA
- \int_M  h \textup{div}_g (q \textup{grad}_g f) d\textup{Vol} = \int_M \langle \textup{grad}_g h, \textup{grad}_g f \rangle_g q d\textup{Vol}, \quad \forall f,h \in C^\infty(M), \label{intbypart}
\EEA
which is valid for any closed manifold $M$. In \eqref{intbypart}, $\langle\cdot,\cdot\rangle_g$ denotes the Riemannian inner product of vector fields, $\text{Vol}$ denotes the volume form inherited by $M$ from the ambient space $\mathbb{R}^n$, and $q$ denotes the density function with respect to the volume form. If the data are uniformly and identically distributed, then $q\propto 1$ and the weighted Laplacian $-q^{-1}\textup{div}_g (q \textup{grad}_g\,)$ simplifies to the Laplace-Beltrami operator $\Delta_M$. It turns out that if we define $\mathbf{G}:\mathbb{R}^N \to \mathbb{R}^{nN}$,
\BEA
\mathbf{G} \mathbf{f} = \left( \mathbf{G}_1 \mathbf{f},
\mathbf{G}_2 \mathbf{f},
\ldots,
\mathbf{G}_n \mathbf{f}
\right)^\top, \label{sec2.1:eq2}
\EEA
as a discrete estimator to the gradient restricted on the training data set $X$ with sampling density $q$ and $\mathbf{G}_\ell$ is defined as in \eqref{diff-op}, then one can employ the following Monte-Carlo estimate:
\BEA
 \int_M \langle \textup{grad}_g f, \textup{grad}_g f \rangle_g q d\textup{Vol} \approx \frac{1}{N}  \mathbf{f}^\top\mathbf{G}^\top \mathbf{G} \mathbf{f}.\label{weightedLaplacian}
\EEA
We should also point out that if the sampling density $q>0$ is given or can be approximated, then one can define a diagonal matrix $\mathbf{Q}\in\mathbb{R}^{N\times N}$ with diagonal entries $\mathbf{Q}_{ii}= q(\mathbf{x}_i)$ for $\mathbf{x}_i\in X$, and approximate the Laplace-Beltrami operator as follows,
\BEA
\int_M  f \Delta_M f d\textup{Vol} =  \int_M \langle \textup{grad}_g f, \textup{grad}_g f \rangle_g  d\textup{Vol} \approx \frac{1}{N}   \mathbf{f}^\top\mathbf{G}^\top \mathbf{Q}^{-1} \mathbf{G} \mathbf{f},\label{LaplacianMCapprox}
\EEA
by re-weighting the non-uniform sampling density with $q^{-1}$. Since $q$ is not available to us in our problem, we choose to expand on eigenfunctions of the weighted Laplacian in \eqref{weightedLaplacian}, and thus, avoid error induced by the density estimation. In the next section, we will discuss solving the eigenvalue problem corresponding to the weak formulation in \eqref{weightedLaplacian} that effectively not only account for the geometry but also the sampling distribution of the data.

\subsection{A second-order local SVD method for tangent space estimation}\label{sec2.4}

Here, we give a short overview of two tangent space estimation methods, which rely on the following Taylor expansion of a local parameterization of the manifold $M$. Particularly, for any $\mathbf{x},\mathbf{y}\in M$, define a geodesic $\iota: \mathbb{R}^d \to M$ such that $\iota \left( \mathbf{s}\right)=\mathbf{y}$ and $\iota(\mathbf{0})= \mathbf{x}$. Consider the Taylor expansion of $\iota \left( \mathbf{s}\right) $ centered at $\mathbf{0},$ where $\mathbf{s}=(s_1,\ldots,s_d)\in \mathbb{R}^d$ is the geodesic normal coordinates of $\mathbf{y}$ corresponding to the usual arc-length parameterization such that $|\mathbf{s}|=s$,
\begin{equation}
\iota \left( \mathbf{s}\right) =\iota (\mathbf{0})+\sum_{i=1}^{d}s_{i}\frac{%
\partial \iota (\mathbf{0})}{\partial s_{i}}+\frac{1}{2}%
\sum_{i,j=1}^{d}s_{i}s_{j}\frac{\partial ^{2}\iota (\mathbf{0})}{\partial
s_{i}\partial s_{j}}+O(s^{3}).  \label{eqn:iots}
\end{equation}

The classical local SVD method \cite{donoho2003hessian,zhang2004principal,tyagi2013tangent} uses the difference vector $\mathbf{y}-\mathbf{x} = \iota(\mathbf{s})-\iota(\mathbf{0})$ to estimate $\mathbf{T} = (\bm{\tau}_1,\ldots,\bm{\tau}_d)$ (up to an orthogonal rotation) and subsequently use it to approximate $\mathbf{P}=\mathbf{T}\mathbf{T}^\top$. Numerically, the first-order local SVD method proceeds as follows:
\begin{enumerate}
\item For each $\mathbf{x}\in X$, let $\{\mathbf{y}_1,\ldots, \mathbf{y}_k\} \subset X$ be the $k$-neareast neighbor (one can also use a radius neighbor) of $\mathbf{x}$. Construct the distance matrix $\mathbf{D}:=[\mathbf{D}_1,\ldots, \mathbf{D}_k] \in \mathbb{R}^{n\times k}$, where $k>d$ and $\mathbf{D}_i: = \mathbf{y}_i-\mathbf{x}$ for $i=1,...,k$.
\item Take a singular value decomposition of $\mathbf{D} = \mathbf{U}\mathbf{\Sigma}\mathbf{V}^\top$. Then the leading $d-$ columns of $\mathbf{U}$ consists of $\mathbf{\tilde{T}}$ which approximates a span of column vectors of $\mathbf{T}$, which forms a basis of $T_{\mathbf{x}}M$.
\item Approximate $\mathbf{P}$ with $\mathbf{\tilde{P}} = \mathbf{\tilde{T}}\mathbf{\tilde{T}}^\top$.
\end{enumerate}
This method (as discussed in \cite{harlim2022rbf}) turns out to converge with a convergence rate of $N^{-1/d}$ for the estimation of $d$-dimensional manifolds. For low-dimensional ($d\leq 2$) manifolds and the symmetric approximation of the weighted Laplacian, it is sufficient to use the approximated projection operator $\tilde{\mathbf{P}}$ in applications since the overall error of the estimated eigensolutions is dominated by a Monte-Carlo rate (see Theorems~4.1-4.2 in \cite{harlim2022rbf}). For pointwise operator estimation, on the other hand, the error induced by the approximated projection operator $\tilde{\mathbf{P}}$ dominates the overall approximation error. In such a case, we will consider the second-order local SVD method in our numerical experiments with an improved error rate $N^{-2/d}$. The improvement is due to additional procedures that remove the error induced by the curvature terms in \eqref{eqn:iots}. Specifically, suppose we denote the components of the estimated local tangent vectors from the first-order local SVD method by $\mathbf{\tilde{T}} = [\bm{\tilde{t}}_1,\ldots, \bm{\tilde{t}}_d] \in \mathbb{R}^{n\times d}$. Instead of employing Step~3 above, we employ the following steps:
\begin{enumerate}
\item For each neighbor $\{\mathbf{y}_i\}_{i=1,\ldots,k}$ of $\mathbf{x}$, compute $\mathbf{\tilde{s}}^{(i)} = (\tilde{s}^{(i)}_1,\ldots,\tilde{s}^{(i)}_d)$, where
\[
\tilde{s}^{(i)}_j = \mathbf{D}_i^\top \bm{\tilde{t}}_j, \quad i= 1,\ldots, k, j=1,\ldots, d,
\]
where $\mathbf{D}_i: = \mathbf{y}_i-\mathbf{x}$ is the $i$th column of $\mathbf{D}\in \mathbb{R}^{n\times k}.$

\item Approximate the Hessian $\mathbf{Y}_{p}=\frac{\partial ^{2}\iota (%
\mathbf{0})}{\partial s_{i}\partial s_{j}} \in \mathbb{R}^n$ using the following ordinary least
squares regression, with $p= 1,\ldots ,D=d(d+1)/2$ denoting the upper triangular components ($p\mapsto (i,j)$ such that  $i\leq  j$) of symmetric Hessian matrix. Notice that for each $\mathbf{y}_\ell\in \{\mathbf{y}_1,\ldots, \mathbf{y}_k\}$ neighbor of $\mathbf{x}$, Eq.~(\ref%
{eqn:iots}) can be written as
\begin{equation}
\sum_{i,j=1}^{d}s_{i}^{(\ell)}s_{j}^{(\ell)}\frac{\partial ^{2}\iota (\mathbf{0})}{\partial
s_{i}\partial s_{j}}=2\left( \iota \left( \mathbf{s}\right) -\iota (\mathbf{0}%
)\right) - 2\sum_{i=1}^{d}s_{i}^{(\ell)}\bm{\tau}_i + O(s^3), \quad \quad \ell=1,\ldots k.  \label{eqn:sisj}
\end{equation}%
In compact form, we can rewrite \eqref{eqn:sisj} as a linear system,
\BEA
\mathbf{C} \mathbf{Y} = 2\mathbf{D}^\top - 2\mathbf{R}+ O(s^3),\label{eqn:Yti}
\EEA
where $\mathbf{R}^\top = (\mathbf{r}_1,\ldots,\mathbf{r}_k)\in\mathbb{R}^{n\times k}$ with each $\mathbf{r}_{j}: = \sum_{i=1}^{d}s_{i}^{(j)}\bm{\tau}_i$ being a vector in the tangent space $T_{\mathbf{x}}M$.
Here, $\mathbf{Y} \in \mathbb{R}^{D\times n}$ is a matrix whose $p$th row is  $\mathbf{Y}_{p}$
and
\BEA
\mathbf{C} = \begin{pmatrix} (s_1^{(1)})^2 & (s_2^{(1)})^2 & \ldots & (s_d^{(1)})^2 & 2(s_1^{(1)}s_2^{(1)}) & \ldots & 2(s_{d-1}^{(1)}s_d^{(1)})  \\ \vdots & \vdots  & & \vdots & \vdots  && \vdots\\ (s_1^{(k)})^2 & (s_2^{(k)})^2 & \ldots & (s_d^{(k)})^2 & 2(s_1^{(k)}s_2^{(k)}) & \ldots & 2(s_{d-1}^{(k)}s_d^{(k)})\end{pmatrix} \in \mathbb{R}^{k\times D}.\label{A_tangent}
\EEA
With the choice of $k>D$, we approximate $\mathbf{Y}$ by solving an over-determined linear problem
\[
\mathbf{\tilde{C}} \mathbf{Y} = 2\mathbf{D}^\top,
\]
where $\mathbf{\tilde{C}}$ is defined as in \eqref{A_tangent} except that  $s_i^{(j)}$ in the matrix entries is replaced by $\tilde{s}_i^{(j)}$.
The regression solution is given by $\tilde{\mathbf{Y}} =(\tilde{\mathbf{C}}^{\top }\tilde{\mathbf{C}})^{-1}\tilde{\mathbf{C}}^{\top }\mathbf{D}^\top$.
Here, each row of $\tilde{\mathbf{Y}}$ is denoted as $\tilde{\mathbf{Y}}_p = (\tilde{Y}_p^{(1)},\ldots,\tilde{Y}_p^{(n)})\in \mathbb{R}^{1\times n}$, which is an estimator of $\mathbf{Y}_p$ up to a difference of a vector in $T_{\mathbf{x}}M$.

\item Apply SVD to
\BEA
2\tilde{\mathbf{R}}^\top :=  2\mathbf{D} - (\tilde{\mathbf{C}}\tilde{\mathbf{Y}})^\top \in
\mathbb{R}^{n\times k}.\label{eq:Rtilde}
\EEA
Let the leading $d$ left singular vectors be denoted as $\bm{\hat{\Psi}}=[\hat{\bm{\psi}}_{1},\ldots ,\hat{\bm{\psi}}_{d}] \in \mathbb{R}^{n\times d}$, which is an estimator of $\bm{\Psi}=[\bm{\psi}_{1},\ldots, \bm{\psi}_{d}]$, where $\bm{\psi}_{j}$ are the leading $d$ left singular vectors of $\mathbf{R}$ as defined in \eqref{eqn:Yti}. We define $\mathbf{\hat{P}=\bm{\hat{\Psi}}\bm{\hat{\Psi}}}^{\top }$ as the estimator for $\mathbf{P} = \bm{\Psi}\bm{\Psi}^\top$.
\end{enumerate}
For the theoretical convergence and its consequence, see Theorem~3.2 and Remark~5 of \cite{harlim2022rbf}.

\section{Spectral PDE solvers}\label{sec3}
In this section, we consider solving the elliptic equation (\ref{eq1}) based on a Galerkin method. In general, we consider a Hilbert space $\mathcal{H}$ of functions $f:M\to \mathbb{R}$, equipped with an inner product $(\cdot,\cdot)_\mathcal{H}$ as a test function space. Suppose that $\mathcal{H}$ is a separable Hilbert space with a set of orthonormal basis functions $\{\varphi_j\}_{j=1}^\infty$.  A Galerkin approximation to the PDE (\ref{eq1}) is to find a solution $u$ of the form $u=\sum_{k=1}^Ku_k\varphi_k$ with unknown coefficients $\{u_k\}_{k=1}^K$ such that
\BEA
(c u-\textup{div}_g(\kappa\textup{grad}_g u),\varphi_j)_\mathcal{H}=(f,\varphi_j)_\mathcal{H},\quad \text{for}\ j=1,...,K.
\label{Galerkin}
\EEA
By the product rule of divergence operator on Riemannian manifolds, we have
$$
-\text{div}_g(\kappa \textup{grad}_g  u)=\kappa\Delta_M u-\langle \textup{grad}_g \kappa,   \textup{grad}_g  u\rangle_g.
$$
Hence, the Galerkin approximation (\ref{Galerkin}) can be rewritten as
\BEA
\sum_{k=1}^Ku_k\big(c\varphi_k+\kappa\Delta_M \varphi_k-\langle  \textup{grad}_g  \kappa,  \textup{grad}_g  \varphi_k\rangle_g,\varphi_j\big)_\mathcal{H}=\big(f,\varphi_j\big)_\mathcal{H},\quad \text{for}\ j=1,...,K.
\label{Galerkin2}
\EEA
The $K$ unkown coefficients $\textup{U}=(u_1,...,u_K)^\top$ can be obtained by solving {\color{black}the} above $K$ equations once each inner product $(\cdot,\cdot)_\mathcal{H}$ is calculated. For the theoretical analysis, we write the Galerkin linear system as $\textup{A} \textup{U} = \textup{b}$, where
\begin{equation}\label{Ab}
\begin{split}
\textup{A}_{jk} &:= \big(c\varphi_k+\kappa\Delta_M \varphi_k-\langle  \textup{grad}_g  \kappa,  \textup{grad}_g  \varphi_k\rangle_g,\varphi_j\big)_\mathcal{H},  \\
\textup{b}_j &:= \big(f,\varphi_j\big)_\mathcal{H}.
\end{split}
\end{equation}
In general, the exact values of integrals over $M$ are not easy to compute. We will resort to the relevant Monte-Carlo approximation over the point cloud data set $X=\{\textbf{x}_{i}\}_{i=1}^{N} \subset M$ that is assumed to be distributed in accordance to sampling measure $d\mu = q d\textup{Vol}$. The specific Monte-Carlo approximation will depend on the choice of the test function space $\mathcal{H}$ and its inner-product $(\cdot,\cdot)_{\mathcal{H}}$, which we will discuss in the next subsection. Subsequently, we conclude this section with a convergence analysis. In the following, we denote the weighted spaces $L^2(\mu):=\{u:\int_M u^2d\mu<\infty\}$ and $H^1(\mu):=\{u\in L^2(\mu): X u\in L^2(\mu) \text{ for any smooth vector field } X \text{on }M\}$.


\subsection{Discrete formulation} \label{sec3.1}

In this discussion, our aim is to develop a formulation that appropriately accounts for the sampling measure, $d\mu = qdV$, that is encoded by the sampled data. Particularly, we will consider a test function space that allows one to avoid the needs of using and/or estimating $q$. Specifically, we consider $\mathcal{H}$ to be $\overline{\textup{span}\{\varphi_k\}_{k\geq 0}} = L^2(\mu)$, where $\varphi_k$ denotes the $k$th-orthonormal eigenfunction associated to eigenvalue $\lambda_k\geq 0$ of the weighted Laplacian operator, $-q^{-1}\textup{div}_g (q\,\textup{grad}_g\,):H^2(\mu)\rightarrow L^2(\mu)$. In this case, given a set of point cloud data $X=\{\textbf{x}_{i}\}_{i=1}^{N} \subset M$ that is assumed to be distributed in accordance to sampling measure  $d\mu = q d\textup{Vol}$, we approximate the integral of any function $f\in M$ as,
\[
\int_M f d\mu \approx \frac{1}{N} \sum_{i=1}^N f(\mathbf{x}_i).
\]
Then, the equations in (\ref{Galerkin2}) can be numerically estimated by a $K\times K$ linear system, with the $j$th row given as,
\BEA
\sum_{k=1}^K \mathbf{A}_{jk} u_k= \frac{1}{N}\mathbf{f}^\top \boldsymbol{\varphi}_j,\quad \text{for}\ j=1,...,K,
\label{Galerkin-eigen2}
\EEA
where we have defined $\boldsymbol{\varphi}_j=(\varphi_j(\mathbf{x}_1),\ldots,\varphi_j(\mathbf{x}_N))^\top$, and
\BEA
\mathbf{A}_{jk}:=\frac{1}{N} \sum_{i=1}^N\Big(c(\textbf{x}_i) \varphi_k(\textbf{x}_i)+\kappa(\textbf{x}_i)\Delta_M\varphi_k(\textbf{x}_i)-\big\langle \textup{grad}_g \kappa(\textbf{x}_i), \textup{grad}_g \varphi_k(\textbf{x}_i)\big\rangle_g\Big)\varphi_j(\textbf{x}_i).\label{A}
\EEA

We approximate the eigen-solutions of $-q^{-1}\textup{div}_g (q\,\textup{grad}_g\,)$ by solving the eigenvalue problem corresponding to the Dirichlet energy in \eqref{weightedLaplacian} over all functions $H^1(\mu)$. Since,
\[
\langle -q^{-1}\textup{div}_g (q\,\textup{grad}_g \varphi_k),\varphi_k \rangle_{L^2(\mu)} =  \langle  \textup{grad}_g \varphi_k,\textup{grad}_g \varphi_k \rangle_{L^2(\mu)},
\]
for discrete data set $X$ one can employ the Monte-Carlo estimate in \eqref{weightedLaplacian} to approximate the orthonormal eigenfunctions (in $L^2(\mu)$ sense) by solving the following discrete eigenvalue problem:
\BEA
\Delta_X^{\mathrm{SRBF}} \boldsymbol{\tilde{\varphi}}_k := \mathbf{G}^\top \mathbf{G} \boldsymbol{\tilde{\varphi}}_k = \tilde{\lambda}_k \boldsymbol{\tilde{\varphi}}_k, \quad k= 1,2,\ldots.\label{sym_eigs}
\EEA

Applying the RBF pointwise approximation for the Laplacian in \eqref{NRBF_Laplacian} and approximating the eigenvectors based on the symmetric formulation in \eqref{sym_eigs}, we arrive at a linear system $\mathbf{\tilde{A}} \mathbf{\tilde{U}} \mathbf{= \tilde{b}}$ as an approximation to \eqref{Galerkin-eigen2} with,
\BEA
\mathbf{\tilde{A}}_{jk}&:=& \frac{1}{N} \sum_{i=1}^N\left((c(\textbf{x}_i)+\kappa(\textbf{x}_i)\Delta_X^{\mathrm{RBF}}) \boldsymbol{\tilde{\varphi}}_k(\textbf{x}_i)-\sum_{\ell=1}^n(\mathbf{G}_\ell \boldsymbol{\kappa})_i(\mathbf{G}_\ell\boldsymbol{\tilde{\varphi}}_k)_i\right) \boldsymbol{\tilde{\varphi}}_j(\textbf{x}_i), \label{Atilde2} \\
\mathbf{\tilde{b}}_j &:=& \frac{1}{N} \mathbf{f}^\top \boldsymbol{\tilde{\varphi}}_j.\label{btilde2}
\EEA

In the following, we outline a computational procedure for the numerical approximation. As inputs, we are given a set of point cloud data  $X=\{\textbf{x}_{i}\}_{i=1}^{N}$ sampled from the smooth manifold $M\subset\mathbb{R}^n$ and we assume that the function values of $f$ and $\kappa$ on $X$ are known. Then,
\begin{enumerate}[(1)]
\item Use the second-order local SVD method to attain $\mathbf{\tilde{P}}$ as a discrete estimate of the projection matrix $\mathbf{P}$.
\item Construct the discrete gradient operator $\mathbf{G}_\ell$ as in (\ref{diff-op}).
\item Solve the eigenvalue problem in \eqref{sym_eigs}.
\item  Precompute and store
\BEA
\mathbf{D}_{ijk}:= c(\textbf{x}_i) \boldsymbol{\tilde{\varphi}}_k(\textbf{x}_i) \boldsymbol{\tilde{\varphi}}_j(\textbf{x}_i), \quad\quad
\mathbf{E}_{ijk\ell}:=(\mathbf{G}_\ell\boldsymbol{\tilde{\varphi}}_k)_i \boldsymbol{\tilde{\varphi}}_j(\textbf{x}_i),\quad\quad
\mathbf{F}_{ijk}:= \Delta^{\mathrm{RBF}}_X \boldsymbol{\tilde{\varphi}}_k(\textbf{x}_i) \boldsymbol{\tilde{\varphi}}_j(\textbf{x}_i),
\label{eqn:24}
\EEA
and $\mathbf{\tilde{A}}^{(1)}_{jk}:= \sum_{i=1}^N \mathbf{D}_{ijk}$.
\item Compute
\[
\mathbf{\tilde{A}}^{(2)}_{jk}:= \sum_{i=1}^N \kappa(\textbf{x}_i) \mathbf{F}_{ijk}, \quad\quad
\mathbf{\tilde{A}}^{(3)}_{jk}:= \sum_{i=1}^N \sum_{\ell=1}^n (\mathbf{G}_\ell \boldsymbol{\kappa})_i  \mathbf{E}_{ijk\ell},
\]
and ${\mathbf{\tilde{A}}}_{jk} = \frac{1}{N} (\mathbf{\tilde{A}}^{(1)}_{jk}+\mathbf{\tilde{A}}^{(2)}_{jk}+\mathbf{\tilde{A}}^{(3)}_{jk})$ and $\mathbf{\tilde{b}}_{j}:=\frac{1}{N} \mathbf{f}^\top \boldsymbol{\tilde{\varphi}}_j$.
\item   Solve the $K\times K$ linear system $\mathbf{\tilde{A}} \mathbf{\tilde{U}} \mathbf{= \tilde{b}}$ to obtain the approximate solution $\tilde{u}^K(\textbf{x}_i)=\sum_{k=1}^K\tilde{u}_k\boldsymbol{\tilde{\varphi}}_k(\textbf{x}_i)$ where $\tilde{\mathbf{U}}=(\tilde{u}_1,\ldots, \tilde{u}_K)^\top$.
\end{enumerate}

We will refer to this approximation as the \emph{Spectral-RBF} scheme. We emphasize that the computational procedure above is to anticipate solving the PDE many times for different $\kappa$ and/or $f$, which arises in inverse problems (see e.g.~\cite{harlim2020kernel,harlim2022graph}). In such a case, Steps (1)-(4) can be performed once (or offline) as they are independent of $\kappa$. For each choice of $\kappa$, we execute steps (5)-(6) to solve the PDE. Since these last two steps depends on $\kappa$, we refer to it as online.

The overall computational cost of the algorithm could be estimated as follows:

\textbf{Offline}: Here, the time complexity is dominated by Step (3), that is $O(N^3)$  {\color{black} operations} (see e.g., \cite{pan1998complexity}). Additionally, the space complexity of storing $\mathbf{G},\mathbf{D}, \mathbf{E}, \mathbf{F}, \mathbf{\tilde{A}}^{(1)}$ to be used in the online steps are $nN^2+NK^2+nNK^2+NK^2+K^2$, combining altogether in a consecutive fashion.



\textbf{Online}: Given $\kappa$ and $f$, the Step~(5) above involves calculating $n$ times matrix-vector multiplications ($\mathbf{G}_\ell \boldsymbol{\kappa},1\leq \ell\leq n$), $K^2$ times the computations of $\mathbf{\tilde{A}}^{(2)}_{jk}$, $\mathbf{\tilde{A}}^{(3)}_{jk}$, $K$ times the computations of $\mathbf{\tilde{b}}_j$, and $2K^2$-additions in constructing $\mathbf{\tilde{A}}$. The time complexities of these computations are consecutively given as, $O(nN^2)+O(K^2N)+O(K^2Nn) + O(KN)+O(K^2)$  {\color{black} operations}, in addition to $O(K^3)$ {\color{black} operations} for solving the linear problem in Step~(6). {\color{black} We should point out that the time complexity in Step~(5) above can be further reduced by a parallel computation on each component of $\mathbf{\tilde{A}}^{(2)}_{jk}$, $\mathbf{\tilde{A}}^{(3)}_{jk}$.}

If one only wants to solve the PDE one time, we should point out that there is no computational saving with the algorithm above relative to the direct inversion of a pointwise approximation, which requires $O(N^3)$  {\color{black} operations}  for an inverse of an $N\times N$ matrix. However, if one wants to solve the PDE many times for various choices of $\kappa$, the algorithm above simply moves the online cost of $O(N^3)$ {\color{black} operations}  from the direct inversion for each $\kappa$ to an offline computational task of solving the eigenvalue problem in Step~(3) since this step is independent of $\kappa$, in addition to the storage complexity from the offline steps. The online computational cost becomes efficient whenever $K\ll N$.

We should point out that another means to overcome the high complexity in global RBF is to sacrifice the accuracy with the computationally cheaper local RBF interpolation technique \cite{shankar2015radial,lehto2017radial} for sparse matrices $\{\mathbf{G}_\ell \}_{\ell=1,\ldots,n}$. {\color{black}We will consider this local interpolation scheme on one of our numerical examples and will refer to this scheme as \emph{Spectral-RBF-FD}.} If we employed the local RBF interpolation, the time complexity in our offline step (2) would  reduce to $O(N\log N)$  {\color{black} operations}  where the constant can depend on the number of local neighbors taken, the complexity in step (3) would also reduce to $O(K^2N)$  {\color{black} operations}, and the online steps (5) and (6) would also reduce to $O(N\log N)$  {\color{black} operations} , thanks to the sparsity of $\mathbf{G}_\ell$.

\begin{rem}[Alternative basis]\label{altbasis} While we discussed an expansion of a particular test function space induced by the leading eigenfunctions of the weighted Laplacian operator, $-q^{-1}\textup{div}_g (q\,\textup{grad}_g\,)$, the proposed formulation can be employed using basis functions of other second-order elliptic operators, such as the Laplace-Beltrami operator, $\Delta_M$. In fact, one can also use eigensolutions from different approximation methods, such as the graph Laplacian-based approximation. In our numerical study, we will demonstrate such an approach with the leading eigenbasis of the Laplace-Beltrami operator that is estimated by the Variable Bandwidth Diffusion Maps (VBDM) algorithm \cite{bh:16vb} (which will be discussed in detail in Section~\ref{sec4.1}). In such {\color{black} an }approximation, however, one approximates only the function values $\varphi_j(\mathbf{x}_i)$ at all training data point $\mathbf{x}_i\in X$. To evaluate the derivatives of the functions, such as $\textup{grad}_g \varphi_j(\mathbf{x}_i)$ and $\Delta_M \varphi_j(\mathbf{x}_i)$, which is needed in approximating components of $\mathbf{A}$ in \eqref{A}, we still employ the RBF pointwise approximation. We should point out that although the VBDM solution produces the eigenvalues associated to the eigenfunction estimates that can be used to approximate $\Delta_M \varphi_j(\mathbf{x}_i) = \lambda_j \varphi_j(\mathbf{x}_i) \approx \tilde{\lambda}_j^{\mathrm{VBDM}} \boldsymbol{\tilde{\varphi}}_j^{\mathrm{VBDM}}(\mathbf{x}_i)$, we numerically found that it is more accurate to use the pointwise RBF approximation, $\Delta_M \varphi_j(\mathbf{x}_i)\approx \Delta_X^{\mathrm{RBF}} \boldsymbol{\tilde{\varphi}}_j^{\mathrm{VBDM}}(\mathbf{x}_i)$. We will refer to this approach as VBDM-RBF.

\end{rem}

\subsection{Convergence analysis}

In this section, we will study the convergence of the proposed Galerkin approximation. To simplify the notation, we only consider uniformly distributed data such that $q\propto 1$. The convergence result is stated in Theorem \ref{error-estimate}. For the weighted space, the analysis follows the same argument with a weighted $L^2$ space with respect to a positive smooth density function $q$ {that is bounded, $0<q_{min}\leq q(x)\leq {q}_{max}$}.

Quoting the basic result for the elliptic PDE in \eqref{eq1}, we assume that $\kappa \in C^1(M), c\in C(M)$ with $\kappa(x)\geq \kappa_{min}>0$ and $c(x)\geq c_{min}>0$ for all $x\in M$. In the following, we denote the $L^2$ inner product on $M$ as $\langle\cdot,\cdot\rangle_{L^2(M)}$. For two vector fields $X,Y\in\mathfrak{X}(M)$, we define the inner product $\langle X,Y\rangle_{L^2(\mathfrak{X}(M))}:=\int_M\langle X,Y\rangle_gd\textup{Vol} $. By applying integration by parts on closed manifolds, a weak formulation of the elliptic problem (\ref{eq1}) is to find $u\in V= H^1(M)$ such that
\BEA
B(u,v)=l(v),\quad \forall v\in V,
\label{weak-form}
\EEA
where $l(v)=\langle f,v\rangle_{L^2(M)}$ and the bilinear form $B(\cdot,\cdot)$ is defined by
$$
B(u,v)=\langle \kappa\textup{grad}_g u,\textup{grad}_g v\rangle_{L^2(\mathfrak{X}(M))}+\langle c u,v\rangle_{L^2(M)}.
$$
The existence and uniqueness of the solution $u$ follows from the Lax-Milgram theorem and the inequality: $\min\{c_{min},\kappa_{min}\}\Vert v\Vert_{H^1(M)}\leq B(v,v)$.

\begin{rem}
We should point out that the weak solution also exists when we consider $L^2(\mu)$ as the test function space as discussed in {\color{black}the} previous section. In this case, suppose that $u\in H^2(\mu)$ is the solution of the elliptic problem (\ref{eq1}), then
\BEA
\langle -\textup{div}_g(\kappa\, \textup{grad}_g u)+c u,v\rangle_{L^2(\mu)}=\langle f,v\rangle_{L^2(\mu)}, \quad\forall v\in L^2(\mu).
\EEA
This variational formula suggests that the weak formulation is to find $u\in H^1(\mu)$ that solves,
$$
B_q(u,v):= \langle \kappa\,\textup{grad}_g u,\textup{grad}_g (v q)\rangle_{L^2(\mathfrak{X}(M))}+\langle c u,v\rangle_{L^2(\mu)} =\langle f,v\rangle_{L^2(\mu)} := l_q(v),\quad \forall u, v\in H^1(\mu).
$$

From the ellipticity condition, we have
\BEA
\kappa_{min}\Vert \textup{grad}_g u\Vert_{L^2(\mu)}^2 &\leq& \langle q \kappa\textup{grad}_g u,\textup{grad}_g u\rangle_{L^2(\mathfrak{X}(M))} \notag \\
&=& B_q(u,u) -\int_M\kappa u \langle \textup{grad}_g u,\textup{grad}_g q\rangle_g d\textup{Vol}  -\int_Mc  u^2q d\textup{Vol}  \notag \\
&\leq& B_q(u,u)  +\Vert\kappa/q\Vert_{L^\infty(M)} \int_Mq|\textup{grad}_g q|_g  \left\langle \textup{grad}_g u, u\frac{\textup{grad}_g q}{|\textup{grad}_g q|_g}\right\rangle_g d\textup{Vol} +  \|c\|_{L^\infty}\Vert u\Vert^2_{L^2(\mu)}\notag\\
&\leq& B_q(u,u)+\Vert\kappa/q\Vert_{L^\infty(M)}\Vert  |\textup{grad}_g q|_g\Vert_{L^\infty(M)} \left(\varepsilon \Vert \textup{grad}_g u\Vert_{L^2(\mu)}^2+\frac{1}{4\varepsilon}\Vert u\Vert^2_{L^2(\mu)}\right)+  \|c\|_{L^\infty}\Vert u\Vert^2_{L^2(\mu)}\notag
\EEA
for any $\varepsilon>0$. Here, $|X|_g:=\langle X,X\rangle_g^{1/2}$ for any vector field $X\in\mathfrak{X}(M)$. Choose $\varepsilon>0$ such that $\Vert\kappa/q\Vert_{L^\infty}\Vert  |\textup{grad}_g q|_g\Vert_{L^\infty} \varepsilon = \kappa_{min}/2$, we have an energy estimate that is similar to the one stated in Theorem~2 in Section 6.2.2 of \cite{evans2010partial},
\[
\beta\Vert \textup{grad}_g u\Vert_{L^2(\mu)}^2 \leq B_q(u,u)+ \gamma \Vert u\Vert_{L^2(\mu)}^2,
\]
where $\beta := \kappa_{min}/2>0$ and $\gamma:= \frac{1}{4\varepsilon}\Vert\kappa/q\Vert_{L^\infty}\Vert  |\textup{grad}_g q|_g\Vert_{L^\infty} +\|c\|_{L^\infty}$. While this suggests that $B_q$ is not coercive with respect to $H^1(\mu)$,
one can follow the argument in the proof of Theorem~4 in Section 6.2.3 of \cite{evans2010partial} to verify the existence of weak solution $u\in H^1(\mu)$.
\end{rem}

Furthermore, the regularity theorem (Chapter 5, Proposition 1.6 in \cite{taylor2013partial}) implies that if $f\in L^2(M)$ and $u\in H^1(M)$ is the unique solution, then $u\in H^2(M)$ with a constant $C>0$ such that
$$
\Vert u\Vert_{H^2(M)}\leq C\Vert f\Vert_{L^2(M)}.
$$
Let $V^K=\text{span}\{\varphi_k\}_{k=1}^K\subset V$, then a discrete problem of the above weak formulation in (\ref{weak-form})  is to find $u^K\in V^K$ such that
\BEA
B(u^K,v)=l(v),\quad \forall v\in V^K.
\label{weak-form-finite}
\EEA

Let $\tilde{u}^K$ be the solution found by the scheme Spectral-RBF proposed  in the last section, i.e.,
\BEA
\tilde{u}^K=\sum_{k=1}^K\tilde{u}_k\tilde{\varphi}_k
\label{spectral-RBF-soln}
\EEA
where we let $\tilde{\varphi}_k:= I_{\phi_s}\boldsymbol{\tilde{\varphi}}_k$ and the coefficient $\tilde{\mathbf{U}}=(\tilde{u}_1,\ldots, \tilde{u}_K)^\top$ is the solution of $\tilde{\mathbf{A}}\tilde{\mathbf{U}}=\tilde{\mathbf{b}}$ with $\tilde{\mathbf{A}}$ and $\tilde{\mathbf{b}}$ defined in (\ref{Atilde2})-(\ref{btilde2}).   \comment{{\color{black} need to show that $\tilde{A}$ is positive definite matrix or nonsingular.} {\color{black} This is in the Assumption \eqref{Anonsingular_assumption}, which is our case for $\epsilon \textup{cond}(\textup{A}) < 1$ in Lemma~3.9, which is equivalent to saying that matrix $\tilde{\mathbf{A}}$ is nonsingular. You can check Theorem 2.1 in James Demmel. Almost the entire proof in Lemma~3.12 to check the size of perturbation that one cannot exceeds to guarantee invertibility of $\tilde{\mathbf{A}}$.}}

The error estimates in {\color{black}the} $L^2$ norm can be divided into two parts:
\BEA
\Vert u-\tilde{u}^K\Vert_{L^2(M)}\leq \Vert u-u^{K}\Vert_{L^2(M)}+\Vert u^K-\tilde{u}^{K}\Vert_{L^2(M)},
\label{error_estimate}
\EEA
where the first part corresponding to the bias-term induced by the truncation and can be bounded above by the standard approach involving Cea's Lemma and spectral interpolation error estimate. Specifically, we can have {\color{black} an }error bound on $u-u^K$ in $H^1$ norm,
\begin{lem} Let $u\in H^2(M)$ be a solution of (\ref{weak-form}) and $u^K\in V^K$ be a solution of (\ref{weak-form-finite}). Here, $V^K=\textup{span}\{\varphi_k\}_{k=1}^K$  with $\{\lambda_k\}$ the associated increasing-ordered eigenvalues of $\Delta_M$ on $M$. then
\BEA
\Vert u-u^K\Vert_{H^1(M)}\leq C \left(\frac{2}{1+\lambda_K}\right)^{1/2}\Vert u\Vert_{H^2(M)}.
\label{cea-2}
\EEA
\end{lem}

\begin{proof} It follows immediately from Cea's Lemma that there exists a constant $C$ such that
\BEA
\Vert u-u^K\Vert_{H^1(M)} \leq C \inf\limits_{v\in V^K}\Vert u-v\Vert_{H^1(M)}.
\label{cea}
\EEA
Let $u=\sum_{k=1}^\infty\langle u,\varphi_k\rangle_{L^2} \varphi_k$ and $v^K=\sum_{k=1}^K\langle u,\varphi_k\rangle_{L^2} \varphi_k\in V^K$, then we have
$$
\big\Vert u-v^K\big\Vert^2_{H^1(M)}\leq \big\Vert \sum_{k=K+1}^\infty \langle u,\varphi_k\rangle_{L^2(M)} \varphi_k\big\Vert^2_{H^1(M)}=  \sum_{k=K+1}^\infty (1+\lambda_k)\langle u ,\varphi_k\rangle_{L^2(M)}^2.
$$
Here, we use the assumption that $\varphi_k$ are eigenfunctions of $\Delta_M$ which implies that
$$\langle \textup{grad}_g \varphi_i,\textup{grad}_g\varphi_j\rangle_{L^2(\mathfrak{X}(M))}=\lambda_i\langle \varphi_i,\varphi_j\rangle_{L^2(M)}=\lambda_i\delta_{ij}.$$
Note that $\{\lambda_k\}$ are the increasing ordered eigenvalues, then we have the spectral interpolation error:
\BEA
\Vert u-v^K\Vert^2_{H^1(M)}\leq  \sum_{k=K+1}^\infty (1+\lambda_k)\langle u ,\varphi_k\rangle^2_{L^2(M)}\leq \frac{1}{1+\lambda_K}\sum_{k=1}^\infty (1+\lambda_k)^2\langle u ,\varphi_k\rangle^2_{L^2(M)}\leq \frac{2}{1+\lambda_K}\Vert u\Vert^2_{H^2(M)},
\label{spectral-interpolation}
\EEA
where in the last inequality we use the fact that
$$
\sum_{k=1}^\infty (1+\lambda_k)^2\langle u ,\varphi_k\rangle_{L^2(M)}^2=\sum_{k=1}^\infty (1+2\lambda_k+\lambda_k^2)\langle u ,\varphi_k\rangle_{L^2(M)}^2=\Vert u\Vert^2_{L^2(M)}+2\vert u\vert^2_{H^1(M)}+\Vert \Delta_Mu\Vert^2_{L^2(M)} \leq 2\Vert u\Vert^2_{H^2(M)}.
$$
Take $v=v^K$ in (\ref{cea}) and  combine it with the inequality (\ref{spectral-interpolation}), we arrive at the desired bound in \eqref{cea-2}.
\comment{then we have
$$
\Vert u-u^K\Vert_{H^1} \leq C \inf\limits_{v\in V^K}\Vert u-v\Vert_{H^1}\leq C \left(\frac{2}{1+\lambda_K}\right)^{1/2}\Vert u\Vert_{H^2}.
$$}
\end{proof}

\begin{rem}
If $u$ is in the weighted space $ H^2(\mu)$, where $d\mu = q d\textup{Vol}$, and we take $\varphi_k$ be the eigenfunction of the weighted Laplacian operator $-q^{-1}\textup{div}_g (q \textup{grad}_g\,)$ that is self-adjoint in $L^2(\mu)$, i.e., $\langle \textup{grad}_g \varphi_i,\textup{grad}_g v\rangle_{L^2(\mu)}=\lambda \langle\varphi_i,v\rangle_{L^2(\mu)}, \forall v\in H^1(\mu)$, then above inequality also holds in the weighted Sobolev norm.
\end{rem}

By combining the lemma above and the standard Aubin-Nitche duality argument, we can have an upper bound of the first term on right hand side of (\ref{error_estimate}).
\begin{lem}\label{lemma_error1}
 Let $u\in H^2(M)$ be a solution of (\ref{weak-form}) and $u^K\in V^K$ be a solution of (\ref{weak-form-finite}) where $V^K$ is defined as in {\color{black}the} previous lemma. Then there exists a constant $C>0$ such that
\BEA
\Vert u-u^K\Vert_{L^2(M)}\leq C\left(\frac{2}{1+\lambda_K}\right)\Vert u\Vert_{H^2(M)}.\notag
\EEA
\end{lem}
\begin{proof}
By the regularity theorem (Chapter 5, Proposition 1.6 in \cite{taylor2013partial}), the adjoint problem
\BEA
B(w,v)=\langle u-u^K,v\rangle_{L^2(M)},\quad\forall v\in V,
\label{dual}
\EEA
has a unique solution $w\in H^2(M)$ with
\BEA
\Vert w\Vert_{H^2(M)}\leq C\Vert u-u^K\Vert_{L^2(M)}.
\label{dual-reg}
\EEA
for some generic constant $C>0$.
Choose $v=u-u^K$ in (\ref{dual}). It follows that
$$
\begin{aligned}
\Vert u-u^K\Vert_{L^2(M)}^2=\left\langle u-u^K,u-u^K\right\rangle_{L^2(M)}&=B(w,u-u^K)=B(w-v^K,u-u^K)\\
&\leq C \Vert u-u^K\Vert_{H^1(M)}\Vert w-v^K\Vert_{H^1(M)},\quad\forall v^K\in V^K,
\end{aligned}
$$
where we have used the continuity of the bilinear form $B$ and Galerkin orthogonality $B(u-u^K,v^K)=0,\forall v^K\in V^K$. If we choose $v^K=\sum_{k=1}^K\langle w,\varphi_k\rangle\varphi_k$, then the spectral interpolation error proved in the previous lemma shows that
$$
\Vert w-v^K\Vert_{H^1(M)}\leq C\left(\frac{2}{1+\lambda_K}\right)^{1/2}\Vert w\Vert_{H^2(M)}.
$$
Consequently,
$$
\begin{aligned}
\left\Vert u-u^K\right\Vert_{L^2(M)}^2&\leq C\left\Vert u-u^K\right\Vert_{H^1(M)}\left\Vert w-v^K\right\Vert_{H^1(M)}\\
& \leq C \left(\frac{2}{1+\lambda_K}\right)^{1/2}\Vert u\Vert_{H^2(M)} \left(\frac{2}{1+\lambda_K}\right)^{1/2}\Vert w\Vert_{H^2(M)}\\
 &\leq C \left(\frac{2}{1+\lambda_K}\right) \Vert u\Vert_{H^2}\Vert u-u^K\Vert_{L^2(M)}
\end{aligned}
$$
where we have used (\ref{dual-reg})  in the last inequality to complete the proof and abused the notation $C>0$ for different constants in each of the inequality above.
\end{proof}
Next, we consider the second term on right hand side in (\ref{error_estimate}), which is sometimes called the variance error in statistical learning literature. From the previous notations, we recall that  $u^K$ and $\tilde{u}^K$ could be written as $u^K=\sum_{k=1}^Ku_k\varphi_k$ and  $\tilde{u}^K=\sum_{k=1}^K\tilde{u}_k\tilde{\varphi}_k$, where $\tilde{\varphi}_k:=I_{\phi_s}\tilde{\boldsymbol{\varphi}}_k$, such that,
\BEA
\Vert u^K-\tilde{u}^{K}\Vert_{L^2(M)}\leq \sum_{k=1}^K|u_k|\Vert \varphi_k-\tilde{\varphi}_k\Vert_{L^2(M)}+\sum_{k=1}^K\Vert \tilde{\varphi}_k\Vert_{L^2(M)}|u_k-\tilde{u}_k|.
\label{error_estimate_2}
\EEA

To estimate the first term in right hand side of (\ref{error_estimate_2}), we need the following results, which we quote from \cite{harlim2022rbf} for reader's convenience.
\begin{lem}(Lemma 4.1 of \cite{harlim2022rbf})
\label{RKHS L2 Convergence}
Assume that the nodes in the data set $X=\{\mathbf{x}_i\}_{i=1}^N$ are uniformly i.i.d. samples of $M$. Let $\phi_s$ be a kernel whose RKHS norm is norm equivalent to a Sobolev space $H^\alpha(\mathbb{R}^n)$ of order $\alpha > {n/2}$. Then with probability $1 - {\frac{1}{N}}$, we have
$$
\| I_{\phi_s} \mathbf{f} - f\|_{L^2(M)} = {O\left( N^{\frac{-2\alpha + (n-d)}{2d}} \right)},
$$
as $N\to\infty$, for all $f\in C^\infty(M)$.
\end{lem}
\begin{rem}
If X is sampled from a smooth density function $q(x)$ with $q(x)\geq q_{min}>0,\forall x\in M$,  following the same proof in Lemma 4.1 of \cite{harlim2022rbf},
one can deduce the following interpolation error
$$
\Vert I_{\phi_s} \mathbf{f} - f\Vert_{L^2(\mu)}={O\left( \big(q_{min}N\big)^{\frac{-2\alpha + (n-d)}{2d}} \right)},
$$
as $N\to\infty$, with the sampling measure $d\mu=qd\textup{Vol}$. Here, the additional factor $q_{min}$ contribute to a larger error bound.

\comment{Indeed, one can follow the arguments in \cite{arcangeli2007extension} and \cite{fuselier2012scattered} to show that
$$
\Vert I_{\phi_s} \mathbf{f} - f\Vert_{L^2(\mu)}\leq Ch^{\alpha-(n-d)/2}_{X,M}\Vert  f\Vert_{H^{\alpha-(n-d)/2}(\mu)}
$$
where the fill distance $h_{X,M}:=\sup\limits_{x\in M}\min\limits_{x_i\in X}d_g(x,x_i)$ with $d_g:M\times M\rightarrow\mathbb{R}$ the geodesic distance. Then one can show that $h_{X,M}=O((q_{min}N)^{-1/d})$ with high probability by following the same proof in Lemma 4.1 of \cite{harlim2022rbf}.}
\end{rem}

Since the radial basis interpolator is known to be unstable (see Chapter~12 of \cite{Wendland2005Scat}), this motivates the following notion of weakly unstable interpolator assumption.

\begin{assumption}[Weakly Unstable Interpolator]\label{weaklyunstable}
Let $\phi_s$ be a kernel whose RKHS norm is norm equivalent to a Sobolev space $H^\alpha(\mathbb{R}^n)$ of order $\alpha > {n/2}$. We should clarify that this class of kernel ensures an interpolation map $I_{\phi_s}:\mathbb{R}^N \to C^{\alpha-\frac{n-d}{2}}(M)$ such that $I_{\phi_s}\mathbf{f} \in W^{\alpha-\frac{n-d}{2},\infty}(M)$ for any $\mathbf{f}\in \mathbb{R}^N$. We assume that $\alpha - \frac{n-d}{2} \geq 2$ and,
\[
\| I_{\phi_s}\mathbf{f} \|^2_{W^{2,\infty}(M)} =O\left(N^{\frac{1-\beta}{2}}\right),
\]
as $N\to\infty$, for some $0<\beta \leq 1$ and any $f\in C^\infty(M)$.
\end{assumption}

We should point out that while numerically we desire a stable interpolator (with $\beta=1)$, one can show that the Mat\'{e}rn kernel that is adopted in this theoretical study is unstable. In fact, the condition number of the matrix $\boldsymbol{\Phi}$ grows polynomially with respect to the size of the data, $N$. Motivated by this fact, this assumption allows for an instability with a polynomial growth rate as a function of $N$ that grows weaker than the Monte-Carlo error rate (with $\beta >0$) to avoid vacuous bounds in the following discussion.

The error of approximated eigenfunctions are considered in the following norm which is consistent with $L^2(M)$ as the number of data points $N$ approaches infinity:
\begin{defn}
Given two vectors $\mathbf{f},\mathbf{h}\in\mathbb{R}^N$, we define $\langle \mathbf{f},\mathbf{h}\rangle_{L^2(\mu_N)}:=\frac{1}{N}\mathbf{f}^\top\mathbf{h}$ and $\Vert f\Vert_{L^2(\mu_N)}:=\langle \mathbf{f},\mathbf{f}\rangle^{1/2}_{L^2(\mu_N)}$.
\end{defn}

\begin{lem}\label{lem_eig_conv} (Theorems 4.1 and 4.2 of \cite{harlim2022rbf}) Let $\phi_s$ be a kernel whose RKHS norm is norm equivalent to a Sobolev space  $H^\alpha(\mathbb{R}^n)$  of order $\alpha> n/2$. Let $\lambda_i$ denote the $i$-th eigenvalue of $\Delta_M$, enumerated $0\leq \lambda_1\leq\lambda_2\leq\cdots $. Then there exists an eigenvalue $\tilde{\lambda}_i$ of $\mathbf{G}^\top\mathbf{G}$ such that
$$
\left|\lambda_{i}-\tilde{\lambda}_{i}\right| = {\color{black}O\left(N^{-\frac{1}{2}}\right)}+O\left(N^{\frac{-2 \alpha+(n-d)}{2 d}}\right),
$$
with probability greater than  $1-\frac{12}{N}$, as $N\to\infty$, for some $0< \beta \leq 1$. Besides, for any normalized eigenvector $\tilde{\boldsymbol{\varphi}}$ of $\mathbf{G}^\top\mathbf{G}$ with $\tilde{\lambda}_i$, there is a normalized eigenfunction $\varphi$ of $\Delta_M$ with eigenvalue $\lambda_i$ such that
\BEA
\left\| \boldsymbol{\varphi}-\tilde{\boldsymbol{\varphi}}\right\|_{L^{2}\left(\mu_{N}\right)}={\color{black}O\left(N^{-\frac{1}{4}}\right)}+O\left(N^{\frac{-2 \alpha+(n-d)}{4d}}\right),\label{error_eigen}
\EEA
with probability higher than $1-\left(\frac{2 m^{2}+4 m+24}{N}\right)$, where $ \boldsymbol{\varphi}=(\varphi(\mathbf{x}_1),...,\varphi(\mathbf{x}_N))^\top$ and $m$ is the geometric multiplicity of eigenvalue $\lambda_i$, as $N\to\infty$.
\end{lem}

The first term in \eqref{error_estimate_2} can be bounded above as follows,
\BEA
\left\| \varphi-\tilde{\varphi}\right\|_{L^2(M)}=\left\| \varphi-I_{\phi_s}\tilde{\boldsymbol{\varphi}}\right\|_{L^2(M)} \leq \left\| \varphi-I_{\phi_s}\boldsymbol{\varphi}\right\|_{L^2(M)} + \left\|I_{\phi_s}\boldsymbol{\varphi}-I_{\phi_s}\tilde{\boldsymbol{\varphi}}\right\|_{L^2(M)}, \label{errorbound3}
\EEA
where we can immediately use the interpolation error rate from Lemma~\ref{RKHS L2 Convergence} to bound the first term on the right-hand-side.

The following result gives a concentration bound for the second term,
\begin{lem}\label{lem:perturb}
Under the same assumptions as in Lemma~\ref{lem_eig_conv}, with probability higher than $1-\frac{1}{N}$,
\[
\Vert I_{\phi_s}\boldsymbol{\varphi}-I_{\phi_s}\tilde{\boldsymbol{\varphi}}\Vert_{L^2(M)} = O(N^{-\frac{\beta}{4}}) + O\left(N^{\frac{-2 \alpha+(n-d)}{4d}}\right),
\]
as $N\to\infty$ for some $0< \beta\leq 1$.
\end{lem}

\begin{proof}
Under the Assumption~\ref{weaklyunstable}, we have
\[
\| I_{\phi_s}(\boldsymbol{\varphi}-\tilde{\boldsymbol{\varphi}})\|^2_{L^\infty(M)} \leq C N^{\frac{1-\beta}{2}},
\]
for some $C>0$ that is independent of $N$ and some $0<\beta\leq 1$. By Hoeffding's concentration inequality, for any $\epsilon>0$,
\[
\mathbb{P}_X \left( \left| \Vert I_{\phi_s}(\boldsymbol{\varphi}-\tilde{\boldsymbol{\varphi}})\Vert _{L^2(M)}^2  - \| \boldsymbol{\varphi}-\tilde{\boldsymbol{\varphi}}\|^2_{L^2(\mu_N)}\right| >\epsilon \right)  \leq 2 \textup{exp}\left( \frac{-2\epsilon^2N}{CN^{1-\beta}} \right),
\]
where we have used the interpolation condition, $\left(I_{\phi_s}(\boldsymbol{\varphi}-\tilde{\boldsymbol{\varphi}})\right) |_X = \boldsymbol{\varphi}-\tilde{\boldsymbol{\varphi}}$. Balancing $\frac{1}{N}$ with the probability on the right-hand-side and solving for the $\epsilon$, we attain with probability higher than $1-\frac{1}{N}$,
\[
\Vert I_{\phi_s}(\boldsymbol{\varphi}-\tilde{\boldsymbol{\varphi}})\Vert_{L^2(M)}^2  \leq   \| \boldsymbol{\varphi}-\tilde{\boldsymbol{\varphi}}\|^2_{L^2(\mu_N)} + O(N^{-\frac{\beta}{2}}),
\]
where the big-oh involves a $\log(N)$ factor that we ignored. Together with \eqref{error_eigen}, the proof is complete.
\end{proof}

Combining Lemmas~\ref{RKHS L2  Convergence} and \ref{lem:perturb}, we can summarize an upper bound for \eqref{errorbound3} as follows.
\begin{lem}\label{lem_stability}Let the assumptions in Lemmas~\ref{RKHS L2  Convergence} and \ref{lem_eig_conv} be valid. Suppose that data points $\mathbf{x}_i \in X$ are i.i.d. samples drawn uniformly from $M$. Then, with probability higher than $1-\left(\frac{2 m^{2}+4 m+25}{N}\right)$,
\[
\left\| \varphi-\tilde{\varphi}\right\|_{L^2(M)} = O\left(N^{-\frac{\beta}{4}}\right)+O\left(N^{\frac{-2 \alpha+(n-d)}{4d}}\right),
\]
as $N\to\infty$, for some $0<\beta\leq 1$.
\end{lem}

Now, it only remains to estimate the second term in {\color{black}the} right hand side of (\ref{error_estimate_2}). Since the RKHS ${\cal V}$ is embedded in $L^2(M)$, it is clear that $\tilde{\varphi}:= I_{\phi_s}\tilde{\boldsymbol{\varphi}} \in {\cal V}$ is also in $L^2(M)$.
Now, we need to characterize the error between the coefficients $\textup{U}=(u_1,...,u_k)^\top$ and  $\tilde{\mathbf{U}}=(\tilde{u}_1,\ldots,\tilde{u}_K)^\top$ that solve the $K\times K$ linear systems $\textup{A}\textup{U} = \textup{b}$ and $\mathbf{\tilde{A}} \mathbf{\tilde{U}} \mathbf{= \tilde{b}}$, respectively. We recall again that $\textup{A}$ and $\textup{b}$ are defined in \eqref{Ab}, whereas $\mathbf{\tilde{A}}$ and $\mathbf{\tilde{b}}$ are defined in \eqref{Atilde2}-\eqref{btilde2}. {\color{black} The analysis below is to account for the effect of discretization errors (in parallel to Strang's Lemma in FEM literature) induced by the Monte-Carlo discretization and approximate basis functions.} We will proceed by employing the following lemma that is a well known result from perturbation theory (written in our notation).

\comment{
By definition, $\{\hat{u}^k\}$ is the solution of $AU=b$ with
\BEA
A_{jk} &:=& \left(c \varphi_k+\kappa\Delta_M \varphi_k-\langle  \textup{grad}_g  \kappa,  \textup{grad}_g  \varphi_k\rangle_g,\varphi_j\right)_{L^2}\label{matrix-A},\\
b_j &:=& (f,\varphi_j)_{L^2},
\EEA
and $\tilde{\hat{u}}^k$ is the solution of linear systems $\tilde{\mathbf{A}}\tilde{U}=\tilde{\mathbf{b}}$ in (\ref{Atilde2})-(\ref{btilde2})    where
\BEA
\mathbf{\tilde{A}}_{jk}&:=& \frac{1}{N}\sum_{i=1}^N\left((c+\kappa(\textbf{x}_i)\Delta_X^{RBF})\tilde{\boldsymbol{\varphi}}_k(\textbf{x}_i)-\sum_{\ell=1}^n(\mathbf{G}_\ell \boldsymbol{\kappa})_i(\mathbf{G}_\ell\tilde{\boldsymbol{\varphi}}_k)_i\right)\tilde{\boldsymbol{\varphi}}_j(\textbf{x}_i) \\
\mathbf{\tilde{b}}_j &:=&\frac{1}{N} \mathbf{f}^\top \boldsymbol{\tilde{\varphi}}_j.
\EEA
Then
$$
\sum_{k=1}^K\Vert \tilde{\varphi_k}\Vert|\hat{u}^k-\hat{\tilde{u}}^k|\leq (\sum_{k=1}^K\Vert \tilde{\varphi_k}\Vert|^2)^{1/2}\Vert U-\tilde{U}\Vert_2
$$
}

\begin{lem}\label{perturbation}(Theorem 2.7.2 of \cite{golub1996matrix}) Let $\textup{b}$ be a nonzero vector in $\mathbb{R}^K$ and $\textup{A}\in\mathbb{R}^{K\times K}$ be a matrix with full rank. Suppose $\textup{U},\tilde{\mathbf{U}} \in \mathbb{R}^K$ satisfy $\textup{A} \textup{U} = \textup{b}$ and $\mathbf{\tilde{A}} \mathbf{\tilde{U}} \mathbf{= \tilde{b}}$. Denote the condition number of $\textup{A}$ by $\textup{cond}(\textup{A})$. If
$$
\varepsilon:=\max\left\{\frac{\Vert \textup{A}-\tilde{\mathbf{A}}\Vert_2}{\Vert \textup{A}\Vert_2}, \frac{\Vert \textup{b}-\tilde{\mathbf{b}}\Vert_2}{\Vert \textup{b}\Vert_2}   \right\}<\frac{1}{\textup{cond}(\textup{A})},
$$
then
$$
\frac{\Vert \textup{U}-\tilde{\mathbf{U}}\Vert_2}{\Vert \textup{U}\Vert_2}\leq \frac{2\varepsilon}{1-\varepsilon \textup{cond}(\textup{A})} \textup{cond}(\textup{A}).
$$
\end{lem}

In order to employ this result, we will characterize the size of perturbations to satisfy the hypothesis above which guarantees that the perturbed matrix $\tilde{\mathbf{A}}$ to be non-singular. For the following error bounds, we define a restriction operator $R_Nf = \mathbf{f} = (f(\mathbf{x}_1),\ldots,f( \mathbf{x}_N))^\top$ for any continuous function in $C(\mathbb{R}^n)$.

\begin{lem}
Let $f, h\in C^\infty(M)$ and $X=\{\mathbf{x}_1,\ldots, \mathbf{x}_N\}\subset M$ be a set of i.i.d. uniform samples and let the Assumption~\ref{weaklyunstable} be valid. Then, with probability higher than $1-\frac{3+2n}{N}$,
\BEA
\| R_N\Delta_M f - \Delta^{\mathrm{RBF}}_X \mathbf{f} \|_{L^2(\mu_N)} = O(N^{-\frac{\beta}{4}}) + O\left(N^{\frac{-2 \alpha+(n-d)}{4d}}\right), \label{errordeltaRBF}
\EEA
for some $0<\beta\leq 1$. Also, with probability higher than $1 - \frac{3}{N}$,
\BEA
\Vert R_N \big \langle \textup{grad}_g f, \textup{grad}_g h\big\rangle_g- R_N\langle \textup{grad}_gI_{\phi_s}\mathbf{f}, \textup{grad}_g I_{\phi_s}\mathbf{h}\rangle \Vert_{{L^2}(\mu_N)}  =O(N^{-\frac{\beta}{4}}) + O\left(N^{\frac{-2 \alpha+(n-d)}{4d}}\right).\label{error_innerproduct}
\EEA
\end{lem}

\begin{proof}
The bound in \eqref{errordeltaRBF} is immediate by Hoeffding's concentration inequality ( Lemma~4.9 in \cite{harlim2022rbf}) and Proposition 4.2 in \cite{harlim2022rbf}. To determine the probability, we clarify that $\delta_f = N^{-1}$ and $\delta_u = nN^{-1}$ for any $u\in \mathfrak{X}(M)$ that is represented by $n$-component vector valued function. Since the Hoeffding's requires another probability $1/N$, in total, we have $1 - 2\delta_{f} - 2 \delta_{ \textup{grad}_g I_{\phi_s} f }- \frac{1}{N}  = 1 - \frac{3+2n}{N}$. The bound in \eqref{error_innerproduct} is an immediate consequence of Hoeffding's concentration inequality and applying Lemma 4.3 as in the proof of Lemma 4.4 in \cite{harlim2022rbf}.
\end{proof}

For reader's convenience, we will also use the following Bernstein's matrix inequality.
\begin{lem}[Corollary 6.1.2 of \cite{tropp2015introduction}.]\label{Bernstein} Let $\mathbf{S}_i$ be i.i.d. random matrix of size $K\times K$. Suppose that,
\[\| \mathbf{S}_i - \mathbb{E}[\mathbf{S}_i] \|_2\leq L, \quad \forall i.\]
Denote $\mathbf{Z} := \sum_{i=1}^N\mathbf{S}_i $ and
\[
v(\mathbf{Z}) = \max \left\{ \left\|\mathbb{E}\left[(\mathbf{Z}-\mathbb{E}[\mathbf{Z}])(\mathbf{Z}-\mathbb{E}[\mathbf{Z}])^\top\right] \right\|_2,\left\|\mathbb{E}\left[(\mathbf{Z}-\mathbb{E}[\mathbf{Z}])^\top (\mathbf{Z}-\mathbb{E}[\mathbf{Z}])\right] \right\|_2 \right\}.
\]
Then for any $t>0$,
\[
\mathbb{P}\left(\| \mathbf{Z} - \mathbb{E}[\mathbf{Z}]\|_2\geq t \right) \leq 2K  \max\left\{\exp\left(-\frac{3t^2}{8v(\mathbf{Z})}\right), \exp\left(-\frac{3t}{8L}\right)\right\}.
\]
\end{lem}

Then we can estimate the second term in right hand side of (\ref{error_estimate_2}) as follows:
\begin{lem}\label{lemma_error3} Let $\textup{A}$ be defined as in (\ref{Ab}), nonsingular, with smallest singular value satisfying,
\BEA
\frac{1}{2}\sigma_{\textup{min}} (\textup{A}) > C_1  K{N}^{-\frac{\beta}{4}}+C_2 K N^{\frac{-2 \alpha+(n-d)}{4d}},\label{Anonsingular_assumption}
\EEA
as $N\to \infty$, for all $C_1, C_2>0$ that are independent of $K$ and $N$, any $\alpha>n/2$, and some $0<\beta\leq 1$. Let $u^K$ be the solution of (\ref{weak-form-finite}) and $\tilde{u}^K$ be the solution of Spectral-RBF as in (\ref{spectral-RBF-soln}). Let $\textup{U}$ and $\tilde{\mathbf{U}}$ be the coefficients of $u^K$ and $\tilde{u}^K$ respectively, and the coefficients $c \in C(M)$, $\kappa\in C^1(M)$, and $f\in L^2(M)$. Let the assumptions in Lemma~\ref{lem_stability} to be valid. Then, with probability higher than $1-\frac{8 m^{2}+16 m+2n + 103}{N}$,
\BEA
\Vert \textup{U}-\tilde{\mathbf{U}} \Vert_2\leq \frac{\Vert \textup{U}\Vert_2}{\sigma_{\textup{min}} (\textup{A}) } \left(C_1  K{N}^{-\frac{\beta}{4}}+C_2 KN^{\frac{-2 \alpha+(n-d)}{4 d}}\right).\notag
\EEA
\end{lem}

\begin{proof}
\comment{
Define
\BEA
\mathbf{A}_{jk}&:= &\frac{1}{N}\sum_{i=1}^N\left(\kappa(\textbf{x}_i)\Delta_M\varphi_k(\textbf{x}_i)-\big\langle \textup{grad}_g \kappa(\textbf{x}_i), \textup{grad}_g \varphi_k(\textbf{x}_i)\big\rangle_g+c \varphi_k(\textbf{x}_i)\right)\varphi_j(\textbf{x}_i),\\
\mathbf{b}_j &:=& \frac{1}{N}\mathbf{f}^\top \boldsymbol{\varphi}_j.
\EEA
\BEA
\Vert \textup{A}-\tilde{\mathbf{A}}\Vert_2\leq \Vert \textup{A}-\mathbf{A}\Vert_2+\Vert \mathbf{A}-\tilde{\mathbf{A}}\Vert_2.
\label{error}
\EEA

First term in (\ref{error_A}): Notice that for $1\leq j,k\leq K$, $A_{jk}=\mathbb{E}[\mathbf{A}_{jk}]$. Then by using Hoeffding's inequality,
$$
P(|A_{jk}-\mathbf{A}_{jk}|\geq \varepsilon)\leq 2 \text{exp}\left(\frac{-2\varepsilon^2N}{D^2}\right)
$$
where $D$ is a constant depends on $\varphi,\kappa,c$. Then we have
$$
|A_{jk}-\mathbf{A}_{jk}|=O(\frac{1}{\sqrt{N}})
$$
with probability higher than $1-\frac{2}{N}$ where we chose $\epsilon=\sqrt{\frac{\log (N)}{N}}$ and ignored the log factor. Then
$$
\Vert A-\mathbf{A}\Vert_2\leq \sqrt{K} \Vert A-\mathbf{A}\Vert_1\leq K^{3/2}O(\frac{1}{\sqrt{N}}).
$$}

We will bound each of the right hand terms below,
\BEA
\Vert \textup{A}-\tilde{\mathbf{A}}\Vert_2\leq \Vert \textup{A}-\mathbf{A}\Vert_2+\Vert \mathbf{A}-\tilde{\mathbf{A}}\Vert_2
\label{error_A}
\EEA
where $\mathbf{A}$ is defined in (\ref{A}) and $\tilde{\mathbf{A}}$ is defined in (\ref{Atilde2}). To bound the first term in \eqref{error_A}, we will use the Bernstein matrix inequality from Lemma~\ref{Bernstein}. In our case, $\mathbf{S}_i$ has components,
\[
[\mathbf{S}_i]_{jk} = \frac{1}{N} \left[\left(\kappa(\textbf{x}_i)\Delta_M\varphi_k(\textbf{x}_i)-\big\langle \textup{grad}_g \kappa(\textbf{x}_i), \textup{grad}_g \varphi_k(\textbf{x}_i)\big\rangle_g+c \varphi_k(\textbf{x}_i)\right)\varphi_j(\textbf{x}_i) \right]
\]
and $[\mathbb{E}[\mathbf{S}_i]]_{jk} = \frac{1}{N} \textup{A}_{jk}$. This also means $\mathbf{Z} = \mathbf{A}$
and $\mathbb{E}[\mathbf{Z}] = \textup{A}$. If $\kappa \in C^1(M), c \in C(M)$, we can uniformly bound each components of $\mathbf{S}_i$, that is,
\[
\| \mathbf{S}_i -\mathbb{E}[\mathbf{S}_i] \|_2 \leq \|\mathbf{S}_i \|_2+ \|\mathbb{E}[\mathbf{S}_i] \|_2 \leq \|\mathbf{S}_i \|_2 + \mathbb{E}[ \|\mathbf{S}_i \|_2]\leq \sqrt{\|\mathbf{S}_i \|_1\|\mathbf{S}_i \|_\infty} + \mathbb{E}[ \sqrt{\|\mathbf{S}_i \|_1\|\mathbf{S}_i \|_\infty}]  \leq C\frac{K}{N}:= L,
\]
for some $C>0$ that depends on $\kappa, c, \varphi_k$. In the second inequality above, we have used the Jensen's inequality. In the third inequality, we used the matrix norm inequality $\| \textup{A}\|^2_2 \leq \|\textup{A}\|_1 \|\textup{A}\|_\infty,$ where $\|\textup{A}\|_1 = \max_{1\leq j\leq K} \sum_{i=1}^K |\textup{A}_{ij}|$ and $\|\textup{A}\|_\infty = \max_{1\leq i\leq K} \sum_{j=1}^K |\textup{A}_{ij}|$. We also note that,
\[
\mathbb{E}\left[(\mathbf{S}_i-\mathbb{E}[\mathbf{S}_i])(\mathbf{S}_i-\mathbb{E}[\mathbf{S}_i])^\top\right] = \mathbb{E}\left[\mathbf{S}_i\mathbf{S}_i^\top\right] - \frac{1}{N^2}\textup{AA}^\top  \preceq \mathbb{E}\left[\mathbf{S}_i\mathbf{S}_i^\top\right] ,
\]
where $\textup{A} \preceq \textup{B}$ means $\textup{B}-\textup{A}$ is positive semi-definite. Then,
\[
\left\|\mathbb{E}\left[(\mathbf{Z}-\mathbb{E}[\mathbf{Z}])(\mathbf{Z}-\mathbb{E}[\mathbf{Z}])^\top\right] \right\|_2 = \big\|\sum_{i=1}^N\mathbb{E}\left[(\mathbf{S}_i-\mathbb{E}[\mathbf{S}_i])(\mathbf{S}_i-\mathbb{E}[\mathbf{S}_i])^\top\right] \big\|_2 \leq N \| \mathbb{E}[\mathbf{S}_i\mathbf{S}_i^\top] \|_2\leq N \mathbb{E}[\|\mathbf{S}_i\mathbf{S}_i^\top\|_2] \leq C^2 \frac{K^{2}}{N},
\]
where we have used the fact that $\|\mathbf{S}_i\mathbf{S}_i^\top\|_2 \leq \sqrt{ \|\mathbf{S}_i\mathbf{S}_i^\top\|_1 \|\mathbf{S}_i\mathbf{S}_i^\top\|_\infty} \leq C^2 K^{2}N^{-2}$. Using the same argument, one can deduce the same error rate for $\left\|\mathbb{E}\left[(\mathbf{Z}-\mathbb{E}[\mathbf{Z}])^\top(\mathbf{Z}-\mathbb{E}[\mathbf{Z}])\right] \right\|_2$ so $v(Z) \leq C^2 K^{2} {N}^{-1}$. With these bounds and Lemma~\ref{Bernstein}, we conclude that with probability higher than $1-\frac{1}{N}$,
\[
\left\| \mathbf{Z} - \mathbb{E}[\mathbf{Z}] \right\|_2  \leq \begin{cases} C\sqrt{\frac{8}{3}\log(2KN)}\frac{K}{N^{1/2}}, & \text{ if } tL \leq v(\mathbf{Z}),  \\
\frac{8C}{3}\log(2KN) \frac{K}{N}, & \text{ if } tL \geq v(\mathbf{Z}).
\end{cases}
\]
Taking the worse rate among the two, we have
\BEA
\|\mathbf{A} -\textup{A} \|_2  = \left\| \mathbf{Z} - \mathbb{E}[\mathbf{Z}] \right\|_2  \leq C\sqrt{\frac{8}{3}\log(2KN)}\frac{K}{N^{1/2}}. \label{boundA1}
\EEA

Now we bound the second term in (\ref{error_A}). The error in $|\mathbf{A}_{jk}-\tilde{\mathbf{A}}_{jk}|$ has three parts:
$$
\begin{aligned}
D_{jk}&:=\frac{1}{N}\left|\sum_{i=1}^N\big[\kappa(\textbf{x}_i)\Delta_M\varphi_k(\textbf{x}_i)\varphi_j(\textbf{x}_i)- \kappa(\textbf{x}_i)\Delta_X^{\mathrm{RBF}}\tilde{\boldsymbol{\varphi}}_k(\textbf{x}_i)\tilde{\boldsymbol{\varphi}}_j(\textbf{x}_i)\big]\right|,\\
E_{jk}&:=\frac{1}{N}\left|\sum_{i=1}^N\big[\big\langle \textup{grad}_g \kappa(\textbf{x}_i), \textup{grad}_g \varphi_k(\textbf{x}_i)\big\rangle_g\varphi_j(\textbf{x}_i) -\sum_{\ell=1}^n(\mathbf{G}_\ell \boldsymbol{\kappa})_i(\mathbf{G}_\ell\tilde{\boldsymbol{\varphi}}_k)_i\tilde{\boldsymbol{\varphi}}_j(\textbf{x}_i)\big]\right|,\\
F_{jk}&:=\frac{1}{N}\left|\sum_{i=1}^N\big[c(\mathbf{x}_i) \varphi_k(\textbf{x}_i)\varphi_j(\textbf{x}_i)-c(\mathbf{x}_i)\tilde{\boldsymbol{\varphi}}_k(\textbf{x}_i)\tilde{\boldsymbol{\varphi}}_j(\textbf{x}_i)  \big]\right|.
\end{aligned}
$$
First, we note that,
\BEA
D_{jk}&\leq& \Vert\boldsymbol{\kappa}\Vert_{\infty}\left(\Vert R_N \Delta_M\varphi_k-\Delta^{\mathrm{RBF}}_X\boldsymbol{\varphi}_k\Vert_{{L^2}(\mu_N)}
+\Vert\Delta^{\mathrm{RBF}}_X\boldsymbol{\varphi}_k
-\Delta^{\mathrm{RBF}}_X\tilde{\boldsymbol{\varphi}}_k\Vert_{{L^2}(\mu_N)}   \right)\Vert\boldsymbol{\varphi}_j\Vert_{{L^2}(\mu_N)} \notag\\
&& +\Vert\boldsymbol{\kappa}\Vert_{\infty}\Vert\Delta^{\mathrm{RBF}}_X\tilde{\boldsymbol{\varphi}}_k\Vert_{{L^2}(\mu_N)}\Vert\boldsymbol{\varphi}_j-\tilde{\boldsymbol{\varphi}}_j\Vert_{{L^2}(\mu_N)}. \notag
\EEA
The second term can be bounded above by the Hoeffding's inequality. Particularly, with probability higher than $1-\frac{2 m^{2}+4 m+25}{N}$,
\BEA
\Vert\Delta^{\mathrm{RBF}}_X\boldsymbol{\varphi}_k-\Delta^{\mathrm{RBF}}_X\tilde{\boldsymbol{\varphi}}_k\Vert_{{L^2}(\mu_N)}  = 
\sum_{\ell=1}^n \|\mathcal{G}_{\ell} I_{\phi_s} \mathcal{G}_{\ell}I_{\phi_s} (\boldsymbol{\varphi}_k-\tilde{\boldsymbol{\varphi}}_k) \|_{L^2(M)}  + O\Big(N^{-\frac{\beta}{4}}\Big) \notag =  O\left(N^{\frac{-2 \alpha+(n-d)}{4d}}\right) + O\Big(N^{-\frac{\beta}{4}}\Big)\notag
\EEA
where the last equality follows the same argument as in Lemma~\ref{lem:perturb} under the stability Assumption~\ref{weaklyunstable} over the Sobolev $W^{2,\infty}(M)$-norm. Together with \eqref{errordeltaRBF} and \eqref{error_eigen} and the boundedness of $\boldsymbol{\kappa}, \boldsymbol{\varphi}_j$ and $\Delta^{\mathrm{RBF}}_X\tilde{\boldsymbol{\varphi}}_k$, with probability higher than $1-\frac{2 m^{2}+4 m+2n+ 28}{N}$,
\[
D_{jk} = O\Big(N^{-\frac{\beta}{4}}\Big)  + O\left(N^{\frac{-2 \alpha+(n-d)}{4d}}\right).
\]

By using the interpolation error in \eqref{error_innerproduct}, the error bound on eigenfunctions (\ref{error_eigen}) and the boundedness assumptions, with probability higher than $1-\frac{2 m^{2}+4 m+27}{N}$,
\BEA
E_{jk}&\leq& \Vert R_N \big \langle \textup{grad}_g \kappa, \textup{grad}_g \varphi_k\big\rangle_g- R_N\langle \textup{grad}_gI_{\phi_s}\boldsymbol{\kappa}, \textup{grad}_g I_{\phi_s}\boldsymbol{\varphi}_k\rangle \Vert_{{L^2}(\mu_N)} \Vert\boldsymbol{\varphi}_j\Vert_{{L^2}(\mu_N)} \notag \\
&& +\Vert R_N\langle \textup{grad}_gI_{\phi_s}\boldsymbol{\kappa}, \textup{grad}_g I_{\phi_s}\boldsymbol{\varphi}_k\rangle- R_N\langle \textup{grad}_gI_{\phi_s}\boldsymbol{\kappa}, \textup{grad}_g I_{\phi_s}\tilde{\boldsymbol{\varphi}}_k\rangle\rangle\Vert_{{L^2}(\mu_N)}\Vert\boldsymbol{\varphi}_j\Vert_{{L^2}(\mu_N)} \notag\\
&&+\Vert R_N\langle\textup{grad}_gI_{\phi_s}\boldsymbol{\kappa}, \textup{grad}_g I_{\phi_s}\tilde{\boldsymbol{\varphi}}_k\rangle\Vert_{{L^2}(\mu_N)}\Vert\boldsymbol{\varphi}_j-\tilde{\boldsymbol{\varphi}}_j\Vert_{{L^2}(\mu_N)}\notag \\
&=& O\Big(N^{-\frac{\beta}{4}}\Big)  + O\left(N^{\frac{-2 \alpha+(n-d)}{4d}}\right).\notag
\EEA
And in the last term, with probability higher than $1-\frac{2 m^{2}+4 m+24}{N}$,
$$
\begin{aligned}
F_{jk}&\leq \Vert \boldsymbol{c}\Vert_\infty\left(\Vert\boldsymbol{\varphi}_k-\tilde{\boldsymbol{\varphi}}_k\Vert_{{L^2}(\mu_N)}\Vert\boldsymbol{\varphi}_j\Vert_{{L^2}(\mu_N)}+\Vert\tilde{\boldsymbol{\varphi}}_k\Vert_{L^2(\mu_N)}\Vert\boldsymbol{\varphi}_j-\tilde{\boldsymbol{\varphi}}_j\Vert_{{L^2}(\mu_N)} \right)\\
&\leq  \Vert\boldsymbol{c}\Vert_\infty\left[O(N^{-\frac{\beta}{4}})+O\left(N^{\frac{-2 \alpha+(n-d)}{4d}}\right)\right].
\end{aligned}
$$
Using \eqref{boundA1}and the bound on each component of $D, E$ and $F$,
\BEA
\Vert \textup{A}-\tilde{\mathbf{A}}\Vert_2 &\leq& C\sqrt{\frac{8}{3}\log(2KN)}\frac{K}{N^{1/2}} +  \sqrt{\Vert \mathbf{A}-\tilde{\mathbf{A}}\Vert_1\Vert \mathbf{A}-\tilde{\mathbf{A}}\Vert_\infty} \notag \\ &=&  O(\log(KN)KN^{-\frac{1}{2}})+O(K{N}^{-\frac{\beta}{4}})+O\left(K N^{\frac{-2 \alpha+(n-d)}{4d}}\right).\label{ErrorA2}
\EEA
Using the same (but simpler) argument as above, one can show that with probability higher than $1-\frac{2 m^{2}+4 m+24}{N}$
$$
\Vert \textup{b}-\tilde{\mathbf{b}} \Vert_2 = O(K^{\frac{1}{2}} N^{-\frac{\beta}{4}})+O\left( K^{\frac{1}{2}} N^{\frac{-2 \alpha+(n-d)}{4d}}\right),
$$
so the perturbation error induced by $\tilde{\mathbf{A}}$ dominates.
By the assumption in \eqref{Anonsingular_assumption} for any $C_{1},C_{2}>0$%
, including the constants in the upper bounds in \eqref{ErrorA2}, one has%
\[
\frac{1}{2}\sigma _{\textup{min}}(\textup{A})>C_{1}K{N}^{-\frac{\beta}{4}%
}+C_{2}KN^{\frac{-2\alpha +(n-d)}{4d}}\geq \Vert \textup{A}-\tilde{\mathbf{A%
}}\Vert _{2},
\]%
as $N\rightarrow \infty $. Here, the $\log (KN)$-term in \eqref{ErrorA2} is
assumed to be small such that the first error term in \eqref{ErrorA2} is
absorbed by the first error term, $C_{1}K{N}^{-\frac{\beta}{4}}$. Let $%
\varepsilon $ be defined in Lemma \ref{perturbation}, and then
the condition therein holds true:%
\[
\varepsilon \textup{cond}(\textup{A})\leq \Vert \textup{A}-\tilde{\mathbf{%
A}}\Vert _{2}\Vert \textup{A}^{-1}\Vert _{2}=\frac{\Vert \textup{A}-\tilde{\mathbf{A}%
}\Vert _{2}}{\sigma _{\textup{min}}(\textup{A})}<\frac{1}{2}<1.
\]%
By Lemma \ref{perturbation}, with
probability higher than $1-\frac{8m^{2}+16m+2n+103}{N}$ (obtained by taking
product of all of the probabilities above),
\[
\Vert \textup{U}-\tilde{\mathbf{U}}\Vert _{2}\leq \frac{2\Vert \textup{U}%
\Vert _{2}\varepsilon \textup{cond}(\textup{A})}{1-\varepsilon \textup{%
cond}(\textup{A})}\leq \frac{4\Vert \textup{U}\Vert _{2}}{\sigma _{%
\textup{min}}(\textup{A})}\left( C_{1}K{N}^{-\frac{\beta}{4}}+C_{2}KN^{\frac{%
-2\alpha +(n-d)}{4d}}\right) .
\]


\end{proof}

We conclude this section with the following error bound in $L^2$ norm.

\begin{thm}\label{error-estimate} Let $u\in H^2(M)$ be a solution of (\ref{weak-form}) and $\tilde{u}^K$ be the solution of Spectral-RBF as in (\ref{spectral-RBF-soln}). Let the assumptions in Lemmas~\ref{lemma_error1}, \ref{lem_stability}, and \ref{lemma_error3} be valid.
Then for any $K>0$, with high probability,
\BEA
\Vert u-\tilde{u}^K\Vert_{L^2(M)}\leq  C\left(\frac{2}{1+\lambda_K}\right)\Vert u\Vert_{H^2(M)} + O\left(K^{\frac{3}{2}}N^{-\frac{\beta}{4}}\right)+O\left(K^{\frac{3}{2}}N^{\frac{-2 \alpha+(n-d)}{4 d}}\right),\label{allbound}
\EEA
as $N\to\infty$.
\end{thm}

\begin{proof}
We first recall from previous definitions that $u^K$ and $\tilde{u}^K$ could be written as $u^K=\sum_{k=1}^Ku_k\varphi_k$ and  $\tilde{u}^K=\sum_{k=1}^K\tilde{u}_k\tilde{\varphi}_k$, where $\{(\varphi_k,\lambda_k)\}_{k=1}^K$ are eigen-pairs of $\Delta_M$ with increasing ordered $\lambda_k$ and $\tilde{\varphi}_k=I_{\phi_s}\tilde{\boldsymbol{\varphi}}_k$ are their numerical approximations. We deduce the error bound in \eqref{allbound} by inserting the three error bounds Lemmas~\ref{lemma_error1}, \ref{lem_stability}, and \ref{lemma_error3} into \eqref{error_estimate} and \eqref{error_estimate_2},
\BEA
\Vert u-\tilde{u}^K\Vert_{L^2(M)}&\leq& \Vert u-u^{K}\Vert_{L^2(M)}+\Vert u^K-\tilde{u}^{K}\Vert_{L^2(M)}\notag \\
&\leq& \Vert u-u^{K}\Vert_{L^2(M)}+ \sum_{k=1}^K|u_k|\Vert \varphi_k-\tilde{\varphi}_k\Vert_{L^2(M)}+\sum_{k=1}^K\Vert \tilde{\varphi}_k\Vert_{L^2(M)}|u_k-\tilde{u}_k| \notag \\
&\leq& \Vert u-u^{K}\Vert_{L^2(M)}+  \|\textup{U}\|_\infty \sum_{k=1}^K\Vert \varphi_k-\tilde{\varphi}_k\Vert_{L^2(M)}+\left(1+ O\left(N^{-\frac{\beta}{4}}\right)+O\left(N^{\frac{-2 \alpha+(n-d)}{4d}}\right)\right)\sum_{k=1}^K |u_k-\tilde{u}_k| \notag\\
&\leq& \Vert u-u^{K}\Vert_{L^2(M)}+  \|\textup{U}\|_\infty \sum_{k=1}^K\Vert \varphi_k-\tilde{\varphi}_k\Vert_{L^2(M)}+\left(1+ O\left(N^{-\frac{\beta}{4}}\right)+O\left(N^{\frac{-2 \alpha+(n-d)}{4d}}\right)\right) \sqrt{K} \| \textup{U} - \tilde{\mathbf{U}} \|_2, \notag
\EEA
where we have used the Cauchy–Schwarz inequality in the last inequality.
In the above, we have used the estimate
\[
\Vert \tilde{\varphi}_k\Vert_{L^2(M)} \leq \Vert {\varphi}_k\Vert_{L^2(M)} + \Vert \varphi_k-\tilde{\varphi}_k\Vert_{L^2(M)} \leq 1 +  O\left(N^{-\frac{\beta}{4}}\right)+O\left(N^{\frac{-2 \alpha+(n-d)}{4d}}\right),
\]
that used the orthonormality of $\varphi_k$ in $L^2(M)$ and error bound in Lemma \ref{lem_stability}.
\end{proof}

\begin{rem} There are two important aspects we should mention here.
\begin{enumerate}
\item For $d$-dimensional closed manifolds, the spectrum of Laplace-Beltrami scales as $\lambda_K\sim K^{\frac{2}{d}}$ \cite{colbois2013eigenvalues}. Comparing the last three error rates in \eqref{allbound}, one can always choose $\alpha>n/2$ such that the dominant error is of order-$K^{\frac{3}{2}}N^{-\frac{\beta}{4}}$. Choosing $K \sim N^{\frac{\beta d}{6d+8}}$ such that  $K^{\frac{3}{2}}N^{-\frac{\beta}{4}}=K^{-\frac{2}{d}}$, we achieve an error rate $K^{-\frac{2}{d}} \sim N^{-\frac{\beta}{4+3d}}$.
\item In the discussion above, we did not account for the error induced by the estimator $\hat{\mathbf{P}}$ discussed in Section~\ref{sec2.4}. Based on the convergence study in \cite{harlim2022rbf}, we will inherit an additional error rate of order-$N^{-2/d}$ from the second-order local SVD method, which will be negligible compared to the error rates in \eqref{allbound} when $d\leq 8/\beta$.
\end{enumerate}
\end{rem}

\section{Numerical study}\label{section_numeric}

In this section, we will test our method proposed in Section \ref{sec3.1} which we referred as Spectral-RBF. We compare them with graph Laplacian based PDE solver and the direct RBF solver. We will organize the discussion as follows: In Section~\ref{sec4.1}, we report some implementation details and discuss the alternative approaches that we are comparing with. In Section~\ref{sec4.2}-\ref{section_bunny}, we report results on a 1D ellipse embedded in $\mathbb{R}^2$, 2D torus embedded in $\mathbb{R}^3$, and an unknown surface, the Stanford Bunny.

\subsection{Experimental design}\label{sec4.1}
In this section, we discuss several practical issues to be considered in the implementation. The first practical issue is in taking the inverse of the RBF interpolation matrix $\boldsymbol{\Phi}$ which is used in (\ref{diff-op}) to define the differential matrix, $\mathbf{G}_\ell$. As mentioned in Section~\ref{review-RBF}, we only apply inverse quadratic kernel to test the numerical performances. In general, a nearly flat (small shape parameter) basis function in RBF interpolation is preferred to obtain the best accuracy over a set of node points. However, the interpolation matrix $\boldsymbol{\Phi}$ defined in (\ref{RBF-inter}) becomes dense and close to singular as shape parameter $s$ decreases and as the training data point increases. This last fact can be quantified as follows.
\begin{prop}\label{prop:unstable}
Let $\Phi_s$ be the inverse quadratic kernel and the matrix $\boldsymbol{\Phi}$ be defined with components $\boldsymbol{\Phi}_{i,j}:=\Phi_s(\|\mathbf{x}_i-\mathbf{x}_j\|_2)$, where the nodes in $X:=\{\mathbf{x}_i\}_{i=1,\ldots, N}$ are uniformly i.i.d. samples on $M$, i.e., the sampling density defined with respect to the volume measure is constant, $q \propto 1$. Assume also that $M$ has a finite scalar Ricci curvature on every element of $X$. Then with probability higher than $1-\frac{2}{N}$, the smallest eigenvalue is bounded from below as,
\BEA
\hat{\lambda}_{\textup{min}}(\boldsymbol{\Phi})\geq C\exp(-\hat{C}(d)N^{2/d}). \label{lambda_lb}
\EEA
\end{prop}
We leave the proof in Appendix~\ref{AppA}. This negative result implies that the condition number grows exponentially in $N$, which yields an unstable interpolator, {\color{black}or, practically, an ill-conditioned linear problem \eqref{RBF-inter}.}

While there exist several stable algorithms \cite{fasshauer2009preconditioning,fornberg2011stable,larsson2013stable} to overcome this ill-conditioned problem, they could not be generalized easily to RBF interpolation over a general manifold. For example, in \cite{fornberg2008stable}, the authors proposed a method which is called RBF-QR to eliminate ill-conditioning in the case when the data points lie on the surface of a sphere. There,  the idea is to rewrite the RBF basis in terms of spherical harmonic basis such that they could span the same space. However, a complete orthonormal basis on a general manifold is not accessible. To overcome this issue, we implement a simple pseudo-inversion based on SVD to deal with the ill-conditioning problem \cite{chen2006some}. Given an $N\times N$ symmetric matrix $\boldsymbol{\Phi}$, the singular value decomposition is given as
$
\boldsymbol{\Phi}=\mathbf{V}\boldsymbol{\Sigma}\mathbf{V}^\top
$
where $\mathbf{V}$ is an orthogonal matrix and $\boldsymbol{\Sigma}$ is a diagonal matrix with singular values (or eigenvalues in this case) $\hat{\lambda}_1\geq\hat{\lambda}_2\geq\cdots\geq \hat{\lambda}_N\geq0$.  Then we can set all terms with $\hat{\lambda}_j\leq\tau$ to be $0$ for $j=J+1,\ldots, N$ in the diagonal matrix $\boldsymbol{\Sigma}$ for some tolerance parameter $\tau>0$. Denote the resulting matrix as $\boldsymbol{\Sigma}_J \in\mathbb{R}^{J\times J}$ and the matrix consists of the first $J$-columns of $\mathbf{V}$ as $\mathbf{V}_J$. Then the inverse of matrix $\boldsymbol{\Phi}$ can be approximated as,
$$
\boldsymbol{\Phi}^{-1}=\mathbf{V}\boldsymbol{\Sigma}^{-1}\mathbf{V}^\top\approx \mathbf{V}_J\boldsymbol{\Sigma}_J^{-1}\mathbf{V}_J^\top.
$$
In MATLAB, this could be easily implemented with the built-in pseudoinverse function ${\asciifamily pinv.m}$. In practice, we usually choose the threshold $\tau=10^{-6}$. We should point out that this SVD pseudo-inversion does not translate to an interpolation solution.  {\color{black} It is also common in statistical literature \cite{cutajar2016preconditioning} to overcome this problem with a ridge regression solution obtained by replacing $\boldsymbol{\Phi}$ by $\boldsymbol{\Phi}+\sigma \mathbf{I}$ where $\sigma$ is a small positive number and combining it with a preconditioning matrix. We will find this approach to be more robust than pseudo-inverse for the local radial basis interpolation.}

The second practical issue is in choosing the shape parameter such that a good accuracy is achieved. It is well known that, the condition number of the interpolation matrix $\boldsymbol{\Phi}$ grows almost exponentially as the shape parameter $s$ decays to zero \cite{fornberg2007runge}. On the other hand, accurate solutions tend to correspond to a small parameter value $s$ before reaching the regime where the error grows almost exponentially. Thus, a good balance between accuracy and stability is desired. In \cite{fasshauer2007choosing}, the authors apply the leave-one-out cross validation approach to choose the "optimal" shape parameter in RBF interpolation matrix and differential matrix. In \cite{shankar2015radial}, the shape parameter $s$ is chosen such that the condition number of the matrix $\boldsymbol{\Phi}$ is kept almost {\color{black}the} same as $N$ grows. In this paper, we tune the shape parameter by hand and fix the parameter for increasing $N$ to check the convergence rate in numerical examples.

The third issue is related to solving {\color{black}the} eigenvalue problem of $\Delta_X^{\mathrm{SRBF}}=\sum_{\ell=1}^n\mathbf{G}_{\ell}^\top \mathbf{G}_{\ell}$ in (\ref{sym_eigs}). Numerically, since the approximate matrix $\Delta_X^{\mathrm{SRBF}}$ are low rank, these matrices contain many irrelevant eigenvalues that are very close to zero, and thus, create a practical issue in estimating the nontrivial leading $K$ spectra, $0=\lambda_1<\lambda_2 \leq \lambda_3 \leq  \ldots \leq \lambda_K$. One way to overcome this issue is to use the rank of $\boldsymbol{\Phi}$ in \eqref{RBF-inter} as an upper bound for the rank of $\Delta_X^{\mathrm{SRBF}}$ in order to approximate the leading  nontrivial eigenvalues. Denoting $r={\asciifamily rank}(\boldsymbol{\Phi},\tau)$ based on the tolerance parameter used in the pseudo-inversion of $\boldsymbol{\Phi}$, we find the eigensolutions
corresponding to the $r$ largest-magnitude numerical eigenvalues. Setting  $\tau_{\mathrm{eig}}=10^{-4}$, we pick the $K$-smallest nontrivial eigenvalues with their magnitudes larger than this threshold $\tau_{\mathrm{eig}}$ out of the $r$ eigensolutions. Then we set the trivial eigenvalue to be zero and the corresponding eigenfunction to be a constant vector $\mathbf{1}$.

For diagnostic, {\color{black}we compare the proposed Spectral-RBF discussed in Section~\ref{sec3.1} with the following} PDE solvers:
\par \textbf{Direct-RBF:} This scheme is basically the RBF pointwise approximation as discussed in Section~\ref{sec2.2}. Using the differentiation matrices in RBF approximation of gradient and Laplacian operators on the training data set $X$, we obtain the $N\times N$ approximate discrete linear system,
$$
\begin{aligned}
f(\textbf{x}_i) &=-\textup{div}_g\left(\kappa \textup{grad}_g u(\textbf{x}_i)\right)+c(\textbf{x}_i)u(\textbf{x}_i)\\
&=-\sum_{\ell=1}^n\left(\mathbf{G}_{\ell} \boldsymbol{\kappa})_i (\mathbf{G}_{\ell}\mathbf{u}\right)_i-\kappa(\textbf{x}_i)\left(\sum_{\ell=1}^n(\mathbf{G}_{\ell} \mathbf{G}_{\ell} \mathbf{u})_i\right)+c(\textbf{x}_i) u(\textbf{x}_i),\quad\text{for }i=1,...,N.
\end{aligned}
$$
The computational cost is dominated by direct inversion of this linear problem, which is order $O(N^3)$. As we mentioned previously (see the last paragraph in Section~\ref{sec3.1}), this solver can be numerically expensive if one needs to solve the PDE many times for different choices of $\kappa$ since the construction depends on $\kappa$.
In some literature, {\color{black}the} above approach is also called the  RBF-pseudo-spectral approach \cite{fasshauer2007choosing}.

{\color{black}\par \textbf{Direct-RBF-FD:} In one of the examples below, we will test the RBF-FD approximation as proposed in \cite{shankar2015radial,lehto2017radial}. This approach is also a pointwise approximation since it is essentially applying Direct-RBF with local RBF interpolation. Since the approximate operator is a sparse matrix, the computational cost is of order $O(N\log N)$ operations. In our implementation, we employ a ridge regression solution $(\Phi+\sigma\mathbf{I})^{-1}$ with $\sigma=10^{-6}/N$. We provide a review of the RBF-FD approximation in Appendix~\ref{AppB}.

\par \textbf{Spectral-RBF-FD:} In one of the same examples below, we also employ the spectral-RBF with the RBF-FD approximation in realizing steps (4)-(6) in the procedure described in Section~\ref{sec3.1}. That is, we approximate differential operators $\mathbf{G}_\ell$ and $\Delta_X^{\mathrm{RBF}}$ using the RBF-FD method. Here, we still use eigenvectors obtained by the global RBF interpolation in step (3) since we cannot obtain any convergence solution in solving the eigenvalue problem in \eqref{sym_eigs} with the RBF-FD.
}

\par \textbf{Variable Bandwidth Diffusion Map (VBDM):} In \cite{gh2019,harlim2020kernel,jiang2020ghost},  they consider a kernel approximation to the differential operator $\mathcal{L}^\kappa :=-\text{div}_g(\kappa\textup{grad}_g \,)$ with a fixed-bandwidth Gaussian kernel. For more accurate estimation when the data is non-uniformly distributed, we extend this graph-Laplacian approximation using the variable bandwidth kernel \cite{bh:16vb}.

Following the pointwise estimation method in \cite{liang2021solving}, we construct our estimator by choosing the kernel bandwidth to be inversely proportional to the sampling density $q$. Since the sampling density is usually unknown, we
first need to approximate it. While there are many ways to estimate density, in our algorithm we employ kernel density estimation with the following kernel that is closely related to the cKNN \cite{berry2019consistent} and self-tuning kernel \cite{zelnik2004self},
\[
K_{\epsilon ,0 }(\mathbf{x},\mathbf{y})=\exp \left( -\frac{\left\Vert
\mathbf{x}-\mathbf{y}\right\Vert^2 }{2\epsilon \rho_0(\mathbf{x})  \rho_0(\mathbf{y}) }\right),
\]
where $\rho_0(\mathbf{x}) := \Big(\frac{1}{k_2-1}\sum_{j=2}^{k_2} \|\mathbf{x} - \mathbf{x}_j \|^2\Big)^{1/2}$
denotes the average distance of $\mathbf{x}$ to the first $k_2$-nearest neighbors $\{\mathbf{x}_j\}, j=2,\ldots k_2$ excluding itself. Using this kernel, the sampling
density $q(\mathbf{x})$\ is estimated by $Q (\mathbf{x})=\sum_{j=1}^{N}K_{\epsilon ,0
}(\mathbf{x},\mathbf{x}_{j})/\rho_0( \mathbf{x})^{d}$\ at given point
cloud data. For our purpose, let us give a quick overview of the estimation of the Laplace-Beltrami operator. In this case, one chooses the bandwidth function to be $\rho(\mathbf{x})= q(\mathbf{x})^{\beta} \approx Q(\mathbf{x})^{\beta}$, with $\beta =-1/2$. With this bandwidth function, we {\color{black} employ the DM algebraic steps to construct  $K_{\epsilon,\rho}= \exp\left(-\Vert\mathbf{x}-\mathbf{y}\Vert^2/(4\epsilon\rho(\mathbf{x})\rho(\mathbf{y}))\right) $} and define $Q_\rho (\mathbf{x})=\sum_{j=1}^{N}K_{\epsilon ,\rho
}(\mathbf{x},\mathbf{x}_{j})/\rho( \mathbf{x})^{d}$. Then, we remove the sampling bias by applying a right
normalization $K_{\epsilon ,\rho ,\alpha }(\mathbf{x}_{i},\mathbf{x}%
_{j})= \frac{K_{\epsilon ,\rho }(\mathbf{x}_{i},\mathbf{x}_{j})}{Q_{\rho }(\mathbf{x}%
_{i})^{\alpha }Q_{\rho }(\mathbf{x}_{j})^{\alpha }}$,
where $\alpha = -d/4+1/2$
. We refer to \cite{harlim:18,bh:16vb} for more
details about the VBDM estimator of weighted Laplacian with other choices of $\alpha$ and $\beta$. Define diagonal
matrices $\mathbf{Q}$ and $\mathbf{P}$ with entries $\mathbf{Q}_{ii}=Q_{\rho
}(\mathbf{x}_{i})$ and $\mathbf{P}_{ii}=\rho (\mathbf{x}_{i})$,
respectively, and also define the symmetric matrix $\mathbf{K}$ with entries
$\mathbf{K}_{ij}=K_{\epsilon ,\rho ,\alpha }(\mathbf{x}_{i},\mathbf{x}_{j})$%
. Next, one can obtain the variable bandwidth diffusion map (VBDM)
estimator, $\mathbf{L}_{\epsilon ,\rho }:=\mathbf{P}^{-2}(\mathbf{D}^{-1}%
\mathbf{K-I})/\epsilon $, where $\mathbf{I}$\ is an identity matrix, as a discrete estimator to the Laplace-Beltrami operator in high probability. For computational efficiency, we also set $k_1$ as the nearest neighbor parameter for constructing $\mathbf{K}_{ij}$ (or eventually $\mathbf{L}_{\epsilon ,\rho }$).
We will report the specific choices for $k_1, k_2$ in each example. As for the choice of parameter $\epsilon$, we simply use the automated tuning technique that is found to be robust for variable bandwidth kernels (see Section~5 of \cite{bh:16vb}).

\subsection{PDE on an Ellipse}\label{sec4.2}
We first consider solving the elliptic equation (\ref{eq1}) on an ellipse which is defined with the usual embedding function,
$$
\iota(\theta)=(\cos \theta, a \sin \theta)^{\top}, \qquad \theta\in[0,2\pi),
$$
and for arbitrary $a>0$. In this example we choose $a=2$. The induced Riemmannian metric is
$$
g_{\textbf{x}^{-1}(\theta)}(v, w)=v^{\top}(\sin^2\theta+a^2\cos^2\theta) w, \quad \forall v, w \in T_{\textbf{x}^{-1}(\theta)} M.
$$
\par In this numerical example, the true solution $u$ is set to be $u(\theta)=\sin^2\theta$, the diffusion coefficient is defined to be $\kappa(\theta)=1.1+\sin ^{2} \theta$ and the constant $c=1$. The forcing term $f$ on the right side of  (\ref{eq1}) is computed explicitly by
$$
\begin{aligned}
f:=-\operatorname{div}_g(\kappa \textup{grad}_g u)+c u&=-\frac{1}{\sqrt{|g|}}\frac{\partial}{\partial \theta}\left(\kappa \sqrt{|g|} g^{11} \frac{\partial u}{\partial \theta}\right)+c u.\\
\end{aligned}
$$

Numerically, the grid points $\{\theta_i\}$ are randomly sampled from the uniform distribution on $[0,2\pi)$. To show the robustness of the results, we will show the result of 15 trials, each trial corresponds to an independent randomly sampled training data $X$ of size $N$. To illustrate the convergence of solutions over the number of points $N=[400,800,1600,3200,6400]$, we fix the shape parameter $s$ used in Spectral-RBF, Direct-RBF, {\color{black}Spectral-RBF-FD, and Direct-RBF-FD.} In particular, we choose $s=1.2$ for Direct-RBF. {\color{black}For Direct-RBF-FD and Spectral-RBF-FD, we use the Polyharmonic spline (PHS) kernel including linear terms (see Appendix~\ref{AppB} for details).} For Spectral-RBF, {\color{black}the} shape parameter is chosen as $s=1$ for solving {\color{black}the} eigenvalue problem of $\sum_{\ell=1}^n\mathbf{G}_{\ell}^\top\mathbf{G}_{\ell}$ and $s=1.2$ for construction of $\Delta_X^{\mathrm{RBF}}=\sum_{\ell=1}^n\mathbf{G}_{\ell}\mathbf{G}_{\ell}$. In VBDM, we choose $k_2=[10,15,25,35,50]$ for density estimation and $k_1=[20,30,45,65,120]$ nearest neighbors to construct the estimator $\mathbf{L}_{\epsilon,\rho}$.

To avoid the high computational cost in finding the eigenpairs of $\Delta_X^{\mathrm{SRBF}}$, we solve only one eigenvalue problem corresponding to one of the 15 randomly generated sample sets, $X$, and then apply the RBF interpolation to extend the estimated eigenvectors onto the function values on the grid points in the other 14 trials. Here, we assume that the error from RBF interpolation is negligible. If one solves the eigenvalue problem for each trial, a better result could be obtained but it will cost much more time in the experiment.

 In Fig.~\ref{fig-ellipse}(a), we show the average of $\ell^\infty$-error over the 15 independent trials. Here, the $\ell^\infty$-error corresponds to the usual vector uniform norm between the truth and the estimated solution attained from the corresponding methods, where the maximum is taken over $X$. The number of eigenfunctions used in Spectral-RBF is 60. To estimate the projection matrix, we choose $k=\sqrt{N}$ nearest neighbors in the 2nd-order local SVD. One can see that the average error decreases roughly on the order of $N^{-2}$ for Spectral-RBF, Spectral-RBF-FD, and Direct-RBF, {\color{black} which are better than Direct-RBF-FD and VBDM whose average error has an order of $N^{-1}$. We should point out that the estimate with Direct-RBF-FD can be improved to order-$N^{-2}$ with accuracy slightly better than Direct-RBF when the Mat\'ern kernel is used (see Figure~\ref{ellipse_matern} in Appendix~\ref{AppB}), which suggests that the local interpolation can be sensitive to the choice of kernel and possibly the distribution of the point cloud data. In addition, Fig.~\ref{fig-ellipse}(a) shows that the Spectral-RBF is about five times more accurate compared to both Spectral-RBF-FD and Direct-RBF.  We also found that the accuracy of Spectral-RBF-FD does not improve even if one uses Mat\'ern kernel as shown in Figure~\ref{ellipse_matern} in Appendix~\ref{AppB}.} In Fig.~\ref{fig-ellipse}(b), we show the average error as a function of the number of eigenfunctions, $K$, used in Spectral-RBF. {\color{black}One can see that each error corresponding to the case $N=[400,800,1600]$ decreases significantly and is plateauing as large enough modes, $K \geq 42$, are used.} In this regime ($K>42$), the error decreases as $N$ increases (confirming the $N^{-2}$ empirical convergence shown in Figure~\ref{fig-ellipse}).

   \begin{figure}[htbp]%
    \centering
      \begin{subfigure}[h]{0.45\linewidth}
    \caption{Error vs $N$.}
\includegraphics[width=\linewidth]{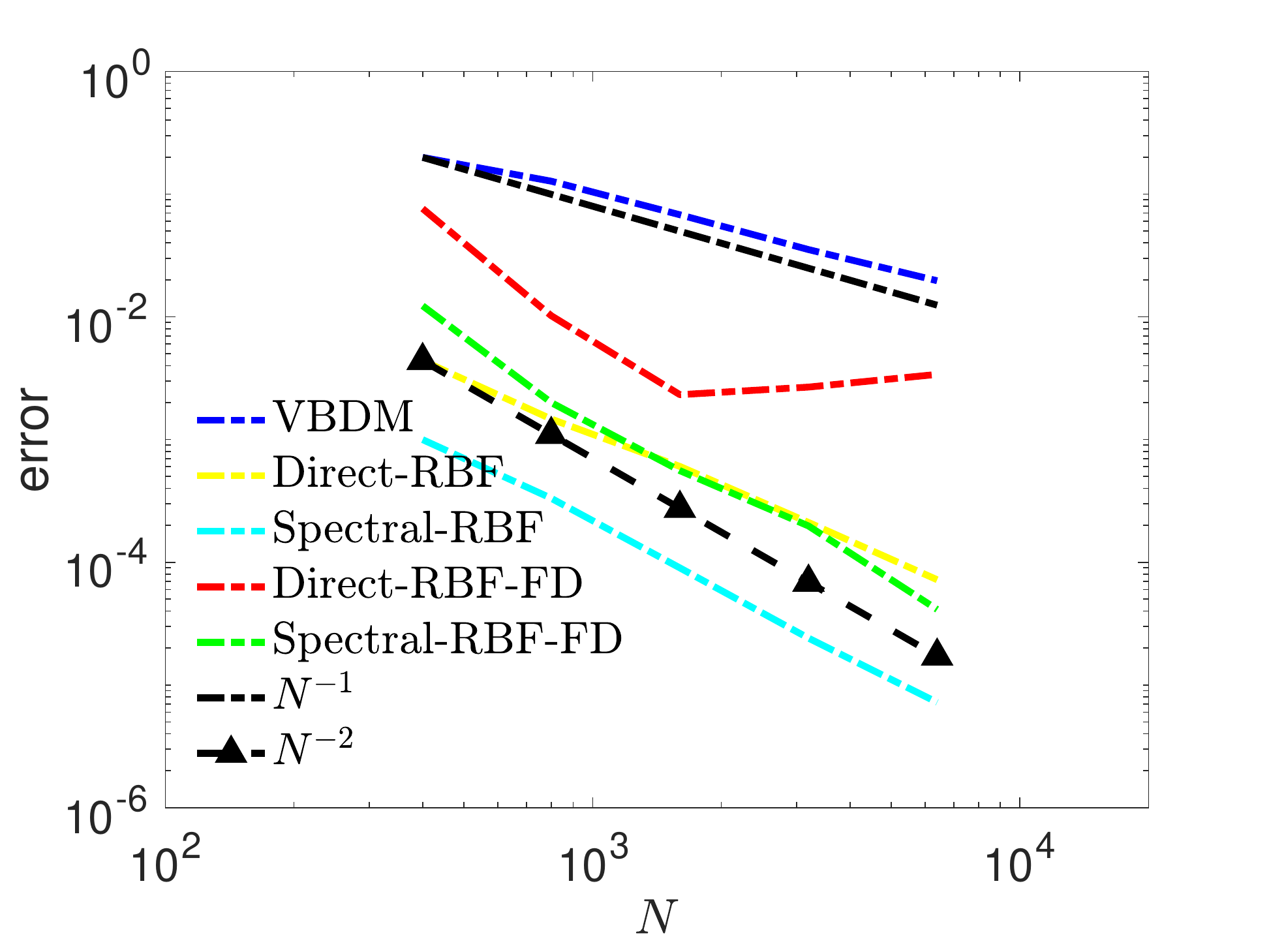}
\end{subfigure}
     \begin{subfigure}[h]{0.45\linewidth}
    \caption{Error vs $K$.}
 \includegraphics[width=\linewidth]{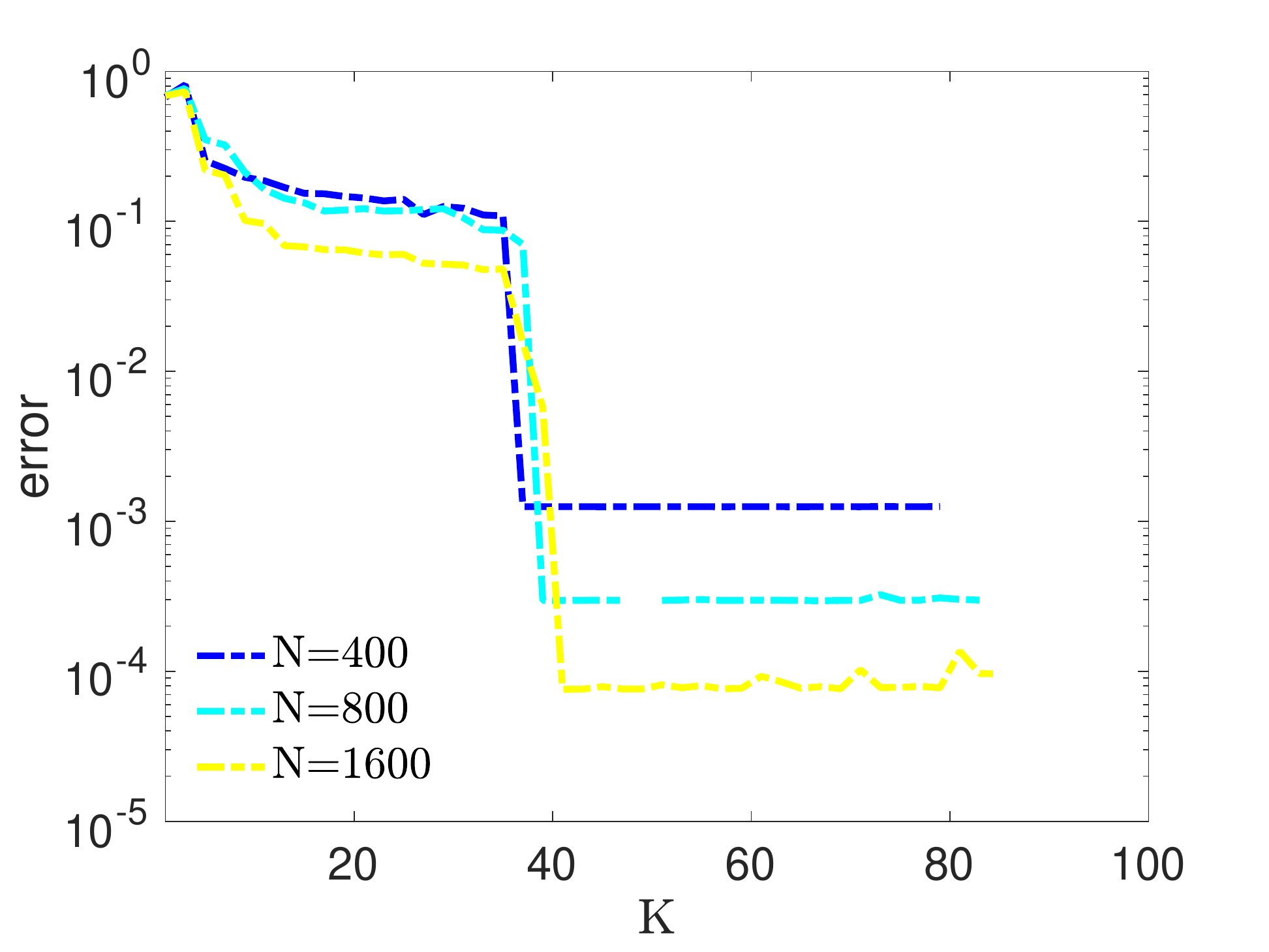}
\end{subfigure}
 \caption{Ellipse example: (a). Average $\ell^\infty$ error with respect to $N$. (b). Average $\ell^\infty$ error of Spectral-RBF with respect to number of eigenfunctions, $K$. }
\label{fig-ellipse}

 \end{figure}


\subsection{PDE on a torus}\label{sec4.3}
For the second example, we consider solving the equation (\ref{eq1}) on a torus  which has the embedding function
$$
\iota(\theta, \phi)=\left(\begin{array}{c}{(R+r\cos \theta) \cos \phi} \\ {(R+r\cos \theta) \sin \phi} \\ {r\sin \theta}\end{array}\right), \quad \theta \in[0,2 \pi),\quad \phi\in[0,2\pi)
$$
where $R$ is the distance from the centre of the tube to the center of torus and $r$ is the distance from the center of the tube to the surface of the tube with $r<R$. The induced Riemannian metric is
$$
g_{\textbf{x}^{-1}(\theta, \phi)}(v, w)=v^{\top}\left(\begin{array}{cc}r^2 & 0 \\ 0 & (R+r\cos \theta)^{2}\end{array}\right) w, \quad \forall v, w \in T_{\textbf{x}^{-1}(\theta, \phi)} M.
$$
In our numerical experiments, we take $R=2$ and $r=1$. The true solution $u$ of the PDE is set to be $u=\sin\phi\sin\theta$, the diffusion coefficient is defined to be $
\kappa=1.1+\sin ^{2} \theta \cos ^{2} \phi$ and the constant $c=1$. The forcing term $f$ could be calculated as
$$
f:=-\operatorname{div}_g(\kappa \textup{grad}_g u)+c u=-\frac{1}{\sqrt{|g|}}\left[\frac{\partial}{\partial \theta}\left(\kappa \sqrt{|g|} g^{11} \frac{\partial u}{\partial \theta}\right)+\frac{\partial}{\partial \phi}\left(\kappa \sqrt{|g|} g^{22} \frac{\partial u}{\partial \phi}\right)\right]+c u.
$$

Numerically, the grid points $\{\theta_i\}$ are randomly sampled from the uniform distribution on $[0,2\pi)\times[0,2\pi)$.  To illustrate the convergence of solutions over number of points $N=[800,1600,3200,6400,12800]$, we fix the shape parameter $s$ used in Spectral-RBF and Direct-RBF. In particular, we choose $s=0.7$ in Direct-RBF. In Spectral-RBF, shape parameter is chosen as $s=0.3$ for solving eigenvalue problem of $\sum_{\ell=1}^n\mathbf{G}_{\ell}^\top\mathbf{G}_{\ell}$ and $s=0.7$ for construction of $\Delta_X^{\mathrm{RBF}}$. Again, we solve the eigenvalue problem on first trial of generated grid points and then apply RBF interpolation to obtain new eigenfunctions on newly generated grid points in other trials.  In VBDM, we choose $k_2=[14,20,28,40,55]$ for density estimation and $k_1=[28,40,56,80,110]$ nearest neighbors to construct the estimator $\mathbf{L}_{\epsilon,\rho}$.

In Fig.~\ref{fig-torus}(a), we show the average $\ell^\infty$-error over 10 independent trials as a function of $N$. To estimate the projection matrix, we pick $k=\sqrt{N}$ in the 2nd-order local SVD scheme. The number of eigenfunctions used in Spectral-RBF is 300. The average $\ell^\infty$ error decrease on order of $N^{-1/2}$ in VBDM. On the other hand, the errors of Direct-RBF and Spectral-RBF decrease roughly on the order of $N^{-1}$ which is better than VBDM again. In Fig.~\ref{fig-torus}(b), we show the average error as a function of the number of eigenfunctions, $K$, used in Spectral-RBF. In Fig.~\ref{fig-torus}(b),  the average error keeps decreasing until $K$ reach at a particular number which is similar to the case on ellipse. And it is also clear that for fixed $K$, the error is always smaller for larger sample size, $N$. In  Fig.~\ref{fig-torus_soln}, we show the distribution of errors of Spectral-RBF and Direct-RBF on the torus for sample size $N=6400$; we can see that the Spectral-RBF estimate is uniformly more accurate than that of Direct-RBF.

   \begin{figure}[htbp]%
    \centering
      \begin{subfigure}[h]{0.45\linewidth}
    \caption{Error vs $K$.}
\includegraphics[width=\linewidth]{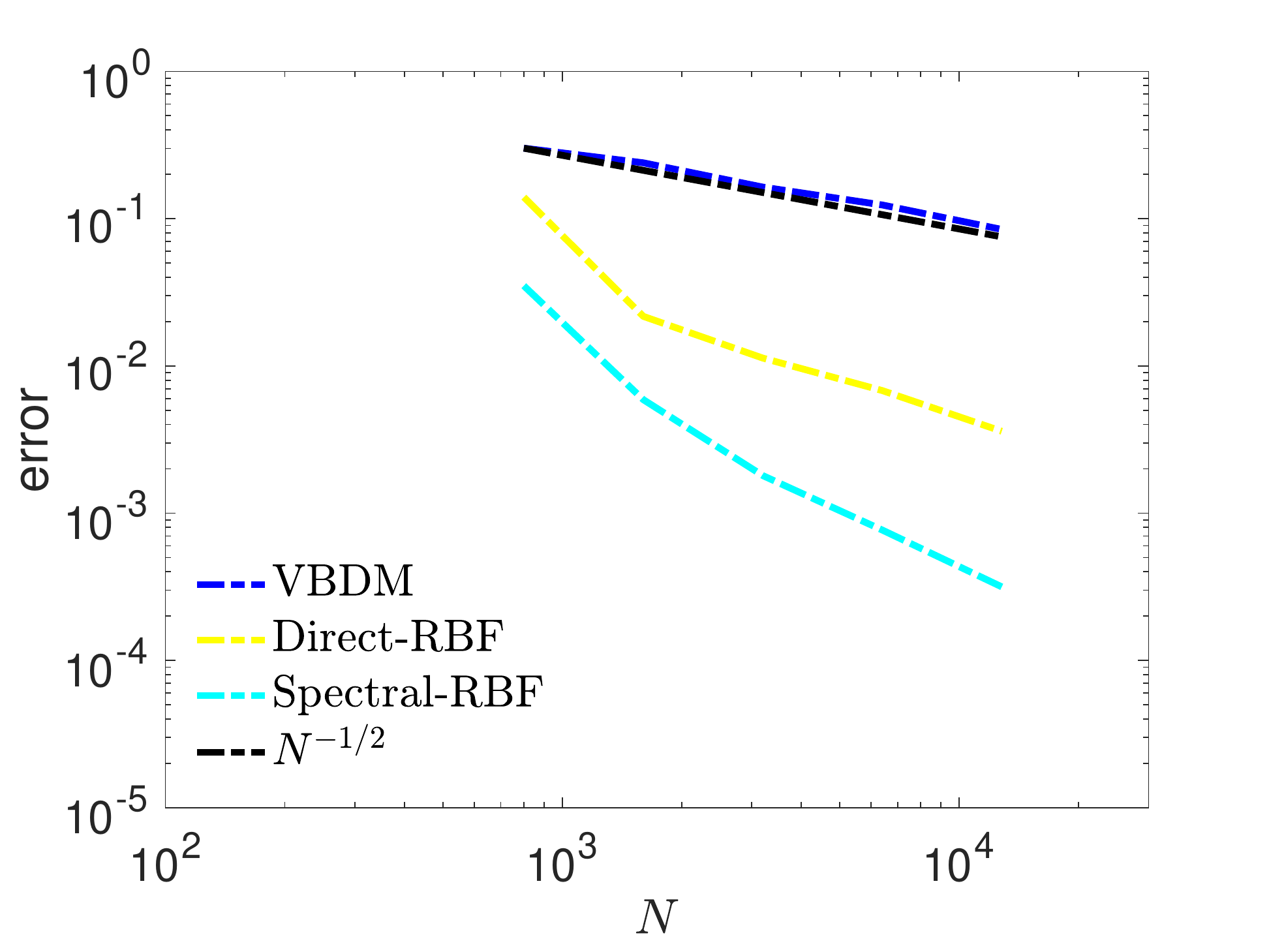}
\end{subfigure}
     \begin{subfigure}[h]{0.45\linewidth}
    \caption{Error vs $K$.}
    \includegraphics[width=\linewidth]{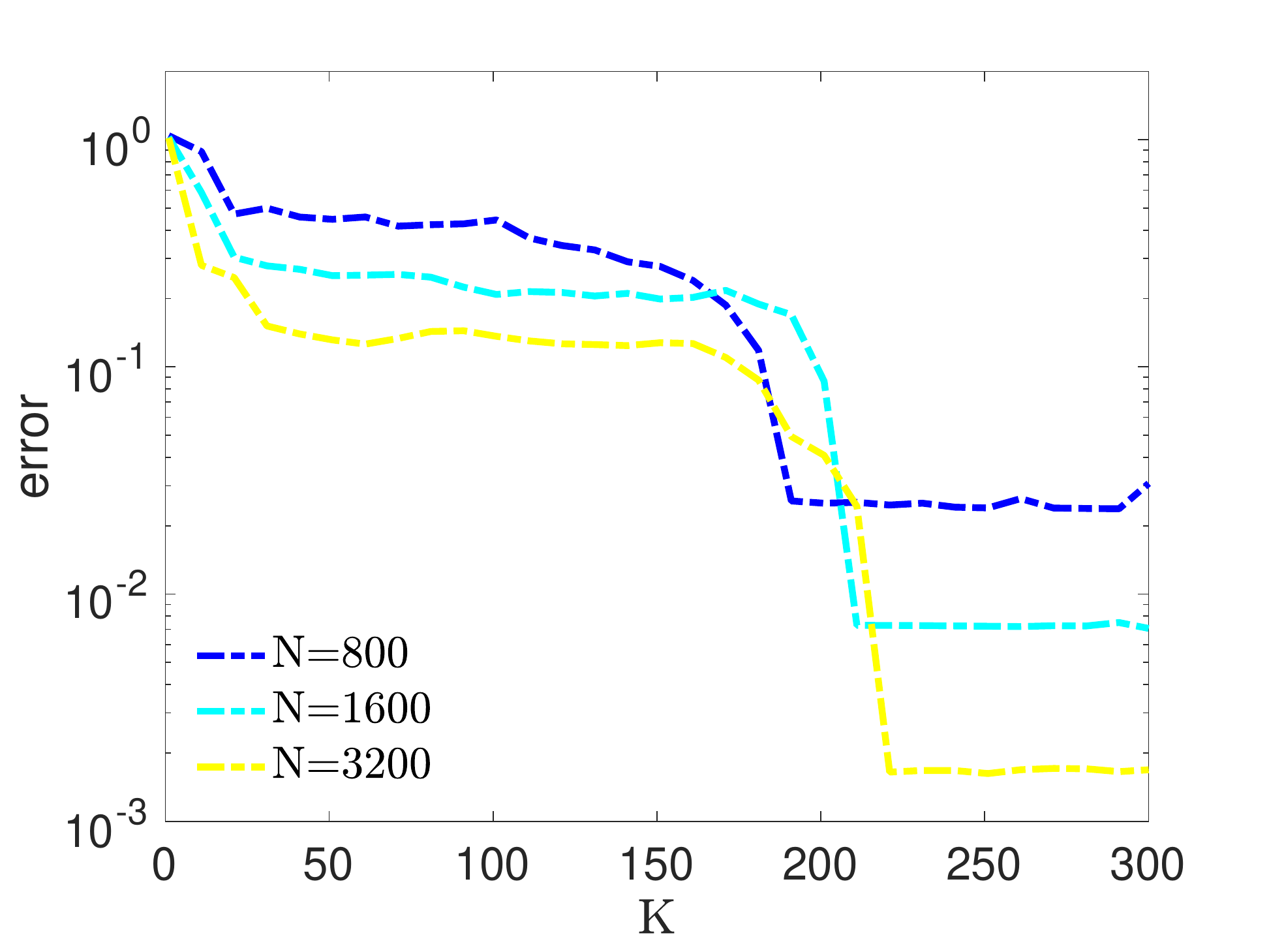}
\end{subfigure}
 \caption{Torus example: (a). Average $\ell^\infty$ error with respect to $N$. (b). Average $\ell^\infty$ error of Spectral-RBF vs number of eigenfunctions, $K$.}
   \label{fig-torus}
 \end{figure}

   \begin{figure}[htbp]%
    \centering
      \begin{subfigure}[h]{0.45\linewidth}
    \caption{Direct-RBF, $\ell^\infty$-error$=0.0060$}
\includegraphics[width=\linewidth]{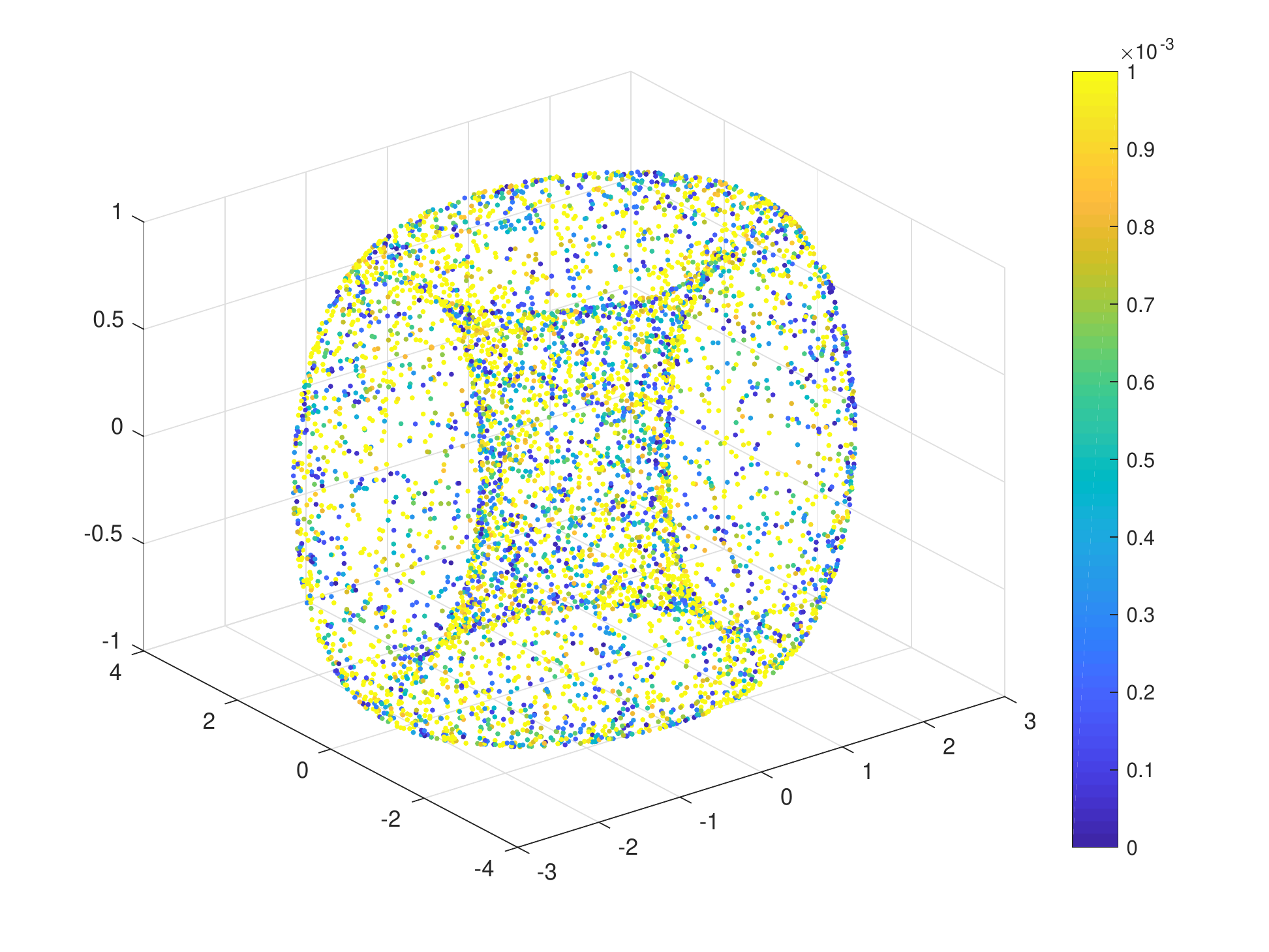}
\end{subfigure}
     \begin{subfigure}[h]{0.45\linewidth}
    \caption{ Spectral-RBF, $\ell^\infty$-error$= 6.7330e-04$}
    \includegraphics[width=\linewidth]{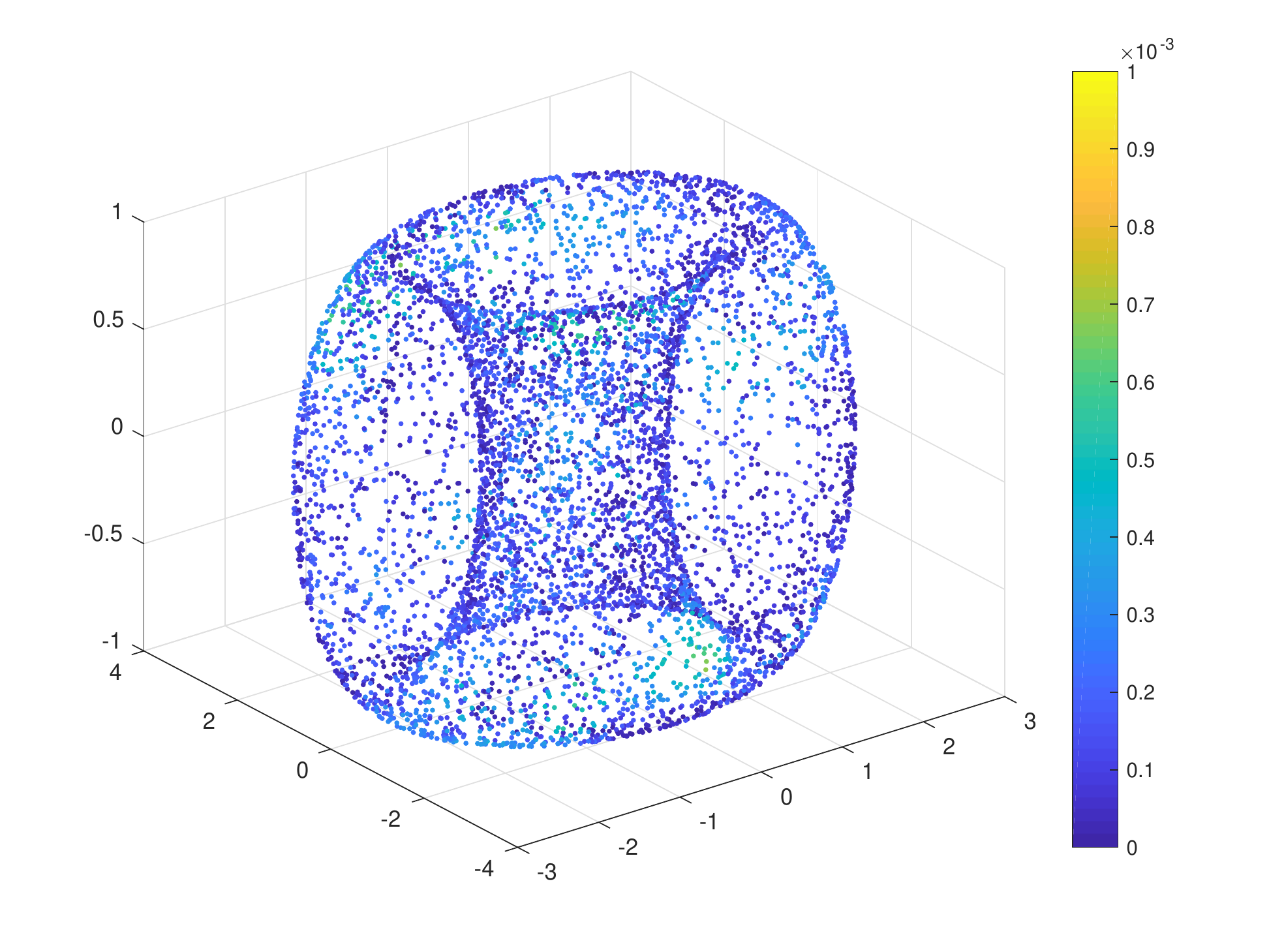}
\end{subfigure}
 \caption{Torus example: absolute difference at each point when $N=6400$. (a). Direct-RBF. (b). Spectral-RBF.}
  \label{fig-torus_soln}
 \end{figure}

\subsection{PDE on the ``Stanford Bunny''}
\label{section_bunny}
In the last example, we consider solving the elliptic problems on the Stanford Bunny model which is a two-dimensional surface embedded in $\mathbb{R}^3$. The data of this model is downloaded from the Stanford 3D Scanning Repository \cite{bunny}. The original data set of the Stanford Bunny comprises a triangle mesh with 34,817 vertices. Since this data set has singular regions at the bottom, we generate a new mesh of the surface using the Marching Cubes algorithm~\cite{lorensen1987marching} that is available through the Meshlab~\cite{cignoni2008meshlab}. We should point out that the Marching Cubes algorithm does not smooth the surface. We will also verify the solution on a smoothed surface, following the numerical work in \cite{shankar2015radial}. Specifically, we smooth the surface using the Screened Poisson surface reconstruction algorithm to generate a watertight implicit surface that fits the point cloud of the bunny. Subsequently, we will use the vertices of the new mesh as sample points to avoid singularity induced by the original data set. To check the convergence behavior of different schemes with respect to different sizes of data points, we subsequently apply the Quadric edge algorithm \cite{garland1997surface} to simplify the mesh obtained by the Marching Cubes algorithm (which is a surface of 34,594 vertices) into 4000, 8000, and 16,000, and 32000 vertices. See Fig.~\ref{bunny_poisson}(a) and (b) for an example of sample points from the un-smoothed and smoothed surfaces with $N=32000$.

The equation we consider here is $\Delta_Mu+0.2u=f$ with $f=0.6(x_1+x_2+x_3)$.
In this example, we have no access to the analytic solution due to the unknown embedding function. For comparisons, we take the surface finite-element method solution obtained from the FELICITY FEM Matlab toolbox \cite{walker2018felicity} as the reference (see Figure~\ref{bunny_poisson}(c) and (d) for the solutions corresponding to un-smoothed and smoothed bunnies).
In the following, instead of using the $\ell^\infty-$error metric, we will compare the relative difference to the FEM solution which is defined as
$$
\Delta E=\text{relative difference }=\max\limits_{1\leq i\leq N}\frac{|\widehat{u}(\textbf{x}_i) - u^{\mathrm{FEM}}(\textbf{x}_i)|}{|u^{\mathrm{FEM}}(\textbf{x}_i)|}
$$
where $u^{\mathrm{FEM}}$ is the FEM solution and $\widehat{u}$ is the estimated solution obtained from the corresponding solver.

First, let us report the results on un-smoothed surface. In our numerical experiment, we found that Direct-RBF do not work on this surface. We suspect that the estimation of the tangent space in this ``rough surface'' is rather poor such that the estimated projection matrix $\hat{\mathbf{P}}$ is not well estimated. In Fig.~\ref{bunny}, we compared solutions obtained from the VBDM and Spectral-RBF.  For Spectral-RBF, the shape parameter is chosen to be $s=0.5$ for solving eigenvalue problem of $\Delta_X^{\mathrm{SRBF}}$ and $s=1$ for the construction of $\Delta_X^{\mathrm{RBF}}$.  For VBDM, we fix $k_2=64$ for the density estimation and $k_1=128$ nearest neighbors for the construction of the estimator $\mathbf{L}_{\epsilon,\rho}$ for each $N$. From the simulations based on these parameters, we found that VBDM is slightly more accurate than Spectral-RBF. Here, we show that the result can be slightly improved with an alternative symmetric formulation VBDM-RBF discussed in Remark~\ref{altbasis}, where the Galerkin expansion is employed with a set of orthonormal basis $\{\varphi_j\}_{j=1,\ldots, K}\subset L^2(M)$ approximated by eigenvectors of the VBDM approximation to Laplace-Beltrami operator. Comparing panels (b) and (c) in Fig.~\ref{bunny_soln}, we can see that the spatial distribution of the relative errors of Spectral-RBF and VBDM are somewhat different while the maximum difference is comparable.
In panel (d), one can see that the alternative VBDM-RBF scheme shows more accurate result.

   \begin{figure*}[tbp]%
    \centering
      \begin{subfigure}[h]{0.45\linewidth}
    \caption{Unsmoothed samples}
\includegraphics[width=\linewidth]{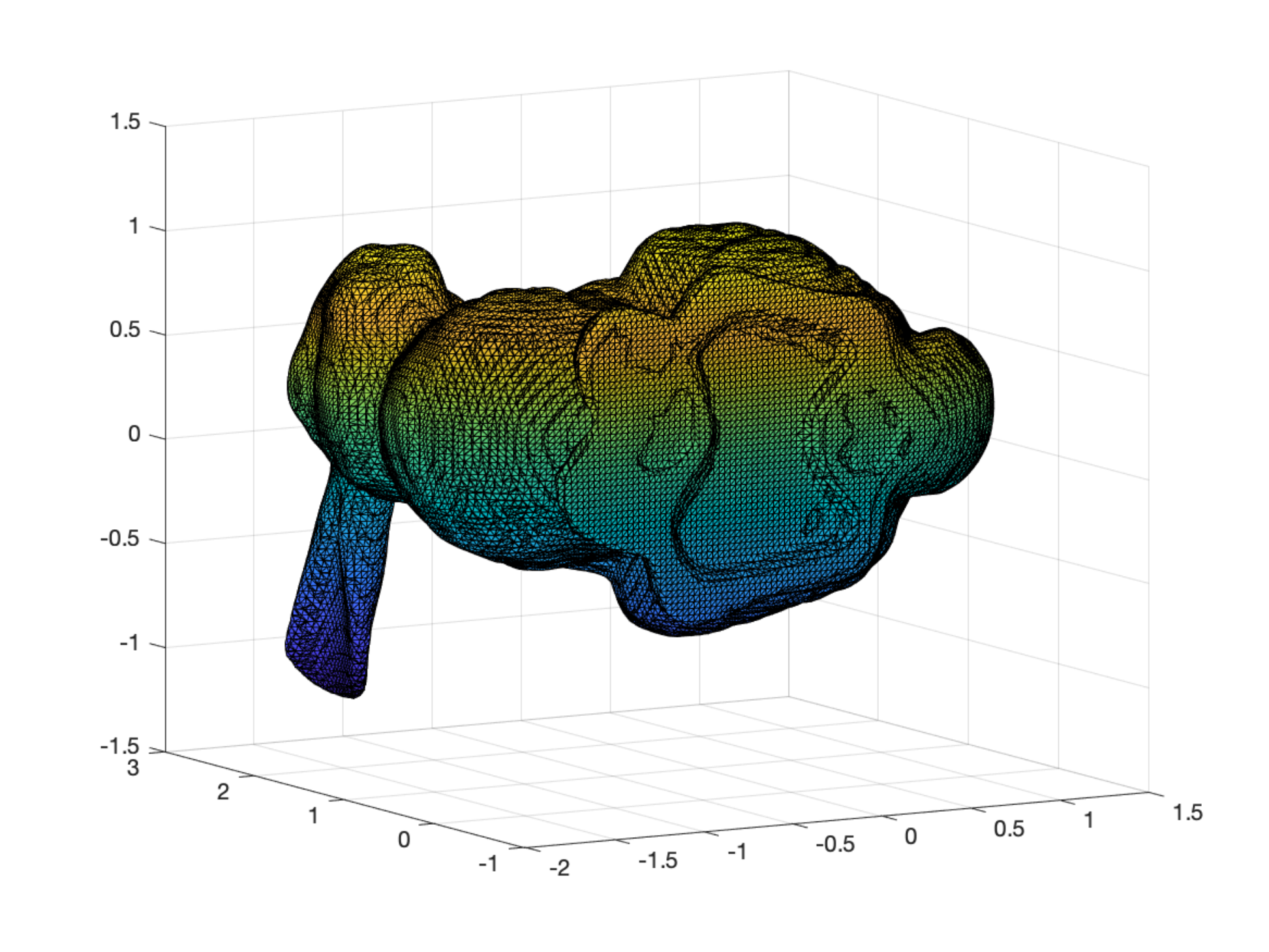}
\end{subfigure}
     \begin{subfigure}[h]{0.45\linewidth}
    \caption{Smoothed samples.}
    \includegraphics[width=\linewidth]{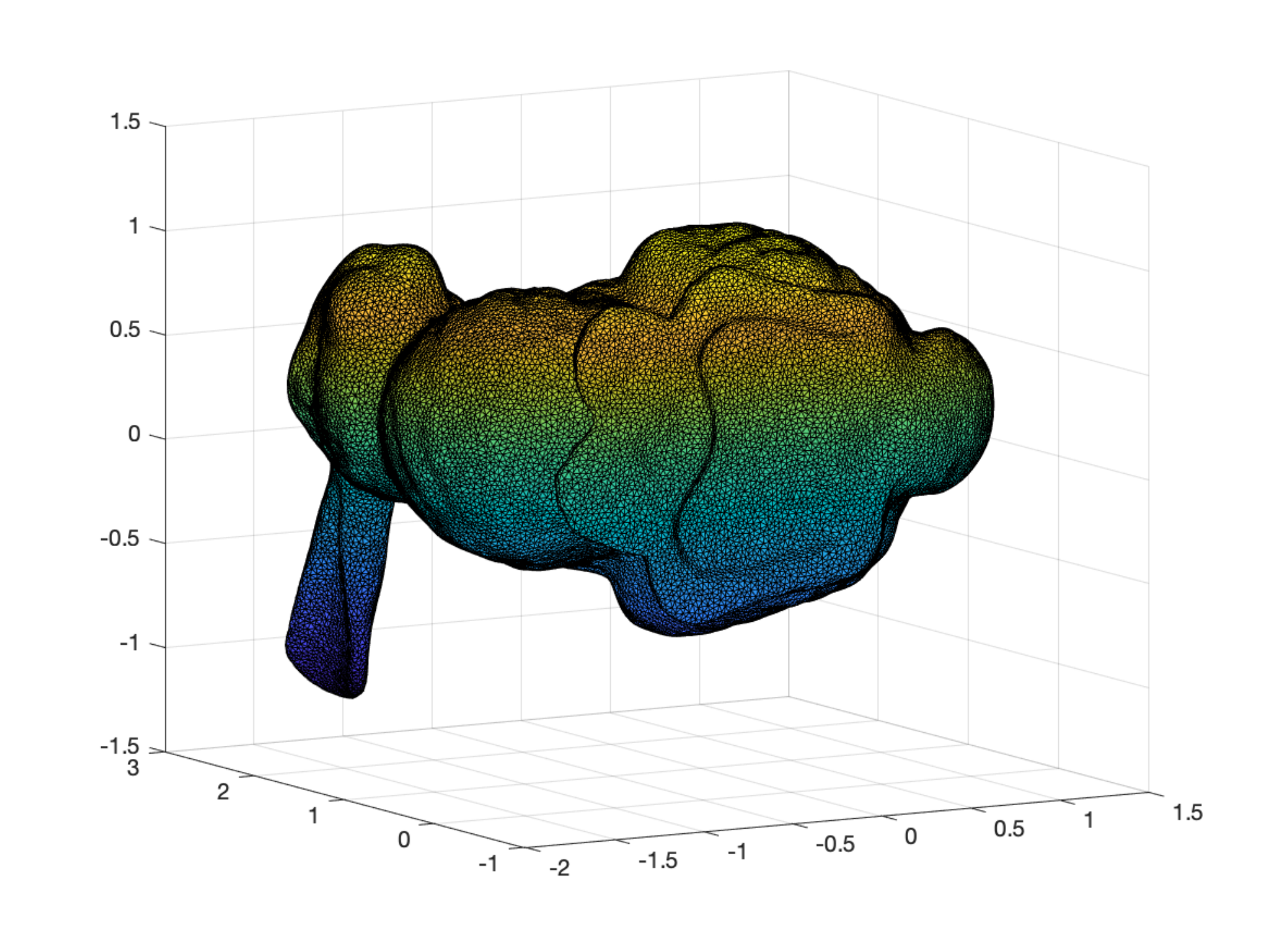}
\end{subfigure}
     \begin{subfigure}[h]{0.45\linewidth}
    \caption{FEM solution on unsmoothed samples.}
    \includegraphics[width=\linewidth]{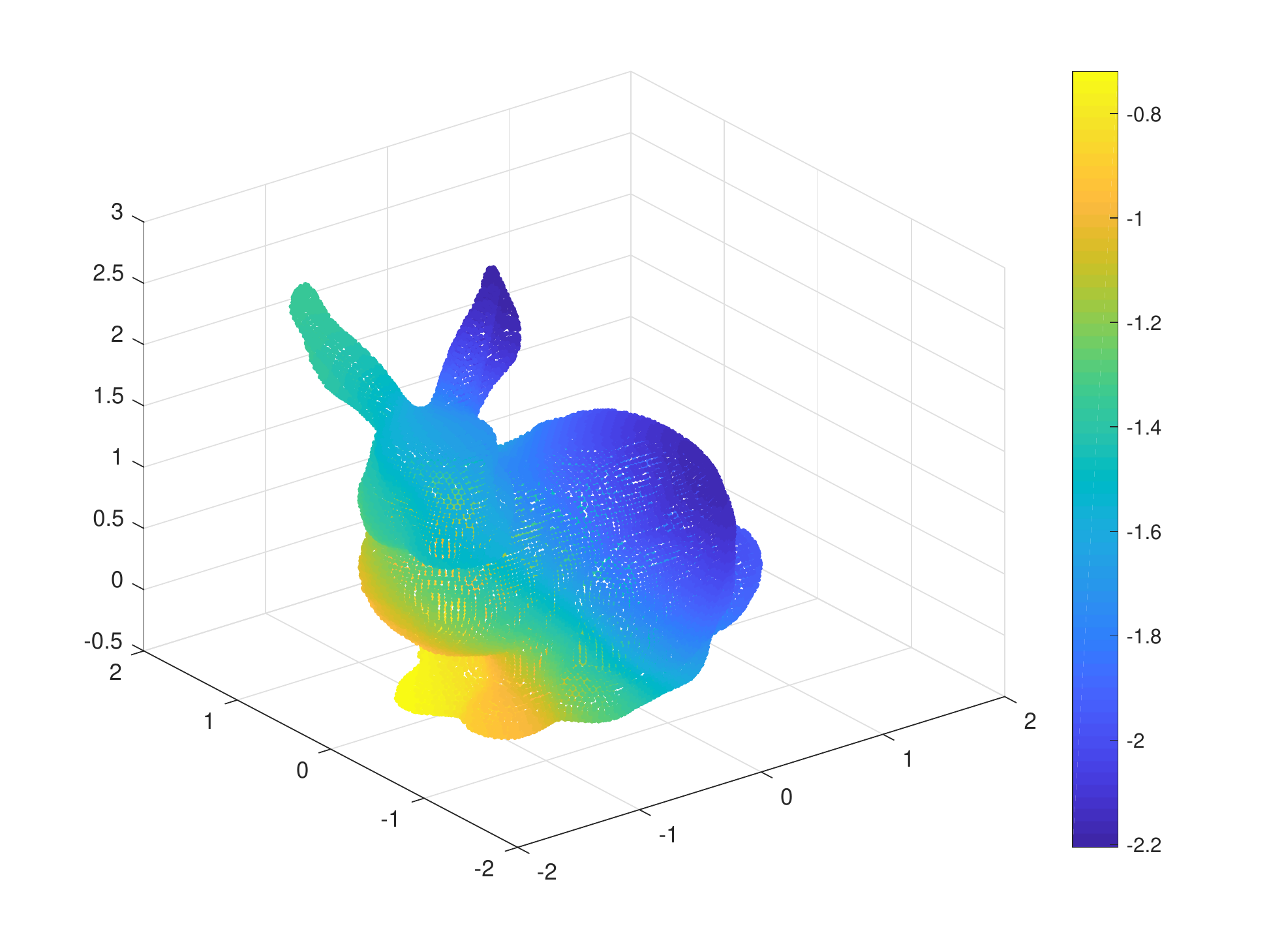}
\end{subfigure}
     \begin{subfigure}[h]{0.45\linewidth}
    \caption{FEM solution on smoothed samples.}
    \includegraphics[width=\linewidth]{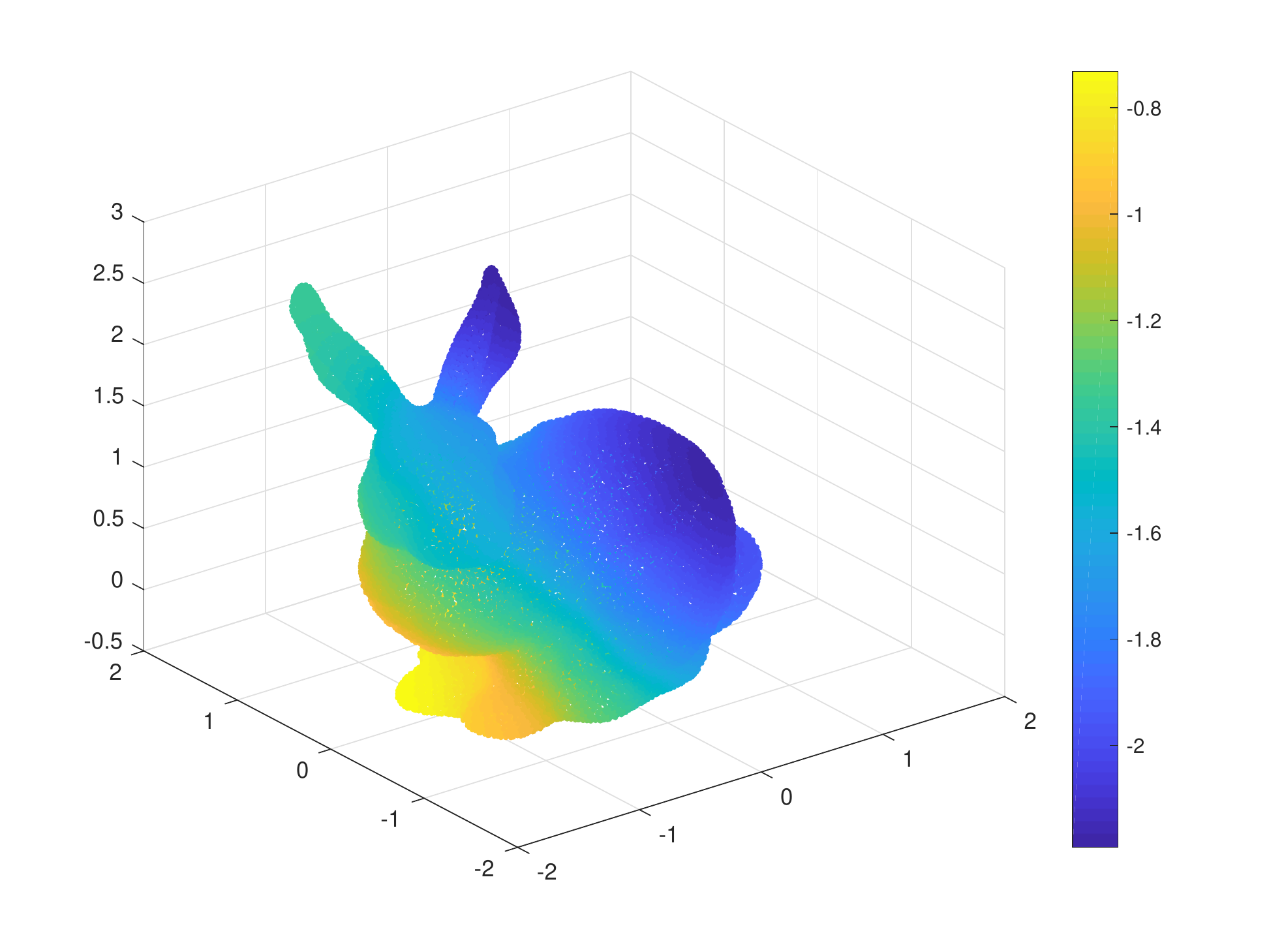}
\end{subfigure}
 \caption{Bunny example: $N=32000$ point cloud samples (a). without Poisson reconstruction, and (b). from a smoothed surface obtained by the Screened Poisson reconstruction algorithm. Shown in panels (c) and (d) are the FEM solutions on unsmoothed samples and smoothed samples, respectively. }
  \label{bunny_poisson}
 \end{figure*}


    \begin{figure*}[tbp]%
    \centering
\includegraphics[width=0.6\linewidth]{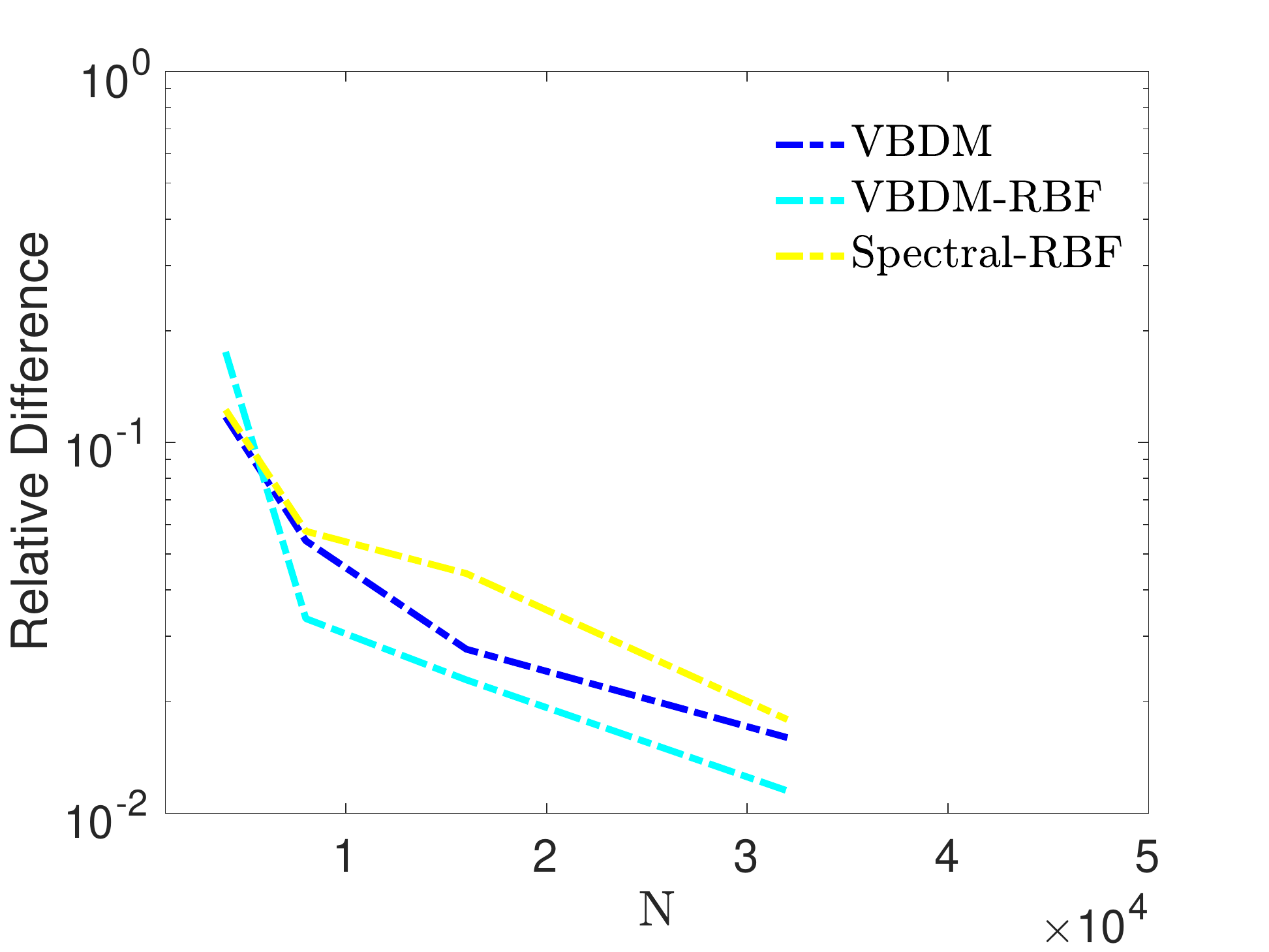}
\caption{Unsmoothed bunny example: Relative difference of different schemes $\Delta E$ with respect to different number of points $N$.}
\label{bunny}
 \end{figure*}

   \begin{figure*}[tbp]%
    \centering
     \begin{subfigure}[h]{0.33\linewidth}
    \caption{VBDM, $\Delta E = 0.0160$.}
    \includegraphics[width=\linewidth]{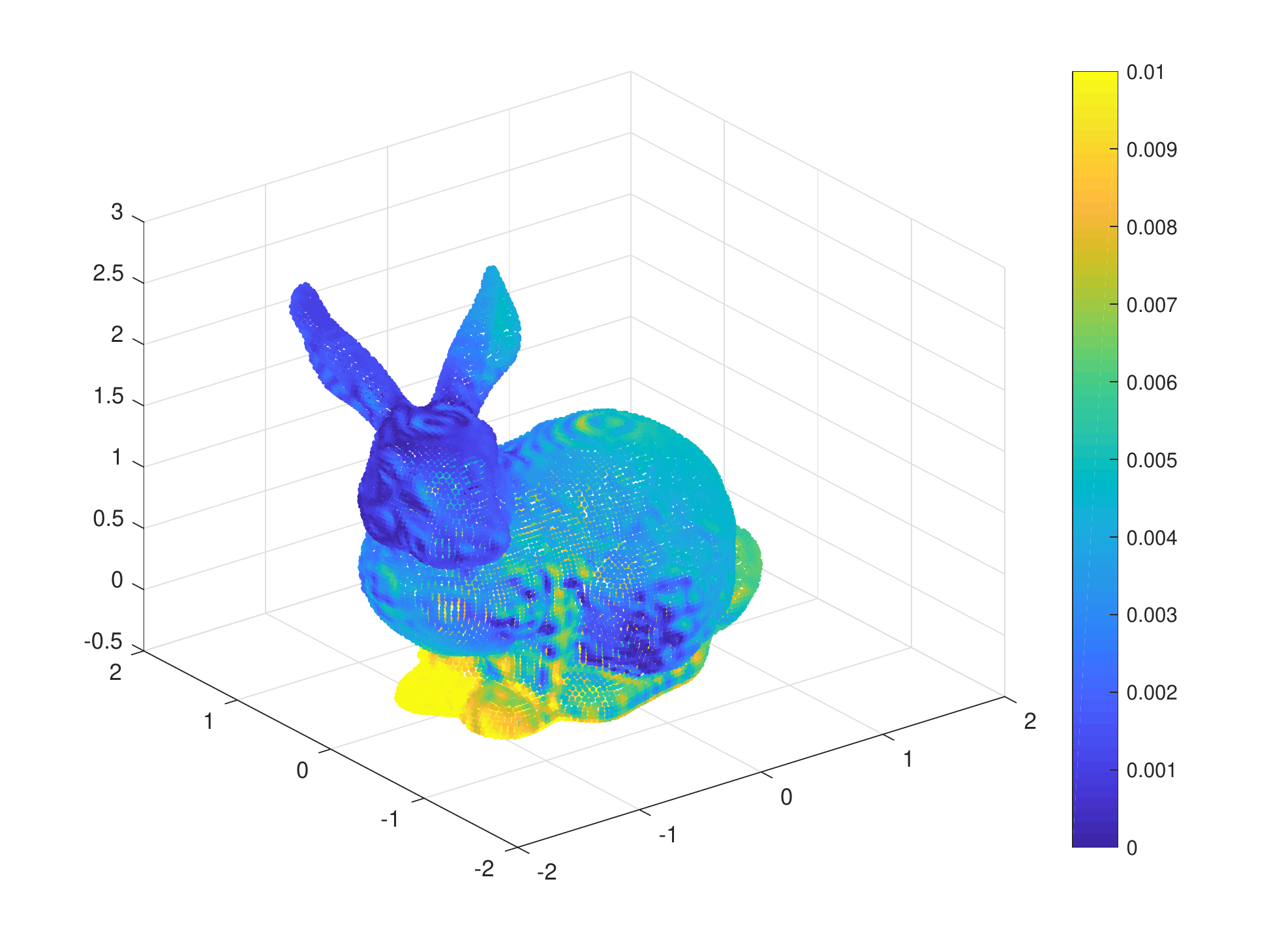}
\end{subfigure}
  \begin{subfigure}[h]{0.33\linewidth}
    \caption{Spectral-RBF, $\Delta E = 0.0178$.}
    \includegraphics[width=\linewidth]{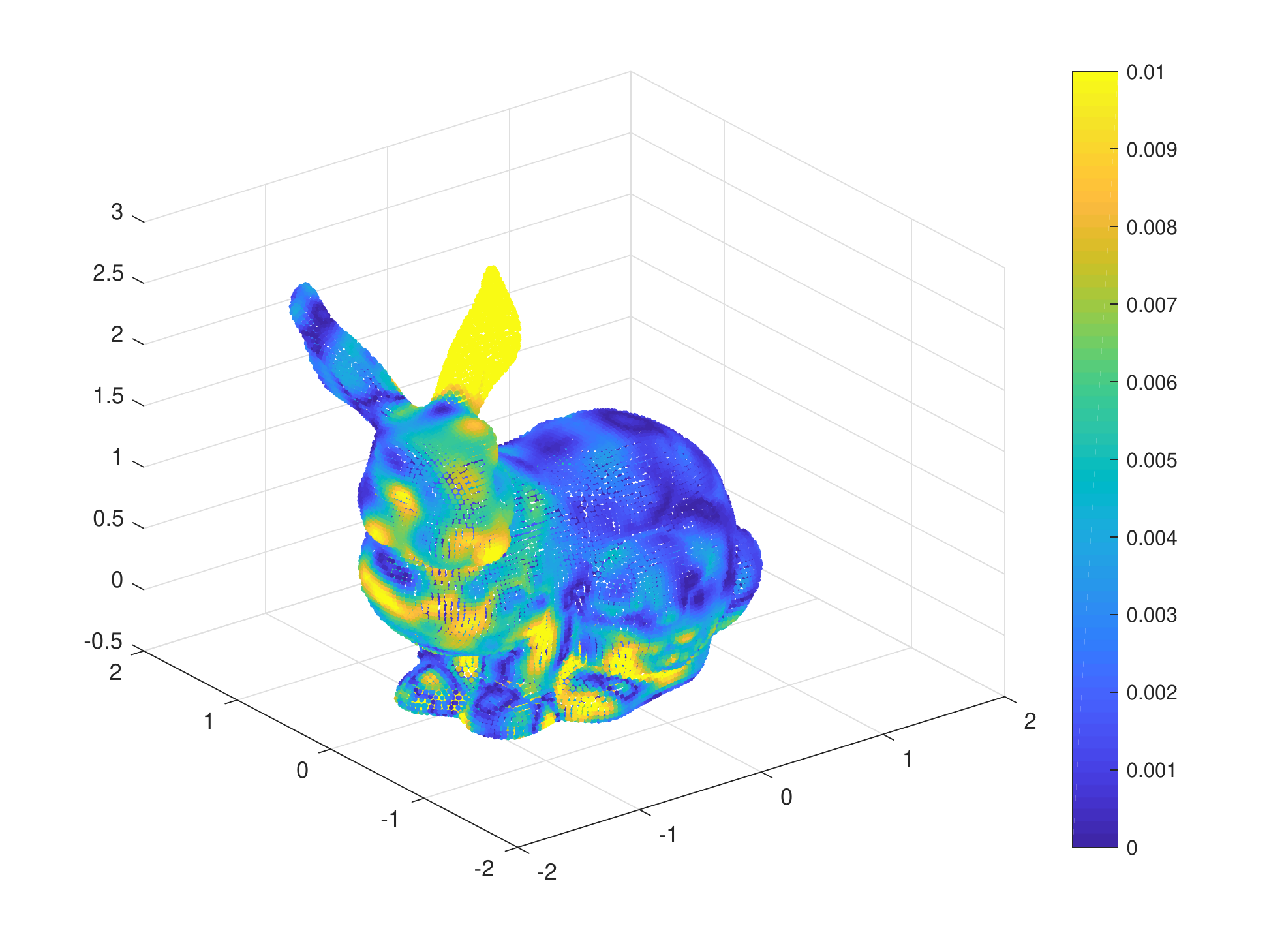}
\end{subfigure}
  \begin{subfigure}[h]{0.33\linewidth}
    \caption{VBDM-RBF, $\Delta E = 0.0115$.}
    \includegraphics[width=\linewidth]{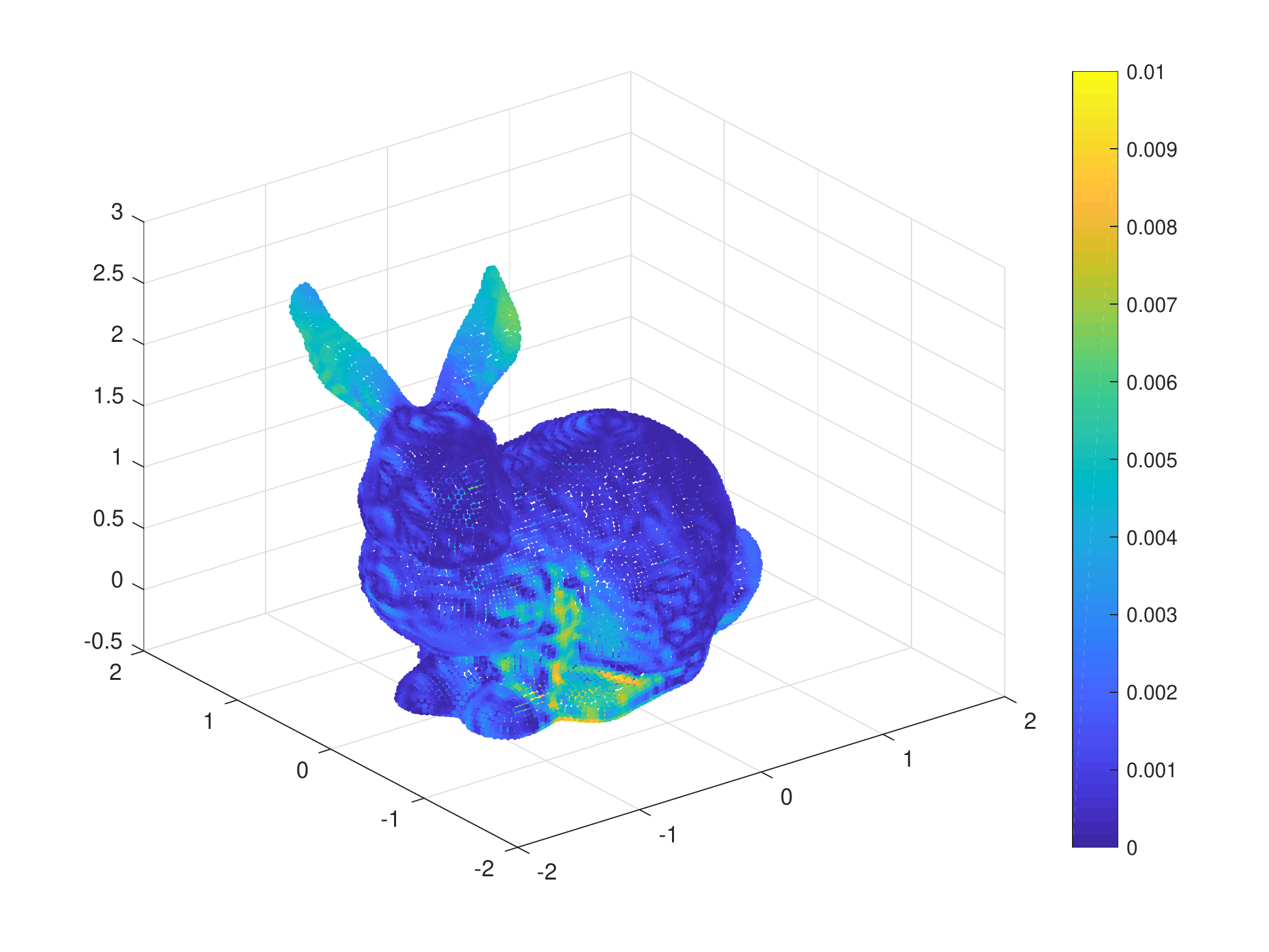}
\end{subfigure}
 \caption{Unsmoothed bunny example: $N=32000$. Relative difference at each point on the bunny for various methods.}
 \label{bunny_soln}
 \end{figure*}


On the smoothed bunny, we use the same parameters as in the previous case. We also include results from the Direct-RBF, obtained with shape parameter $s=13$. In Fig.~\ref{bunny_smooth_soln}, we can see that Direct-RBF does work on this smoothed bunny while their relative difference to the FEM solutions is larger than those obtained from the other three methods. On this smoothed surface, we note that Spectral-RBF performs slightly better than VBDM and the alternative scheme VBDM-RBF still produces the most accurate solution in the sense that it is closest to the FEM solution.

   \begin{figure*}[tbp]%
    \centering
     \begin{subfigure}[h]{0.45\linewidth}
      \caption{VBDM, $\Delta E =0.0163$.}
    \includegraphics[width=\linewidth]{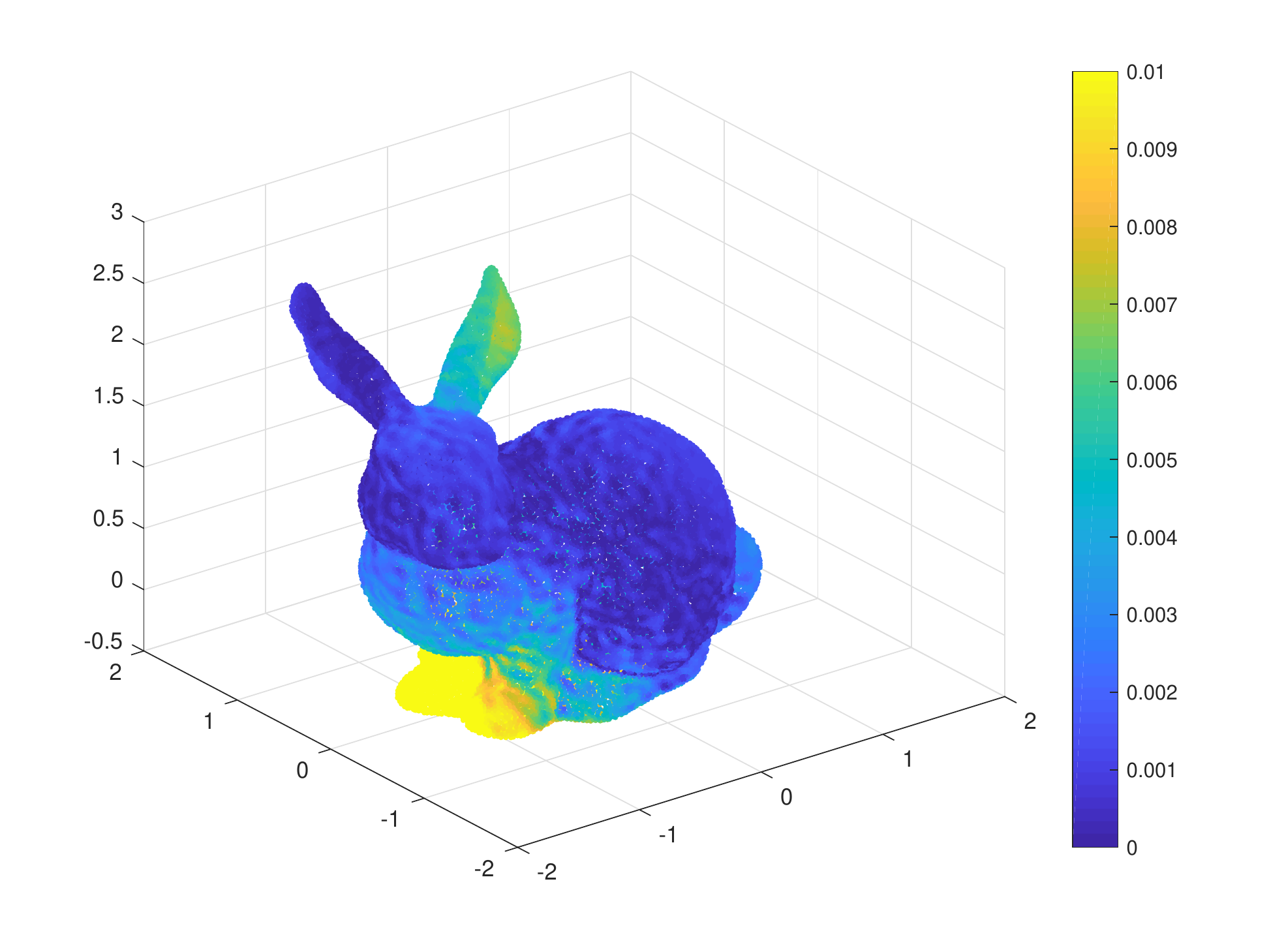}
\end{subfigure}
  \begin{subfigure}[h]{0.45\linewidth}
    \caption{Spectral-RBF, $\Delta E =0.0111, s=0.75$.}
    \includegraphics[width=\linewidth]{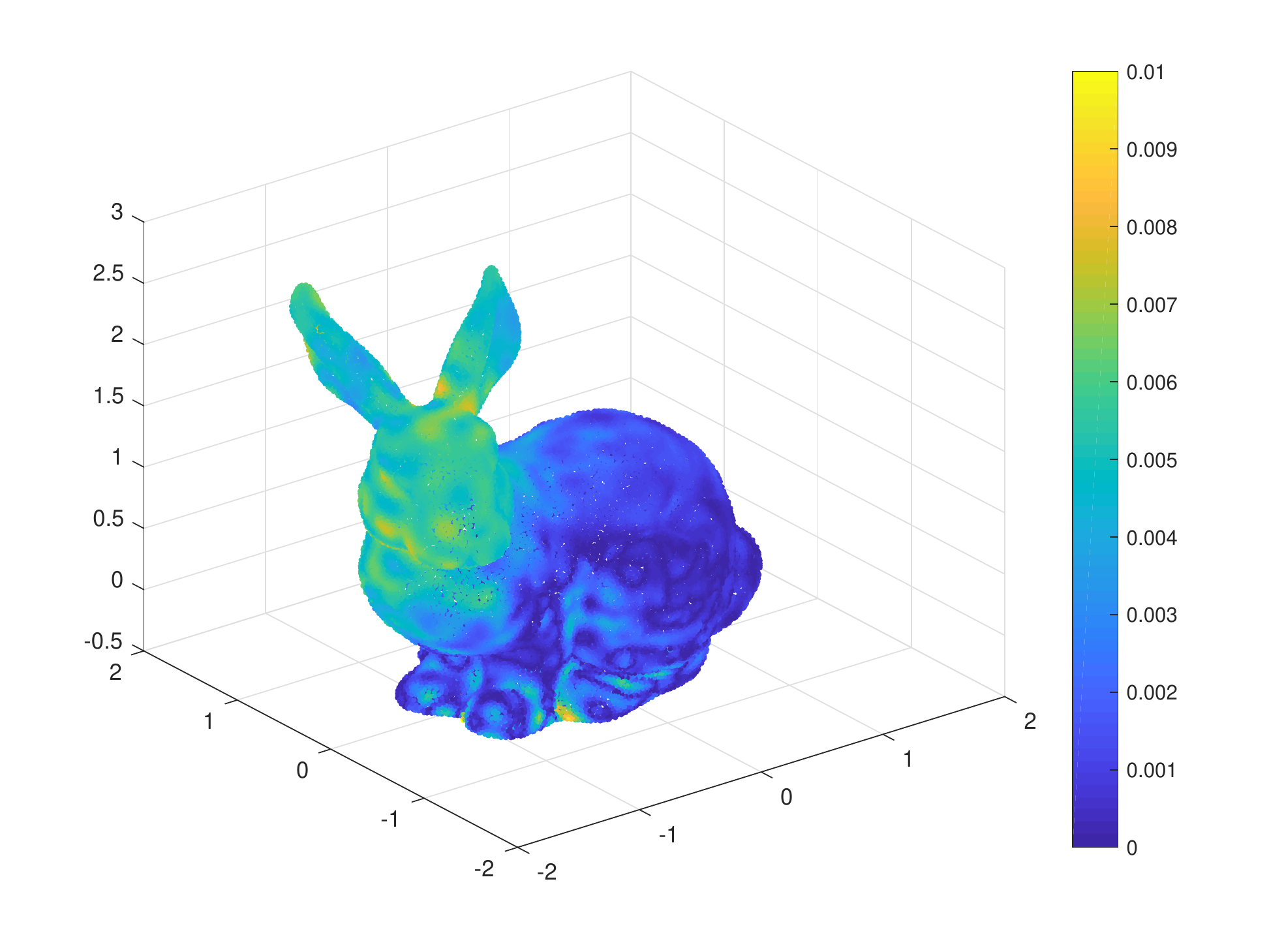}
\end{subfigure}
  \begin{subfigure}[h]{0.45\linewidth}
    \caption{VBDM-RBF, $\Delta E =0.0073$.}
   \includegraphics[width=\linewidth]{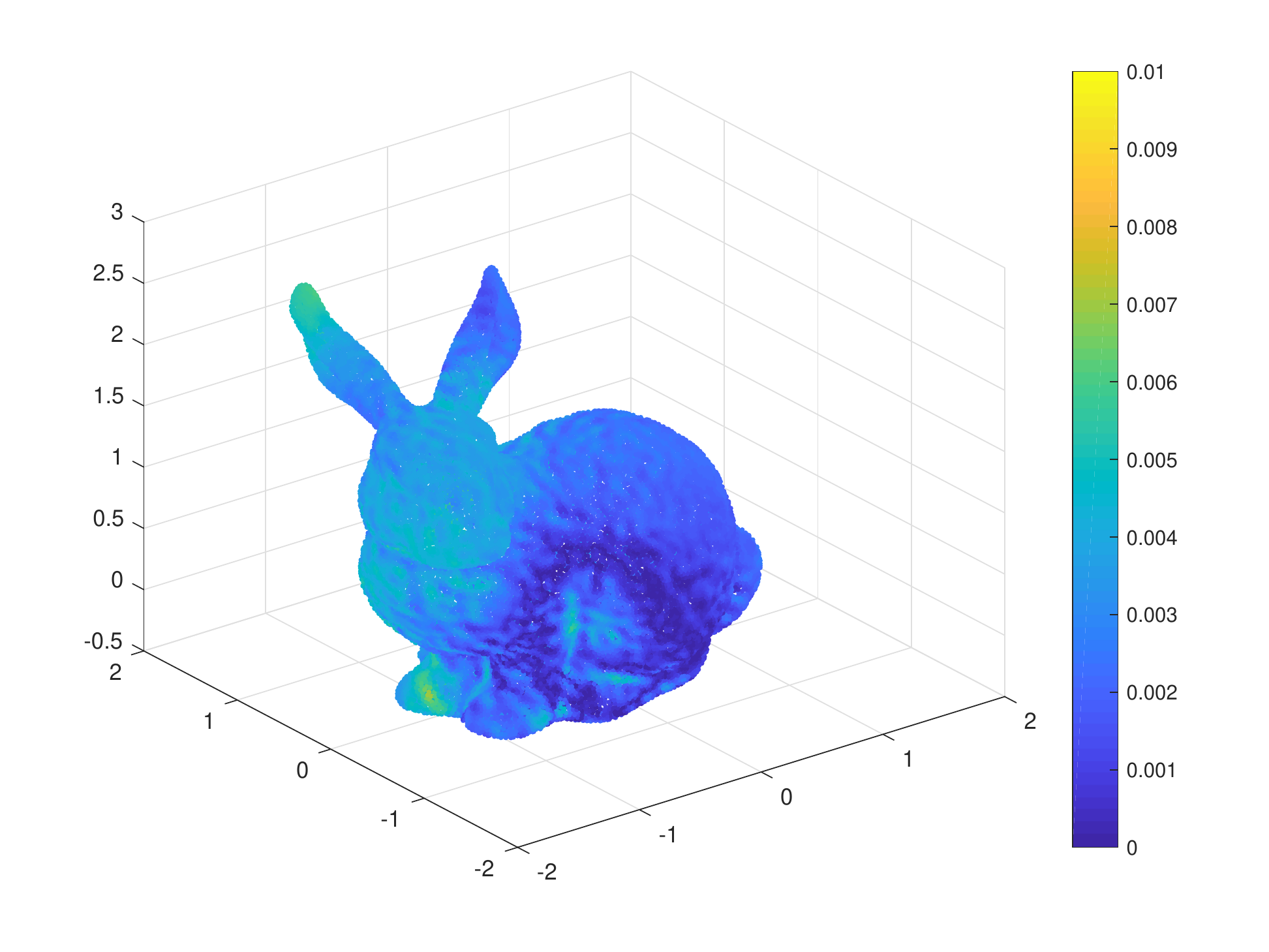}
\end{subfigure}
  \begin{subfigure}[h]{0.45\linewidth}
    \caption{Direct-RBF, $\Delta E =0.0414$, $s=13$.}
    \includegraphics[width=\linewidth]{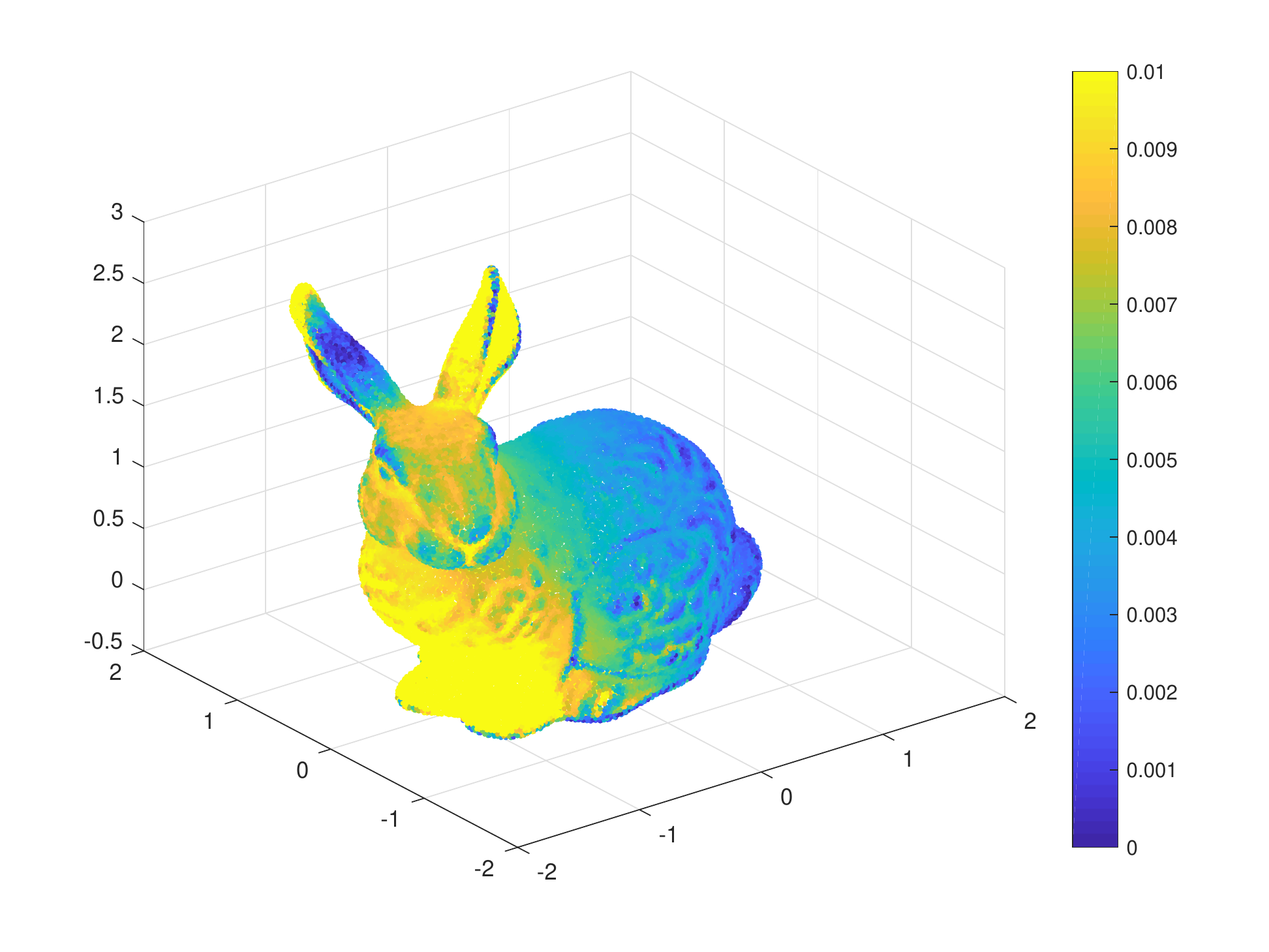}
\end{subfigure}
 \caption{Smoothed bunny example: Relative difference of different schemes at each point on the bunny with points $N=32000$.}
 \label{bunny_smooth_soln}
 \end{figure*}

\section{Summary}\label{sec5}
In this paper, we proposed a spectral method for solving elliptic PDEs on unknown manifolds identified by a set of point clouds. In particular, we solved the PDEs using a Galerkin approximation on a test function space spanned by a finite number of eigenfunctions of a weighted Laplacian operator that is weakly approximated by a symmetric RBF formulation \cite{harlim2022rbf}. We also employed a second-order local SVD method to improve the estimation of the projection matrix on point clouds, which subsequently improves the RBF interpolation accuracy. Theoretically, we deduced an error bound in terms of the number of eigenfunctions and also the size of the point cloud under appropriate assumptions of the PDE problems as well as weakly unstable RBF interpolants. The latter assumption is motivated by the fact that RBF interpolation is known to be unstable especially when the data is randomly distributed. From the analysis, we found that the instability growth rate of the interpolation operator plays a dominant role in characterizing the overall error rate, provided that we chose smooth enough kernels. From the numerical results, we found that the proposed method has a better convergence rate than the Variable Bandwidth Diffusion Maps on simple smooth manifolds with randomly sampled grid points.  On the unknown "bunny" manifold, we found that the accuracy of the proposed method is comparable to that of VBDM when the samples point from an un-smoothed surface. In both un-smoothed and smoothed bunnies, we found improved estimates by employing the same Galerkin formulation on a test function space induced by eigenbasis estimated from the VBDM scheme (VBDM-RBF), which suggests the flexibility of the proposed framework that is valuable in applications.

While the proposed method produces encouraging results, there are several open questions. In this paper, we demonstrated the proposed method to solve PDEs on closed manifolds. To handle PDEs imposed by Dirichlet or Neumann boundary conditions, a natural extension of this method could be to use eigenfunctions of homogeneous boundary conditions \cite{peoples2021spectral} in combinations with the harmonic functions of nontrivial boundary conditions as in \cite{harlim2022graph}, which is subjected to future research. Moreover, it would also be interesting to extend the proposed method to vector-valued nonlinear PDEs, employing operator estimation methods in \cite{harlim2022rbf}. The theoretical analysis in this paper relies on the interpolation error in Lemma~\ref{RKHS L2 Convergence} and a weakly unstable interpolator assumption (see Assumption~\ref{weaklyunstable}), where we allow the interpolation operator to be unstable with a polynomial rate in $N$ that grows slower than the Monte-Carlo error rate.
In practice, this assumption is not even valid for the choice of kernel used in our numerical experiments (as we reported in Proposition~\ref{prop:unstable}). Therefore, a further theoretical investigation that takes into account the pseudo-inversion operation that is numerically implemented to avoid unstable interpolators is needed to close the gap between theory and numerics. This issue is rather complicated since the RBF approximation with the pseudo-inverse solution is not an interpolation, so regression-type error bounds need to be considered, and we leave this to future investigation.

\section*{Acknowledgment} The research of JH was partially supported under the NSF grants DMS-1854299, DMS-2207328, DMS-2229435, and the ONR grant N00014-22-1-2193. The research of SJ was supported by the NSFC Grant No. 12101408. The authors also thank Senwei Liang for providing some sample codes for using Meshlab.

\appendix

\section{Proof of Proposition~\ref{prop:unstable}.}\label{AppA}

The main proof is an application of Theorem12.3 in  \cite{Wendland2005Scat} that characterizes a lower bound for the smallest eigenvalue in terms of the separation distance defined as follows.

\begin{defn} Let $\{\mathbf{x}_1,\ldots, \mathbf{x}_N\}\in X \subset M \subset \mathbb{R}^n$, we define the separation distance as,
\[
q_{X,\mathbb{R}^n} := \frac{1}{2} \min_{i\neq j} \| \mathbf{x}_i - \mathbf{x}_j\|_2,
\]
where the norm is defined as before, Euclidean norm in $\mathbb{R}^n$.
\end{defn}

Since Euclidean distance is a good approximation for the geodesic distance $d_g:M\times M\to\mathbb{R}$ for small distances (see e.g., Appendix B of \cite{coifman2006diffusion} or Appendix A of \cite{Berry2016IDM}), for large enough $N\gg 1$, then $q_{X,\mathbb{R}^n}$ is close to,
\[
q_{X,M} := \frac{1}{2} \min_{i\neq j} d_g(\mathbf{x}_i, \mathbf{x}_j).
\]
To specify how large $N$ needs to be, we need an upper bound of the separation distance in term of $N$. We also need to find a lower bound for separation distance to be bounded from below. For uniformly i.i.d. samples, we can deduce the following concentration bound.

\begin{lem}
Let $\{\mathbf{x}_1,\ldots,\mathbf{x}_N\} = X\subset M \subset \mathbb{R}^n$ be a set of i.i.d. samples from $d-$dimensional Riemannian manifold $M$ with positive Ricci scalar curvature at all $\mathbf{x}_i$. Then, with probability higher than $1-\frac{2}{N}$, there exists constants $c(d), C(d)>0$ such that,
\BEA
c(d) N^{-2/d} \leq q_{X,M} \leq C(d) \left(\frac{\log N}{N^2}\right)^{\frac{1}{d}}.\label{q_bound}
\EEA
\end{lem}

Before we prove this, we note that if we denote the upper bound in \eqref{q_bound} as $\epsilon$, then we can specify the size of $N$ which allows $q_{X,M} = q_{X,\mathbb{R}^d}+O(\epsilon^3)$ as $\epsilon\to 0$ or $N\to \infty$.

\begin{proof}
Let $\delta>0$ and suppose that $q_{X,M}(\mathbf{x}_i,\mathbf{x}_j)< 2\delta$. This means for some $i$, there exists $j\neq i$ such that $d_g(\mathbf{x}_i,\mathbf{x}_j)<2\delta$. If the geodesic ball of radius $\delta$ centered at $\mathbf{x}_i$ is $\textup{Vol}(B_\delta(\mathbf{x}_i))$, then the probability such that $\mathbf{x}_j \in \textup{Vol}(B_\delta(\mathbf{x}_i))$, i.e., $d_g(\mathbf{x}_i,\mathbf{x}_j)<2\delta$, is $\textup{Vol}(B_\delta(\mathbf{x}_i))/\textup{Vol}(M)$. Based on Corollary~3.2 in \cite{gray1974volume}, if $M$ has positive Ricci scalar curvature at $\mathbf{x}_i$, then $\textup{Vol}(B_\delta(\mathbf{x}_i))< C(d)\delta^d$, for some constant $C(d)>0$. This means,
\[
\mathbb{P}[q_{X,M}<\delta] \leq \sum_{i=1}^N \mathbb{P}[d_g(\mathbf{x}_i,\mathbf{x}_j)<2\delta \, \textup{for some } \mathbf{x}_j\in X, \,\mathbf{x}_j \neq \mathbf{x}_i] \leq \frac{NC(d)\delta^d}{\textup{Vol}(M)}.
\]
Choosing $\frac{N C(d) \delta^d}{\textup{Vol}(M)} = \frac{1}{N}$, then with probability higher than $1-\frac{1}{N}$,
\BEA
q_{X,M}\geq \left(\frac{\textup{Vol}(M)}{C(d)}\right)^{1/d} N^{-\frac{2}{d}}:=c(d)N^{-\frac{2}{d}}.\label{qlowerbdd}
\EEA
Next, we deduce an upper bound for $q_{X,M}$. For each $\mathbf{x}_i$, the probability of $d_g(\mathbf{x}_i,\mathbf{x}_j) >2\delta_0$ is given by $1 - \textup{Vol}(B_{\delta_0}(\mathbf{x}_i))/\textup{Vol}(M) \leq 1 - C\delta_0^d$, using the lower bound  (see Proposition 14 in \cite{croke1980some}) $\textup{Vol}(B_{\delta_0}(\mathbf{x}_i)) \geq C \delta_0^d$, for all $\delta_0\leq |\text{inj}(M)|/2$ where $\text{inj}(M)$ is the injectivity radius. This means,
\[
\mathbb{P}(\min_{j\neq i} d_g(\mathbf{x}_i,\mathbf{x}_j) >2\delta_0, \textup{for a fixed }i) \leq (1-C\delta_0^d)^N \leq \exp(-CN\delta_0^d),
\]
which also implies that,
\[
\mathbb{P}(q_{X,M} >\delta_0)\leq \Pi_{i=1}^N\mathbb{P}(\min_{j\neq i} d_g(\mathbf{x}_i,\mathbf{x}_j) >2\delta_0, \textup{for fixed }\mathbf{x}_i) \leq \exp(-CN^2\delta_0^d).
\]
Setting $ \exp(-CN^2\delta_0^d) = \frac{1}{N}$, then with probability higher than $1-\frac{1}{N}$,
\[
q_{X,M} \leq C\left(\frac{\log N}{N^2}\right)^{\frac{1}{d}}.
\]
Together with \eqref{qlowerbdd}, the proof is completed.
\end{proof}

\begin{proof}[Proof of Proposition~\ref{prop:unstable}]
When $\Phi$ is an inverse quadratic kernel, the lower bound to $\hat{\lambda}_{\textup{min}}(\boldsymbol{\Phi})$ in terms of the separation distance has been reported in Table~12.1 in  \cite{Wendland2005Scat},
\BEA
\hat{\lambda}_{\textup{min}}(\boldsymbol{\Phi})&\geq& C\exp(-Cq_{X,\mathbb{R}^n}^{-1})  \notag\\
&\geq& C \exp(-Cq_{X,M}^{-1}) \notag\\
&\geq&C\exp(-\hat{C}(d)N^{2/d})  \notag
\EEA
where we have used the fact that $q_{X,\mathbb{R}^n}\approx q_{X,M}$ for large enough $N$ and we have also employed the bounds in \eqref{q_bound}.
\end{proof}

{\color{black}\section{A short review of the RBF-FD approximation.}\label{AppB}
In this section, we provide a short review of the RBF-generated Finite Difference (FD) method to approximate Laplace-Beltrami operator $\Delta_M$ which was introduced in \cite{shankar2015radial}.

Let $M$ be a $d$-dimensional smooth manifold embedded in $\mathbb{R}^n$. Given a set of (distinct) nodes $X=\{\textbf{x}_i\}_{i=1}^N\subset M$ and function values $\mathbf{f}:=(f(\textbf{x}_1),\ldots, f(\textbf{x}_N))^\top$ where $f:M\rightarrow\mathbb{R}$ is an arbitrary smooth function. For a point $\underline{\textbf{x}}\in X$, we denote its $K$ nearest neighbors in $X$ by $S_{\underline{\textbf{x}}}=\{\underline{\textbf{x}_k}\}_{k=1}^K\subset X$. The set $S_{\underline{\textbf{x}}}$ is called  a "stencil" on the manifold corresponding to $\underline{\textbf{x}}$. By definition, we have $\underline{\textbf{x}}=\underline{\textbf{x}_1}$. Denote $\textbf{f}_{\underline{\textbf{x}}}=(f(\underline{\textbf{x}_1}),...,f(\underline{\textbf{x}_K}))^\top$. Our goal is to approximate the Laplacian $\Delta_M$ at the point $\underline{\textbf{x}}$ by a linear combination of the function values $\{f(\underline{\textbf{x}_k})\}_{k=1}^K$, i.e., find the weights $\{w_k\}_{k=1}^K$ such that
\BEA
\Delta_Mf(\underline{\textbf{x}})\approx \sum_{k=1}^Kw_kf(\underline{\textbf{x}_k}).
\label{eqn:RBF-FD-weights}
\EEA
Arranging the weights at each point into each row of a sparse $N$ by $N$ matrix $L_X$, we can approximate the operator over all points by $L_X\mathbf{f}$.

To find the weights $\{w_k\}_{k=1}^K$, we apply the RBF interpolation locally over $S_{\underline{\textbf{x}}}$, i.e., for $\textbf{x}$ in a neighborhood of $\underline{\textbf{x}}$,
$$
I_\phi\textbf{f}(\textbf{x})=\sum_{k=1}^Kc_k\phi(\Vert\textbf{x}-\underline{\textbf{x}_k}\Vert).
$$
Following the steps from (\ref{G_ell}) to (\ref{NRBF_Laplacian}) , one can obtain a $K$ by $K$ matrix $\Delta_{S_{\underline{\textbf{x}}}}^{\mathrm{RBF}}$ to approximate $\Delta_M$ over the stencil $S_{\underline{\textbf{x}}}$. Because of the ordering of nodes in $S_{\underline{\textbf{x}}}$, the weights $\{w_k\}_{k=1}^K$ in (\ref{eqn:RBF-FD-weights}) are given by the entries in the first row of the matrix $\Delta_{S_{\underline{\textbf{x}}}}^{\mathrm{RBF}}$.

In the Figure \ref{fig-ellipse}(a), we test the RBF-FD approximation using the Polyharmonic spline (PHS) kernel $\phi(r)=r^3$ with polynomials of degree up to $1$ on the simple case (ellipse) in Section \ref{sec4.2}. And we choose $K=21$. As mentioned in \cite{flyer2016role}, the RBF interpolation for conditionally positive definite radial basis function like PHS are often modified to include polynomial terms. In this example, we consider the RBF interpolation including linear terms in 2-D which could be written as,
$$
I_\phi\textbf{f}(\textbf{x})=\sum_{k=1}^Kc_k\phi(\Vert\textbf{x}-\underline{\textbf{x}_k}\Vert)+\gamma_1+(\gamma_2 x^1+\gamma_3 x^2),\quad \textbf{x}=(x^1,x^2)\in\mathbb{R}^2
$$
together with the constraints
$$
\sum_{k=1}^Kc_k=\sum_{k=1}^Kc_kx^1_k=\sum_{k=1}^Kc_kx^2_k=0.
$$

In the Figure \ref{ellipse_matern}, we show an additional result from the RBF-FD approach using the following Mat\'ern kernel,
$$
\phi_{\frac{n+3}{2}}(r)=(sr)^{3/2}\frac{e^{-sr}}{\sqrt{sr}}(1+\frac{1}{sr})=(1+sr)e^{-sr}.
$$
In the numerical result, the shape parameter is chosen as $s=2.5$.  To obtain the convergence, we have to choose $K=2\sqrt{N}$. To invert the interpolation matrix $\mathbf{\Phi}$, we use the ridge-regression solution $(\mathbf{\Phi}+\sigma\textbf{I})^{-1}$ with $\sigma=10^{-6}/N$. Compared to Figure~\ref{fig-ellipse}(a), we notice the significant improvement of Direct-RBF-FD when the kernel is changed from PHS to the Mat\'ern function, which indicates the result is sensitive to the choice of kernel, and possibly, to the distribution of the point cloud data. On the other hand, the error corresponding to Spectral-RBF-FD remains the same when the kernel is changed from PHS to the Mat\'ern function. That is, Spectral-RBF-FD is still about five time less accurate compared to Spectral-RBF that used a global RBF interpolation function.

 \begin{figure*}[tbp]%
 \centering
\includegraphics[width=0.6\linewidth]{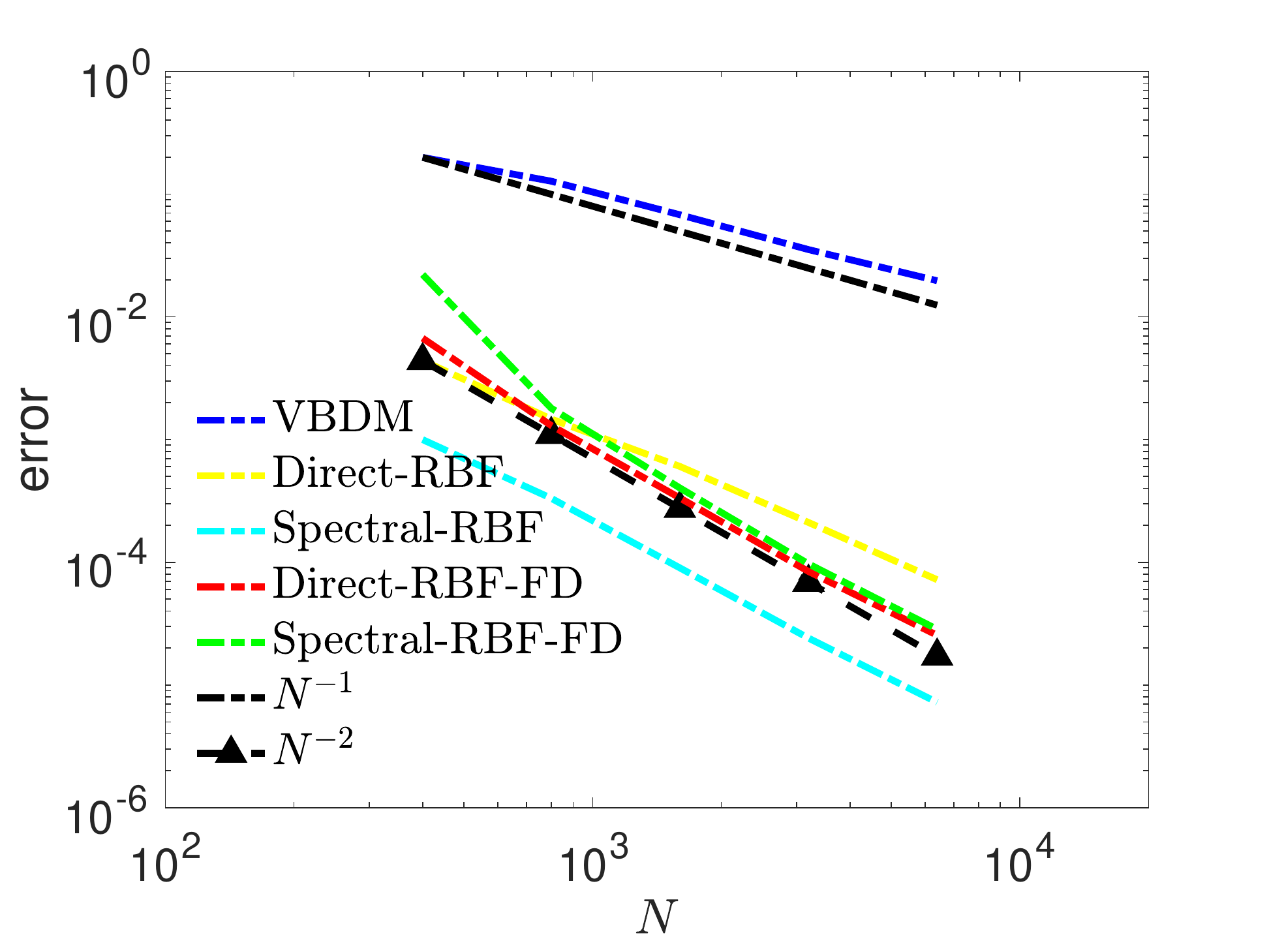}
\caption{Ellipse example: All curves except Direct-RBF-FD and Spectral-RBF-FD in this figure are exactly identical to those in Figure~\ref{fig-ellipse}(a). For these two schemes that involve local RBF interpolation, the errors shown in this figure are based on the Mat\'ern kernel whereas those in Figure~\ref{fig-ellipse}(a) are based on the PHS kernel.}
\label{ellipse_matern}
 \end{figure*}

}


\end{document}